\documentclass[12pt, a4paper]{amsart}
\usepackage{amssymb,fullpage}
\usepackage{amsmath,esint,  bm}
\usepackage[colorlinks, allcolors=RedViolet]{hyperref}
\usepackage{colortbl}
\usepackage[dvipsnames]{xcolor}

\newtheorem{theorem}{Theorem}[section]
\newtheorem{lemma}[theorem]{Lemma}
\newtheorem{proposition}[theorem]{Proposition}
\newtheorem{assumption}[theorem]{Assumption}
\newtheorem{corollary}[theorem]{Corollary}
\theoremstyle{definition}

\theoremstyle{remark}
\newtheorem{remark}[theorem]{Remark}
\numberwithin{equation}{section}

\newcommand{ \R }{ \mathbb{R} }

\newcommand{ \bA }{ \mathbf{A} }
\newcommand{ \bV }{ \mathbf{V} }

\newcommand{ \bphi }{ \bar{\varphi} }
\newcommand{ \bfw }{ \mathbf{w} }
\newcommand{ \bfu }{ \mathbf{u} }
\newcommand{ \bfzeta }{ \mathbf{\zeta} }
\newcommand{ \bfv }{ \mathbf{v} }
\newcommand{\tphi}{{\tilde\phi}}
\renewcommand{\d}{\mathrm{d}}

\newcommand{\loc}{{\operatorname{loc}}}
\renewcommand{\epsilon}{\varepsilon}
\renewcommand{\phi}{\varphi}
\renewcommand{\le}{\leqslant}
\renewcommand{\ge}{\geqslant}
\renewcommand{\leq}{\leqslant}
\renewcommand{\geq}{\geqslant}
\renewcommand{\div}{\operatorname{div}}

\newcommand{\ainc}[1]{\hyperref[ainc]{{\normalfont(aInc){\ensuremath{_{#1}}}}}}
\newcommand{\adec}[1]{\hyperref[adec]{{\normalfont(aDec){\ensuremath{_{#1}}}}}}
\newcommand{\inc}[1]{\hyperref[inc]{{\normalfont(Inc){\ensuremath{_{#1}}}}}}
\newcommand{\dec}[1]{\hyperref[dec]{{\normalfont(Dec){\ensuremath{_{#1}}}}}}

\newcommand{\wMAe}[1]{\hyperref[wMAe]{{\normalfont(wMA){\ensuremath{_{#1}}}}}}

\begin{document}

\title{
Regularity theory for parabolic systems with Uhlenbeck structure 
}



%
\author{Jihoon Ok}
\address[Jihoon Ok]{Department of Mathematics, Sogang University, 35 Baekbeom-ro, Mapo-gu, Seoul 04107, Republic of Korea.}
\email{jihoonok@sogang.ac.kr}

\author{Giovanni Scilla}
\address[Giovanni Scilla]{Department of Mathematics and Applications ``R. Caccioppoli'', University of Naples Federico II, Via Cintia, Monte S. Angelo, 80126 Naples, Italy}
\email{giovanni.scilla@unina.it}

\author{Bianca Stroffolini}
\address[Bianca Stroffolini]{Department of Mathematics and Applications ``R. Caccioppoli'', University of Naples Federico II, Via Cintia, Monte S. Angelo, 80126 Naples, Italy}
\email{bstroffo@unina.it}

\thanks{}

\subjclass[2020]{Primary: 35K40, 46E35; Secondary: 49N60.}



\keywords{full regularity, $\varphi$-Laplace parabolic system, Uhlenbeck-type parabolic systems, general growth}

\begin{abstract}
We establish local regularity theory for parabolic systems of Uhlenbeck type with $\phi$-growth. In particular, we prove local boundedness of weak solutions and their gradient, and then  local  H\"older continuity of the  gradients, providing suitable assumptions on the growth function $\phi$. 
Our approach, being independent of the degeneracy of the system, allows for a unified treatment of both the degenerate and the singular case.

\ \\

\noindent  \textsc{R\'esum\'e.}
On \'etablit la régularit\'e locale pour les solutions de syst\`emes 
paraboliques de type Uhlenbeck avec croissance générale $\phi$. En particulier, on prouve que les solutions telles que leur gradients sont born\'ees localement. De plus, on montre la continuit\'e h\'olderienne des gradients sous des conditions convenables pour la fonction $\phi$.  La m\'ethode est ind\'ependant de la nature du syst\`eme. Comme cons\'equence, on obtient, de mani\`ere unifi\'ee, la r\'egularit\'e dans le cas d\'eg\'en\'er\'ee comme singuli\`ere. 
\end{abstract}

\maketitle


\section{Introduction}

We study local regularity theory for the following parabolic $\varphi$-Laplace system
\begin{equation}
\bfu_t - \div \left(\frac{\phi'(|D\bfu|)}{|D\bfu|}D\bfu\right) =0  \quad \text{in }\ \Omega_T=\Omega\times (0,T)\,,
\label{eq:system}
\end{equation}
where $\Omega\subset \R^n$ ($n \ge 2$) is open, $\phi$ is an Orlicz function verifying suitable growth conditions (see Section 2), $\bfu=(u^1,\dots,u^N)$ is a vector-valued function of $(x,t)\in \Omega\times (0,T)$, $\bfu_t$ is the derivative of $\bfu$ for time variable $t$, and $D\bfu=D_x\bfu$ is the gradient of $\bfu$ for the spatial variable $x$.  In particular, we prove the local boundedness of $\bfu$ and $D\bfu$ and the local H\"older continuity of $D\bfu$.

 A special case of $\phi$ in \eqref{eq:system} is the $p$-power function, i.e., $\phi(t)=\frac{1}{p}t^p$ with $1<p<\infty$. In this case, we have the elliptic and parabolic $p$-Laplace systems
$$
 \div \left(|D\bfu|^{p-2}D\bfu\right) =0  \quad \text{in }\ \Omega 
 \qquad \text{and} \qquad
  \bfu_t - \div \left(|D\bfu|^{p-2}D\bfu\right) =0  \quad \text{in }\ \Omega_T\,.
$$
For the elliptic $p$-Laplace system, Uhlenbeck \cite{Uhl77} proved the local H\"older continuity of $D\bfu$ when $p>2$. In \cite{Uhl77}, Uhlenbeck considered the system   
\begin{equation}\label{eq:Uhlenbeck}
 \div \left(\varrho(|D\bfu|^2)D\bfu\right) =0
\end{equation}
and assumed that $\varrho$ satisfies a $p$-growth condition. Note that by setting $\phi(s):=\int_0^s \tau \varrho(\tau^2)\,\d \tau$ (i.e., $\varrho(s^2)=\phi'(s)/s$) the previous system is changed to 
\begin{equation}\label{eq:ellipticsystem}
\div \left(\frac{\phi'(|D\bfu|)}{|D\bfu|}D\bfu\right) =0,
\end{equation}
which is the elliptic counterpart of \eqref{eq:system} and the Euler-Lagrange system corresponding to the following  autonomous and isotropic energy functional
$$
\int_\Omega \phi(|D\bfu|)\, \d x\,.
$$
From this fact, we sometimes say that the system \eqref{eq:ellipticsystem} or \eqref{eq:system}  has the \textit{Uhlenbeck structure}.  It is worth to point out that the radial structure, meaning the dependence through the modulus of the gradient, is the only one that prevent the formation of singularities (even boundedness of minimizers) and allows to prove  everywhere regularity results in the vectorial case, {see counterexamples in \cite{necas,sverakyan} and also \cite[Section 3]{Min_darkside}.}

 Examples of $\phi(t)$ satisfying the conditions in the paper are $t^p$, $t^p\log (1+t)$, $\max\{t^p,t^q\}$, $\min\{t^p,t^q\}$, and so on.  A more complicate example, which can fit experimental data, can be found in \cite[Section 2.3]{BL}. 
Moreover,  the system \eqref{eq:Uhlenbeck} is strongly concerned with stationary, irrotational flows of  compressible fluids. Precisely, when $N=1$ hence $\bfu = u$,  if  $\varrho$ is the density of an irrotational flow, then  the gradient $Du$ of a solution to \eqref{eq:Uhlenbeck}  represents the velocity of the flow hence the solution $u$ is called the velocity potential of the flow. At this stage, for an ideal flow (e.g. a polytropic flow) the density function $\varrho$ depends on $|Du|^2$. We refer to \cite{Bers_book,FinGil57} for applications of the above system to stationary, irrotational flow of compressible fluids.
 
 After the pioneering work of Uhlenbeck, Tolksdorf \cite{Tol83} obtained $C^{1,\alpha}$-regularity results for more  general elliptic systems with $p$-growth when $1<p<\infty$. We also refer to \cite{GiaMod86,AceFus89,Ham92} for everywhere $C^{1,\alpha}$-regularity results for elliptic systems with $p$-growth.  For the parabolic $p$-Laplace system, DiBenedetto and Friedman \cite{DiBeFried84,DiBeFried85} (see also the monograph \cite{DiB_book}) proved H\"older continuity of $D\bfu$ when $\frac{2n}{n+2} < p< \infty$ and we refer to \cite{Chen86,Wie86,CheDiBe89,Choe91,Choe92,BDLS22} for further related results for parabolic $p$-Laplace systems.

For a general function $\phi$, Lieberman studied regularity theory for elliptic equations (i.e., $N=1$) with $\phi$-growth, and around the same time Marcellini \cite{Mar89,Mar91} had considered elliptic equations with general $(p,q)$-growth.  Full $C^{1,\alpha}$-regularity for the elliptic $\phi$-Laplace system \eqref{eq:ellipticsystem} was established by Marcellini and Papi \cite{MarPapi} and by Diening, Stroffolini and Verde \cite{DieStrVer09}.  Marcellini and Papi proved Lipschitz regularity for local
minimizers of functionals with growth conditions general enough to embrace linear
and exponential ones.  The conclusion then follows using the $C^1$-regularity of the operator, with the help of classical results.
The second result, instead, is reminiscent of the Uhlenbeck proof: a nonlinear quantity $\phi(|D\bfu|)$ is shown to be a subsolution of an elliptic equation.  In addition, the authors were able to  prove an excess decay estimate for $\bV_p(D\bfu)$ (see Section~\ref{Sec2} for the definition of $\bV_p$) which implies the H\"older continuity of $\bV_p(D\bfu)$ and hence of $D\bfu$.

On the other hand, $C^{1,\alpha}$-regularity for the parabolic $\phi$-Laplace system \eqref{eq:system} has remained an open problem. There have been partial developments in this direction. Lieberman \cite{Lie06} proved that if $D\bfu$ is bounded, then $D\bfu$ is H\"older continuous. Hence the local boundedness of $D\bfu$ is missing.  Diening,  Scharle and  Schwarzacher \cite{DieSchSch19} obtained the local boundedness of $D\bfu$ for \eqref{eq:system} under an additional integrability condition on $D\bfu$ which is unnatural in the singular case, that is, $p<2$ in \eqref{characteristic1}.  {Moreover, Isernia \cite{Ise18} obtained the local boundedness of $\bfu$ for \eqref{eq:system}.}

We note that in \cite{Lie06} the approximation of the parabolic system \eqref{eq:system} with nondegenerate systems is omitted and the weak solution is assumed to be twice differentiable with respect to the $x$ variable. In fact, one has to consider approximate nondegenerate parabolic systems (e.g. \eqref{eq:nondegenerate system}), and obtain uniform regularity estimates by differentiating these systems. 
At this stage, the twice differentiability of weak solutions of these systems with respect to the $x$ variable is needed.
However, the proof of the twice differentiability for parabolic system with $\phi$-growth is unclear and not an easy generalization of the one for the parabolic $p$-Laplace system. Even in the elliptic case, a more delicate analysis is required, see \cite[Section 4]{DieEtt08}.  In addition, Baroni and Lindfors \cite{BL} obtained the H\"older and Lipschitz continuity of solutions to Cauchy-Dirichlet problems for parabolic equations ($N=1$) with $\phi$-growth, see also \cite{Lin} for similar results for parabolic obstacle problems with $\phi$-growth. For more regularity results for the parabolic system with $\phi$-growth we refer to \cite{Cho18,DieSchStrVer17,HasOk21,HwaLie15,HwaLie15-1,OhOk22}.
%
%

In this paper, we establish full $C^{1,\alpha}$-regularity for the parabolic $\phi$-Laplace system \eqref{eq:system} by filling all the gaps in previous results. Let us state the main result.

\subsection{Setting of the problem and main result}

Suppose the function $\phi:[0,\infty)\to[0,\infty)$ is an N-function satisfying Assumption~\ref{Ass1}.
A function $\bfu=(u^1,u^2,\dots,u^N) \in C_{\loc}(0,T; L^2_{\loc}(\Omega,\R^N)) \cap L^\phi_\loc(0,T;W^{1,\phi}_{\loc}(\Omega,\R^N))$ is said to be a (local) \emph{weak solution} to \eqref{eq:system} if it satisfies the following weak form of \eqref{eq:system}: 
$$
-\int_{\Omega_T} \bfu \cdot \bfzeta_t \,\d z + \int_{\Omega_T}  \frac{\phi'(|D\bfu|)}{|D\bfu|}D\bfu : D\bfzeta \, \d z =0 
\quad  \text{for all }\ \bfzeta\in C^\infty_{\mathrm c}(\Omega_T,\R^N)\,,
$$
where  ``$\cdot$" and ``$:$"  are the Euclidean inner products in $\R^N$ and $\R^{Nn}$, respectively. 
 By the density of smooth functions in Orlicz-Sobolev spaces and a standard approximation argument (see e.g. the proof of Theorem 1 in \cite{Ise18}), one can see that the weak solution $\bfu$ to \eqref{eq:system} also satisfies for every $0<t_1<t_2<T$,
 \begin{equation}
\left.\int_{\Omega'} \bfu \cdot \zeta(x,t)\,dx\right|_{t=t_1}^{t=t_2} + \int_{\Omega'}\int_{t_1}^{t_2} \left[-\bfu \cdot \zeta_t  +  \frac{\phi'(|D\bfu|)}{|D\bfu|}D\bfu : D\zeta\right] \, \d t\, \d x =0 
\label{eq:weakformul1}
\end{equation}
for all $\zeta\in W^{1,2}(t_1,t_2;L^2(\Omega',\R^N))\cap L^\phi(t_1,t_2;W^{1,\phi}_0(\Omega',\R^N))$ and $\Omega'\Subset\Omega$. We note that weak solution $\bfu$ is not  weakly differentiable with respect to  $t$. Therefore, we cannot take a test function $\zeta$ involving the weak solution  directly.  This technical obstacle can be overcome by using  approximation via Steklov average, see \cite[I. 3-(i) and II. Proposition 3.1]{DiB_book}, which is by now a standard approximation argument. Hence we will {assume that $\mathbf u$ is differentiable, and} consider test functions involving the weak solution without specific comment.

We state $C^{1,\alpha}$-regularity, which is the maximal regularity, for the weak solution $\bfu$ of the system \eqref{eq:system}. {This result  follows  directly by combining Corollary \ref{cor:bddgrad} and Theorem \ref{thm:holdergrad_Lie}.}

\begin{theorem}\label{mainthm}
Suppose $\phi\in C^1([0,\infty))\cap C^2((0,\infty))$ satisfies Assumption~\ref{Ass3} with 
\begin{equation}
p>\frac{2n}{n+2}\,,
\label{eq:hypp}
\end{equation}
and let $\bfu$ be a weak solution to the parabolic system \eqref{eq:system}. Then $D\bfu$ is locally H\"older continuous. Moreover, there exist $\alpha\in(0,1)$ and $c>0$ depending on  $n,N,p,q,\gamma_1,c$ such that for every $Q_{2R}(z_0)\Subset\Omega_T$ and every $0<r\le R$,
$$
 \underset{Q_r(z_0)}{\mathrm{osc}} \, D\bfu \le c \lambda \left(\max\left\{\phi''(\lambda)^{\frac{1}{2}},\phi''(\lambda)^{-\frac{1}{2}}\right\} \frac{r}{R}\right)^{\alpha}
$$
where 
$$
\lambda := \left(\fint_{Q_{2R}(z_0)} \phi(|D\bfu|) \, \d z +1\right)^{\frac{2}{(n+2)p-2n}}\,.
$$
\end{theorem}

We remark that the condition \eqref{eq:hypp} is essential in the regularity theory even for the parabolic $p$-Laplace system, without any additional integrability condition on the solution $\bfu$, see \cite{DiBeFried85,Choe91} and also \cite{DiB_book}.

We shall introduce the strategy of our paper.  We  prove sequentially local $L^\infty$-regularity of 
the weak solution $\bfu$ to \eqref{eq:system}, the local $L^\infty$-regularity and $C^\alpha$-regularity of $D\bfu$, by providing essentially sharp conditions on $\phi$.      
As for the local boundedness of $\bfu$ (Theorem~\ref{thm:bound}), we apply the  Moser iteration to  a suitable test function.  Next,  using the parabolic embedding result in Lemma~\ref{Lem:paraembedding}, we reach the conclusion.
Once the $L^\infty$-regularity result is achieved, we prove twice differentiability of weak solutions to approximate nondegenerate systems in Lemma~\ref{Lem:seconddifferential} by using the difference quotients  and a Giaquinta-Modica type covering argument. Note that  the boundedness of $\bfu$ plays an important role in the proof of Lemma~\ref{Lem:seconddifferential}  since the constant $p$ in \eqref{characteristic1} can be less than $2$. Then by differentiating the approximate nondegenerate system and applying Moser iteration again, we obtain $L^\infty$ estimate for $D\bfu$  in Theorem~\ref{thm:bddgrad}  and Corollary~\ref{cor:bddgrad}.  Finally, we revisit the results with the proofs in \cite{Lie06}, and  prove Theorem~\ref{mainthm}.

\section{Preliminaries}\label{Sec2}

\subsection{Notation}
We write  ${\bf u}=(u^\alpha)=(u^1,\dots,u^N)\in \R^N$ and  
${\bf Q}=(Q^\alpha_i)\in\R^{N\times n}=\R^{Nn}$ 
where 1$\le i\le n$ and  $1\le \alpha \le N$.
For $z=(x,t)\in\R^n\times\R$, we introduce the parabolic cylinder
\begin{equation}
Q_r(z):=B_r(x)\times (t-r^2,t]\,, 
\label{eq:cylinder}
\end{equation}
where $B_r(x)$ denotes the open ball in $\R^n$ with center $x$ and radius $r$. The symbol $\partial_{\rm p} Q_r(z)$ denotes the usual parabolic boundary of $Q_r(z)$.

Let $f : E\to [0,\infty)$ with $E\subset \R$. $f$ is called almost increasing (resp. almost decreasing) if there is $L\ge 1$ such that $f(s)\le L f(t)$ for all $s,t\in E$ with $s\le t$ (resp. $t\le s$). In particular, if we can choose $L=1$, then $f$ is simply called increasing (resp. decreasing).

By $\chi^*$ we denote the Sobolev conjugate exponent of $\chi$; i.e., $\chi^*:=\frac{n\chi}{n-\chi}$ for $\chi<n$, while we agree that $\chi^*:=2\chi$ if $\chi\ge n$.  

{The notation $f\sim g$ means that there exists constant $c\ge 1$ such that $\frac{1}{c}f\le g \le c f$.
We will use the Einstein summation convention, that is, we will omit the summation symbol for indexes that appear twice, see e.g. \eqref{ellipticity} and the next inequality.}
\subsection{Orlicz functions}

In this paper, $\phi:[0,\infty)\to[0,\infty)$ is always an $N$-function, that is, $\phi(0)=0$, there exists a right continuous derivative $\phi'$ of  $\phi$,  $\phi'$ is 
{increasing}
with $\phi'(0)=0$ and $\phi'(t)>0$ when $t>0$. 
For simplicity, we shall assume that
$$
\phi(1)=1\,.$$
{Note that if we do not assume the above condition, then constants $c$ may depend on $\phi(1)$.}
Moreover, we assume that $\phi$ satisfies the following growth conditions: 
\begin{assumption} \label{Ass1}
$\phi:[0,\infty)\to[0,\infty)$ is an $N$-function, and  there are $1<p\le q$ such that $\frac{\phi(t)}{t^p}$ is almost increasing and   $\frac{\phi(t)}{t^q}$ is almost decreasing for $t\in (0,\infty)$ with constant $L\ge 1$. 
\end{assumption}

The almost decreasing and increasing conditions in Assumption~\ref{Ass1} are equivalent to the $\Delta_2$ and $\nabla_2$ conditions for $\phi$, respectively. Compared with the $\Delta_2$ type conditions, the benefit of the almost increasing/decreasing condition is that we can directly see the lower and upper bounds of an exponent factor of $\phi$. In particular, we will prove the boundedness of the weak solution to \eqref{eq:system} under the above assumption where the lower bound $p$ will play a crucial role.  We also remark that Assumption~\ref{Ass1} with $L=1$ is equivalent to the following inequality 
\begin{equation}\label{characteristic}
1< p \le  \frac{t \phi'(t)}{\phi(t)} \le q\,,\quad t>0\,.
\end{equation}
For any $t>0$ and $0<c<1<C$ there holds
\begin{equation}
c^q \varphi(t) \leq \varphi(ct) \leq c^p \varphi(t) \,\, \mbox{ and } \,\, C^p \varphi(t) \leq \varphi(Ct) \leq C^q \varphi(t)\,.
\label{eq:estimpq}
\end{equation}

The conjugate function of $\phi$ is defined as
$$\phi^*(t):= \sup_{s \geq 0}\, (st - \phi(s))\,.$$ 
From the definition, the following Young's inequality  
\begin{equation}  \label{eq:young}
  st \leq \phi(t) +  \phi^\ast(s)\,,\quad s,t\geq0\,,
\end{equation}
holds true. Since the exact value of $\phi^*$ is not always explicitly computable, the estimate
\begin{equation}
\phi^*\left(\frac{\phi(t)}{t}\right)\sim \phi^*(\phi'(t)) \sim \phi(t)
\label{eq:hok2.4}
\end{equation}
will often be useful in computations (see \cite[Theorem~2.4.10]{HH}).
In fact, the above relation 
holds true since Assumption~\ref{Ass1} guarantees the $\Delta_2$ condition for both $\phi$ and $\phi^*,$ and relevant constants depends on $p$, $q$ and $L$.

For higher order regularity results  we will consider stronger assumptions. 

\begin{assumption} \label{Ass2}
$\phi:[0,\infty)\to[0,\infty)$ is an $N$-function and satisfies
\begin{itemize}
\item[(1)] $\phi \in C^1([0,\infty))\cap C^2((0,\infty))$ 
\item[(2)] There exist $1<p\le q$ such that 
\begin{equation}\label{characteristic1}
0< p -1  \le  \frac{t \phi''(t)}{\phi'(t)} \le q -1 \,, \quad t>0\,.
\end{equation}
\end{itemize}
\end{assumption}

Note that Assumption~\ref{Ass2} implies Assumption~\ref{Ass1} with the same $p$ and $ q$ and with the constant $L=1$, hence we have \eqref{characteristic}.  

We notice that in the above two assumptions we can replace $p$ and $q$ by $\min\{p,2-\epsilon\}$ and $\max\{q,2+\epsilon\}$ for $\epsilon>0$, respectively. Therefore, without loss of generality,  we always assume that $p$ and $q$ satisfy
\begin{equation}
 1< p < 2 <q\, .
 \label{p2q}
\end{equation}

The next assumption is adding an H\"older type continuity on the Hessian of $\phi(|{\bf Q}|)$ for ${\bf Q}\in \R^{Nn}$, denoted by $D^2_{\bf Q}\phi(|{\bf Q}|)$.

\begin{assumption} \label{Ass3}
$\phi:[0,\infty)\to[0,\infty)$ is an $N$-function and satisfies Assumption \eqref{Ass2}. Furthermore, 
 there exist positive constants $\gamma_1$ and $c_h$ such that for every ${\bf Q},{\bf P}\in \R^{Nn}$ with $|{\bf Q}-{\bf P}| \le \frac{1}{2}|{\bf Q}|$,
\begin{equation}\label{ass3_holder}
|D^2_{\bf Q}\phi(|{\bf Q}|)-D^2_{\bf Q} \phi(|{\bf P}|)| \le c_h \left(\frac{|{\bf Q}-{\bf P}|}{|{\bf Q}|}\right)^{\gamma_1} \phi''(|{\bf Q}|).
\end{equation}
\end{assumption}

Note that if $\phi(t)=t^p$ with $1<p<\infty$, then it satisfies Assumption~\ref{Ass3}. Similar assumptions were used  for proving the $C^{1,\alpha}$-regularity for minimizers of functionals with  general growth in \cite{DieStrVer09}.

If $\phi$ satisfies Assumption~\ref{Ass1}, we define the Orlicz space $L^\phi(\Omega,\R^N)$ as the set of all measurable functions $f:\Omega\to\R^N$ such that
$$
\int_\Omega \phi(|f(x)|)\, \d x < \infty, 
$$ 
and 
the Orlicz-Sobolev space $W^{1,\phi}(\Omega,\R^N)$ as the set of all  $f\in L^\phi(\Omega,\R^N)\cap W^{1,1}(\Omega,\R^N)$ such that
$$
\int_\Omega \phi(|Df(x)|)\, \d x < \infty.
$$ 
 $L^\phi(\Omega,\R^N)$ and $W^{1,\phi}(\Omega,\R^N)$ are endowed with the usual Luxembourg type norms. Then they are reflexive Banach spaces. Moreover, the parabolic space
$L^\phi(t_1,t_2;W^{1,\phi}(\Omega,\R^N))$ denotes the set of all functions 
$f\in L^1(t_1,t_2;W^{1,1}(\Omega,\R^N))$
such that $f(\cdot,t)\in W^{1,\phi}(\Omega,\R^N)$ for a.e. $t\in(0,T)$ and 
$$
\int_{t_1}^{t_2} \int_\Omega \phi(|Df(x,t)|)\, \d x\, \d t < \infty.
$$

\subsection{Shifted $N$-functions and related operators}

The following definitions and results about shifted $N$-functions can be found in \cite{DieEtt08, DieStrVer09}.

For an $N$-function  $\phi$ and for $a\ge0 $,  we define the shifted $N$-function $\phi_a$ by
$$
\phi_{a}(t):=\int_{0}^{t} \frac{\phi'(a+s) s}{a+s} \, \d s  \quad \left(\text{i.e., }\ \phi_a'(t) =\frac{\phi'(a+t)}{a+t}t\right).
$$
We note that if $\phi$ satisfies Assumption~\ref{Ass1} or \ref{Ass2} or \ref{Ass3}, then $\phi_a$ also satisfies Assumption~\ref{Ass1} or \ref{Ass2} or \ref{Ass3} uniformly in $a\ge 0$ with the same $p$ and $q$. 
We then recall useful inequalities for shifted $N$-function $\phi_a$ in \cite{DieEtt08} and \cite{DieStrVer09}.
\begin{lemma}\cite{DieStrVer09}
Let $\varphi$
satisfy Assumption~\ref{Ass2}. Then we have
\begin{align}
&\phi_a(t) \sim \phi'_a(t)\,t\,; \label{(2.6a)} \\
&\phi_a(t) \sim \phi''(a+t)t^2\sim\frac{\varphi(a+t)}{(a+t)^2}t^2\sim \frac{\varphi'(a+t)}{a+t}t^2\,,\label{(2.6b)}\\
& \phi(a+t)\sim [\phi_a(t)+\phi(a)]\,,\label{eq:approx}
\end{align}
which hold uniformly with respect to $a\geq0$.
\end{lemma}

\begin{lemma}\cite[Lemma~32]{DieEtt08}
Let $\varphi$
satisfy Assumption~\ref{Ass1}. 
Then for all $\delta>0$ there exists $c_\delta>0$ depending only on 
$p$, $q$, $L$ and $\delta$,  
such that for all $t,u,a\geq0$
\begin{equation}\label{eq:young4}
\left\{\begin{aligned}
& tu\leq \delta \varphi(t) + c_\delta \varphi^*(u)\,,\\
& t\varphi'(u) + u\varphi'(t) \leq \delta \varphi(t) + c_\delta\varphi(u)\,, \\
& tu\leq \delta \varphi_a(t) + c_\delta \varphi_a^*(u)\,,\\
& t\varphi_a'(u) + u\varphi_a'(t) \leq \delta \varphi_a(t) + c_\delta\varphi_a(u)\,. 
\end{aligned}\right.
\end{equation}
\end{lemma}

\begin{lemma} \cite[Lemmas 24 and 29]{DieEtt08}.
Let $\varphi$ satisfy Assumption~\ref{Ass2}.
\begin{itemize} 
\item[(1)] Uniformly in $s,t\in\R^n$ with $|s|+|t|>0$
\begin{equation}
\varphi''(|s|+|t|)|s-t| \sim \varphi'_{|s|}(|s-t|)\,,
\label{eq:DE(6.7)}
\end{equation}
\item[(2)] There exists $c=c(p,q)>0$ such that for all $s_1,s_2,t\in\R^n$
\begin{equation}
\phi'_{|s_2|}(|s_1-s_2|)\lesssim  \phi'_{|t|}(|s_1-t|)+\phi'_{|t|}(|s_2-t|)\,,
\label{eq:DE(6.19)}
\end{equation}
\end{itemize}
where the hidden constants above depend only on $p$ and $q$.
\end{lemma}

The following lemma (see \cite[Corollary~26]{DieKre08}) deals with the \emph{change of shift} for $N$-functions.
\begin{lemma}[change of shift]\label{lem:changeshift}
Let $\varphi$ be an $N$-function with $\Delta_2(\varphi),\Delta_2(\varphi^*)<\infty$. Then for any $\eta>0$ there exists $c_\eta>0$, depending only on $\eta$ and $\Delta_2(\varphi)$, such that for all ${a}, {b}\in\R^n$ and $t\geq0$
\begin{equation}
\varphi_{|{a}|}(t) \leq c_\eta \varphi_{|{b}|}(t) + \eta \varphi_{|{a}|}(|{a}-{b}|)\,.
\label{(5.4diekreu)}
\end{equation}
\end{lemma}

We next define vector valued functions $\bA, \bV:\R^{Nn}\to \R^{Nn}$ by
$$
\bA({\bf Q}):= \frac{\phi'(|{\bf Q}|)}{|{\bf Q}|} {\bf Q} = D_{\bf Q} [\phi(|{\bf Q}|)]
\quad\text{and}\quad
\bV({\bf Q}):= \sqrt{\frac{\phi'(|{\bf Q}|)}{|{\bf Q}|}} {\bf Q}.
$$
In particular, for $1<p<\infty$, we denote by  $\bV_p({\bf Q})$ the function $\bV$ associated to $\phi(t)=\frac{1}{p}t^p$; i.e., $\bV_p({\bf Q}) := |{\bf Q}|^{\frac{p-2}{2}}{\bf Q}$. With shifted $N$-function $\phi_a$, we define accordingly
\begin{equation}
\bA^a({\bf Q}):= \frac{\phi'_a(|{\bf Q}|)}{|{\bf Q}|} {\bf Q},\ \ 
\bV^a({\bf Q}):= \sqrt{\frac{\phi_a'(|{\bf Q}|)}{|{\bf Q}|}} {\bf Q}
\ \ \text{and}\ \  
\bV^a_p({\bf Q}):=\left(a+ |{\bf Q}|\right)^{\frac{p-2}{2}} {\bf Q}\,.
\label{eq:shiftedop}
\end{equation}

We further suppose that  $\phi$ satisfies Assumption~\ref{Ass2}. Denote
\begin{equation}\label{defA}
A_{ij}^{\alpha\beta}({\bf Q}):=\frac{\partial\bA({\bf Q})_{j}^{\beta}}{\partial Q_{i}^{\alpha}} = \frac{\phi'(|{\bf Q}|)}{|{\bf Q}|}\left\{\delta_{ij}\delta^{\alpha\beta} + \left(\frac{\phi''(|{\bf Q}|)|{\bf Q}|}{\phi'(|{\bf Q}|)}-1 \right) \frac{Q_i^\alpha Q^{\beta}_{j}}{|{\bf Q}|^2}\right\},
\end{equation}
where $1\le i,j\le n$ and $1\le \alpha,\beta \le N$. Here, $\delta^{\alpha \beta}$ and $\delta_{i j}$ are the Kronecker symbols. Note that $D^2_{\bf Q}\phi(|{\bf Q}|)=(A_{ij}^{\alpha\beta}({\bf Q}))$.
Then we see that
\begin{equation}\label{ellipticity}
\min \{p-1,1\} \frac{\varphi^{\prime}(|{\bf Q}|)}{|{\bf Q}|}|\bm\omega|^{2} \leqslant A^{\alpha \beta}_{i j}({\bf Q}) \omega^{\alpha}_{i} \omega^{\beta}_{j} \leqslant \max \{q-1,1\} \frac{\varphi^{\prime}(|{\bf Q}|)}{|{\bf Q}|}|\bm\omega|^{2}
\end{equation}
for all ${\bf Q}, \bm\omega \in \mathbb{R}^{n N}$.
Moreover, since
\[
[\bA({\bf P}) - \bA({\bf Q})]^{\beta}_{j} =\int_0^1\frac{\partial }{\partial \tau}   [\bA(\tau {\bf P} + (1-\tau) {\bf Q})]^{\beta}_{j} \, \d \tau =  \int_0^1 A_{ij}^{\alpha\beta}(\tau {\bf P} + (1-\tau) {\bf Q})({\bf P}-{\bf Q})_{i}^{\alpha}\,\d\tau ,
\]
using the above results and  \cite[Lemma 20]{DieEtt08}, we have that
$$
\begin{aligned}
(\bA({\bf P}) - \bA({\bf Q})):({\bf P}-{\bf Q}) &\ge \frac{1}{c}\left(\int_0^1 \frac{\phi'(|\tau {\bf P} + (1-\tau) {\bf Q}|)}{|\tau {\bf P} + (1-\tau) {\bf Q}|} \, \d \tau\right) |{\bf P}-{\bf Q}|^2\\
&\ge \frac{1}{c}\frac{\phi'(|{\bf P}| + |{\bf Q}|)}{|{\bf P}| + |{\bf Q}|}  |{\bf P}-{\bf Q}|^2,
\end{aligned}
$$
and
\begin{equation}\label{PhiPQ}
\begin{aligned}
|\bA({\bf P}) - \bA({\bf Q})| \le c \left(\int_0^1 \frac{\phi'(|\tau {\bf P} + (1-\tau) {\bf Q}|)}{|\tau {\bf P} + (1-\tau) {\bf Q}|} \, \d \tau\right)|{\bf P}-{\bf Q}| \le c \frac{\phi'(|{\bf P}| + |{\bf Q}|)}{|{\bf P}| + |{\bf Q}|}  |{\bf P}-{\bf Q}|.
\end{aligned}
\end{equation}
Moreover, 
we have that 
\begin{equation}\label{monotonicity}
({\bf A}( {\bf P})-{\bf A}({\bf Q})) : ( {\bf P} -  {\bf Q})  \sim \varphi_{|{\bf P}|}(| {\bf P} -  {\bf Q}|) \sim |\bV( {\bf P})-\bV( {\bf Q})|^2
\end{equation}
and
\begin{equation}
|{\bf A}( {\bf P})-{\bf A}({\bf Q})|  \sim \varphi'_{|{\bf P}|}(|{\bf P} - {\bf Q}|)\,.
\label{eq:(3.4)}
\end{equation}
(see \cite[Lemma~3.1]{DieSchSch19}). 
We note that the estimates in above still hold for $\phi_a$ and the related operators $\bA^a$ and $\bV^a$.

From \cite[Lemma~3.3]{DieSchSch19}, it follows that
$$
|{\bf A}^a({\bf Q})-{\bf A}({\bf Q})| \leq \varphi'_{|{\bf Q}|}(a)\,.
$$
Applying the same argument to the $N$-function $\bar{\phi}_{|{\bf Q}|}$ defined by $\bar{\phi}'_{|{\bf Q}|}(t):=\sqrt{\varphi'_{|{\bf Q}|}(t) t}$, we obtain
\begin{equation}
|\bV^a({\bf Q})-\bV({\bf Q})|^2 \leq c |\bar{\phi}'_{|{\bf Q}|}(a)|^2\sim \varphi_{|{\bf Q}|}(a)\,.
\label{eq:(3.22)}
\end{equation}

{Note that all constants concerned with the relation $\sim$ and $c$  in above depend only on $p$ and $q$.}

\subsection{Embedding}

We recall a Gagliardo-Nirenberg type inequality for Orlicz functions in \cite[Lemma~2.13]{HasOk21}.
A function $\varphi:[0,\infty)\to[0,\infty)$ is said to be a \emph{weak $\Phi$-function} if it is increasing with $\varphi(0)=0$, $\lim_{t\to0^+}\varphi(t)=0$, $\lim_{t\to +\infty}\varphi(t)=+\infty$ and such that the map $t\to \frac{\varphi(t)}{t}$ is almost increasing. Note that every $N$-function is a weak $\Phi$-function.

\begin{lemma}
\label{Lem:hok}
Assume that $\psi:[0,\infty)\to[0,\infty)$ is a weak $\Phi$-function and such that $t\mapsto \frac{\psi(t)}{t^{q_1}}$ is almost decreasing with constant $L\ge 1$ for some $q_1\geq 1$. For $p\in [1,n)$ and $q_2>0$ we have
\begin{equation*}
\bigg(\fint_{B_{r}}\psi\big(\big|\tfrac{f}{r}\big|\big)^\gamma\, \d x\bigg)^{\frac1\gamma}
\le c \bigg(\fint_{B_r}\left[\psi(|Df|)^p+\psi\big(\big|\tfrac{f}{r}\big|\big)^p\right]\,\d x\bigg)^{\frac{\theta}{p}}
\psi\bigg(\Big(\fint_{B_r}\big|\tfrac{f}{r}\big|^{q_2}\,\d x\Big)^\frac1{q_2}\bigg)^{1-\theta}
\end{equation*}
for some $c=c(n,L,q_1,q_2)>0$, provided that $\theta\in (0,1)$ and $\gamma$ satisfies
\begin{equation*}
\frac{1}{\gamma}\geq \frac\theta{p^*}+\frac{(1-\theta)q_1}{q_2}\,.
\end{equation*}
\end{lemma}

Applying the above lemma we can obtain a parabolic embedding result for an Orlicz function $\phi$. 

\begin{lemma}\label{Lem:paraembedding}
Let $m>0$. Suppose that  $\phi$ satisfies Assumption~\ref{Ass1} with 
$$
\max\{1,\tfrac{mn}{n+m}\}<p\leq q.
$$ 
There exists $\theta=\theta(n,m,p,q)\in (0,1)$ and $c=c(n,m,p,q,L)>0$ such that for every
$$
f\in   L^\infty(t_1,t_2; L^m(B_r)) \cap L^\phi(t_1,t_2; W^{1,\phi}(B_r)) 
$$
we have 
$$\begin{aligned}
\fint_{B_r \times  [t_1,t_2]} \phi\big(\big|\tfrac{f}{r}\big|\big)^{\frac{n+m}{n}} \,dz
& \le c   \bigg(\fint_{B_r}\left[\phi(|Df|)+\phi\big(\big|\tfrac{f}{r}\big|\big)\right]\,\d x\bigg)^{\frac{\theta(n+m)}{n}}\\
&\qquad \times \phi\bigg(\bigg(\underset{t\in [t_1, t_2]}{\mathrm{ess\, sup\,}} \fint_{B_r}\big|\tfrac{f}{r}\big|^m\,dx\bigg)^{\frac{1}{m}}\bigg)^{\frac{(1-\theta)(n+m)}{n}}.
\end{aligned}$$
\end{lemma}
\begin{proof} Without loss of generality, we can assume that $p<n$. If $p\geq n$, it is enough to consider any $\tilde p\in (\tfrac{mn}{n+m},n)$ instead of $p$. Note that by \eqref{characteristic} the function $\phi^{\frac{1}{p}}$ is a weak $\Phi$-function, and the function $\frac{\phi^{1/p}}{t^{q/p}}$ is decreasing. Therefore, applying Lemma~\ref{Lem:hok} with $\psi=\phi^{\frac{1}{p}}$ and $(\gamma,p,q_1,q_2)=(p\tfrac{n+m}{n},p,\tfrac qp,m)$, we have that for a.e. $t\in [t_1,t_2]$, 
$$\begin{aligned}
\fint_{B_{r}}\phi\big(\big|\tfrac{f(t)}{r}\big|\big)^{\frac{n+_m}{n}}\, dx 
& \le c \bigg(\fint_{B_r}\left[\phi(|Df(t)|)+\phi\big(\big|\tfrac{f(t)}{r}\big|\big)\right]\,dx\bigg)^{\frac{\theta(n+m)}{n}}\\
& \qquad \times \phi\bigg(\bigg(\fint_{B_r}\big|\tfrac{f(t)}{r}\big|^{m}\,dx\bigg)^\frac1{m}\bigg)^{\frac{(1-\theta)(n+m)}{n}},
\end{aligned}$$
where $f(t)=f(x,t)$ and $\theta$ satisfies 
$$
\frac{n}{n+m}=\frac{\theta(n-p)}{n}+\frac{(1-\theta)q}{m}\quad  \Longleftrightarrow\quad \theta=\frac{n(nm-nq-mq)}{(n+m)(nm-nq-mp)}.
$$
Note that $\theta\in(0,1)$ by the assumption $\tfrac{mn}{n+m}<p\leq q$, which yields $\frac{n-p}{n}<\frac{n}{n+m}<\frac{q}{m}$, and that $\frac{\theta(n+m)}{n} \in(0,1]$.
Finally integrating for $t$ in $[t_1,t_2]$ and using H\"older's inequality when $p<q$ (i.e., $\frac{\theta(n+m)}{n}<1$), we obtain the desired estimate.
\end{proof}

\begin{remark}\label{rmk:embedding}
In the above definition, if $m=2$ , we see that 
$$
\phi(|f|)^{1+\frac{2}{n}}\in L^1(B_r\times [t_1,t_2]).  
$$
Note that this implies
$$
f \in L^{\frac{p (n+2)}{n}} (B_r\times [t_1,t_2])\,, 
$$
where $\frac{p(n+2)}{n} >2 $ if $p>\frac{2n}{n+2}$.
\end{remark}

\vspace{0.5cm}

\section{Local boundedness}


We first prove that any weak solution $\bfu$ to \eqref{eq:system} is locally bounded by using the Moser iteration technique (for similar arguments, cfr. \cite[Theorem 2]{Choe92}, and \cite{Ise18}, where the superquadratic case is addressed). 
{The key points in our approach are the introduction of the function $\psi$ in \eqref{psi} and use of the embedding result Lemma~\ref{Lem:paraembedding} for Orlicz functions in the parabolic setting,} 
that measures the superquadratic or subquadratic character of the function $\phi$.

\begin{theorem}\label{thm:bound}
Let $\phi$ satisfy Assumption~\ref{Ass1} with
\eqref{eq:hypp}
and let $\bfu$ be a weak solution to \eqref{eq:system}. 
Then $\bfu\in L^\infty_{\loc}(\Omega_T,\R^N)$. Moreover, there exists $c\ge 1$  depending on $n$, $N$, $p$, $q$ and $L$ such that for any $Q_{2r}\Subset \Omega_T$,
\begin{equation}\label{estimate:bounded}
\sup_{Q_r} \phi(|\tfrac{\bfu}{r}|) \le c \left(\fint_{Q_{2r}} \psi(|\tfrac{\bfu}{r}|) \phi(|\tfrac{\bfu}{r}|)^{\chi_0}\,\d z \right)^{\frac{1}{\chi_0}} +c,
\end{equation}
where
\begin{equation}\label{psi}
\psi(s):=\max\{s^2,\phi(s)\}, \quad s>0,
\end{equation}
and $\chi_0>0$ is determined in \eqref{chi0} below.
\end{theorem}



\begin{proof} 
We divide the proof into three steps.

\textit{Step 1. (Caccioppoli type inequality)} 
Let $0<r<1$, $z_0=(x_0,t_0)\in\R^n\times \R$ and $Q_{2r}(z_0)\Subset\Omega_T$ be defined as in \eqref{eq:cylinder}. Let $\rho_1=s_1r$ and $\rho_2=s_2r$ with $1\le s_1<s_2\le 2$, $\xi \in C^\infty_0(B_{\rho_2}(x_0))$ be such that 
\begin{equation}
0 \le \xi \le 1\,,\quad \xi\equiv 1\,\,\, \mbox{ in } \,\,\, B_{\rho_1}(x_0) \,\,\, \mbox{ and }\,\,\,|D\xi|\le \frac{2}{(s_2-s_1)r}\,, 
\label{eq:cutoff1}
\end{equation}
and let $\eta\in C^\infty(\R)$ be such that 
\begin{equation}
0\le \eta \le 1\,,\quad \eta\equiv 0 \,\,\, \mbox{ in } (-\infty,- {\rho^2_2}]\,,\,\,\, \eta \equiv 1\,\,\, \mbox{ in } \,\,\, [-{\rho^2_1},\infty)\,,\,\,\, 0\le \eta_t \le \frac{2}{(s_2^2-s_1^2)r^2}\,.
\label{eq:cutoff2}
\end{equation}
With fixed $\chi>0$ to be determined later, we take 
\[
\zeta =  \phi(|\tfrac{\bfu}{r}|)^{\chi}\eta^2 \xi^q \bfu
\]
as a test function in \eqref{eq:weakformul1} 
{and integrate by parts.} Then for every $\tau\in (-\rho_1^2,0]$ we have
\begin{equation*}\begin{aligned}
0&=\underbrace{\int_{-\rho^2_2}^\tau \int_{B_{\rho_2}(x_0)} \bfu_t \cdot \bfu\phi(|\tfrac{\bfu}{r}|)^{\chi}\eta^2 \xi^q \, \d x\,\d t}_{=:I_1}
\\ &\qquad+ \underbrace{\int_{-{\rho^2_2}}^\tau \int_{B_{\rho_2}(x_0)}  \frac{\phi'(|D{\bf u}|)}{|D{\bf u}|}D{\bf u} :  D [\bfu\phi(|\tfrac{\bfu}{r}|)^{\chi}\eta^2 \xi^q] \, \d x\,\d t}_{=:I_2} \,.
\end{aligned}\end{equation*}

Now, we estimate both the terms $I_1$ and $I_2$ separately. For what concerns $I_1$, setting 
\begin{equation}\label{Phigamma}
\Phi_\chi(s) := \int_0^{s^2}  \phi(\sqrt{\sigma})^\chi \,\d\sigma \le s^2 \phi(s)^\chi, \quad s\ge 0,
\end{equation}
we obtain
\begin{equation}
\begin{aligned}
I_1 & = r^2\int_{-\rho^2_2}^\tau \int_{B_{\rho_2}(x_0)} \left[\tfrac{1}{2} \partial_t( \Phi_\chi(|\tfrac{\bfu}{r}|)\eta^2) - \Phi_\chi(|\tfrac{\bfu}{r}|)\eta \eta_t \right]\xi^q \, \d x\,\d t \\
& \ge \frac{r^2}{2} \int_{B_{\rho_2}(x_0)}   \Phi_\chi(|\tfrac{\bfu(x,\tau)}{r}|)\eta(\tau)^2 \xi^q \, \d x  -  \frac{2}{s_2^2-s_1^2}\int_{Q_{\rho_2}(z_0)}  \Phi_\chi(|\tfrac{\bfu}{r}|)   \, \d z\,.
\end{aligned}
\label{eq:estimate1}
\end{equation}
Integrating by parts and taking into account \eqref{characteristic},
$$
\Phi_\chi(s)  = s^2 \phi(s)^\chi - \chi \int_0^{s^2} \sigma  \phi(\sqrt{\sigma})^{\chi-1}\phi'(\sqrt\sigma) \frac{1}{2\sqrt \sigma} \, \d \sigma \ge s^2\phi(s)^\chi -\frac{q\chi}{2} \Phi_\chi(s) 
$$
hence
\begin{equation}\begin{aligned}
 \frac{2}{2+q\chi}  s^2 \phi(s)^\chi \le \Phi_\chi(s) \le s^2\phi(s)^\chi . 
\label{eq:equivalence}
\end{aligned}\end{equation}

As for $I_2$, we have
\begin{equation}
\begin{aligned}
I_2 & = \int_{-{\rho^2_2}}^{\tau} \int_{B_{\rho_2}(x_0)}  \frac{\phi'(|D{\bf u}|)}{|D{\bf u}|}\left(|D{\bf u}|^2\phi(|\tfrac{\bfu}{r}|)^{\chi} + \chi r^2 \phi(|
\tfrac{\bfu}{r}|)^{\chi-1}\frac{\phi'(|\tfrac{\bfu}{r}|)}{|\tfrac{\bfu}{r}|} \frac{|D [|\tfrac{\bfu}{r}|^2]|^2}{4}\right)\eta^2 \xi^q\, \d x\,\d t\\
&\quad +  \int_{-\rho_2^2}^{\tau} \int_{B_{\rho_2}(x_0)}  \frac{\phi'(|D{\bf u}|)}{|D{\bf u}|} \left[D{\bf u} : ( D\xi\otimes \bfu)\right] \phi(|\tfrac{\bfu}{r}|)^{\chi}\eta^2 q \xi^{q-1} \, \d x\,\d t \\
& \geq \frac{1}{c_1}\int_{-\rho_2^2}^{\tau} \int_{B_{\rho_2}(x_0)}  \left(\phi(|D{\bf u}|)\phi(|\tfrac{\bfu}{r}|)^{\chi} +  \chi r^2 \frac{\phi'(|D{\bf u}|)}{|D{\bf u}|}  \frac{\phi(|\tfrac{\bfu}{r}|)^{\chi}}{|\tfrac{\bfu}{r}|^2}|D [|\tfrac{\bfu}{r}|^2]|^2\right)\eta^2 \xi^q\, \d x\,\d t\\
&\quad -  \int_{-\rho_2^2}^{\tau} \int_{B_{\rho_2}(x_0)} \left[  \frac{1}{2c_1} \phi^*(\phi'(|D{\bf u}|)\xi^{q-1}) + c \phi(|D\xi| |\bfu|)\right] \phi(|\tfrac{\bfu}{r}|)^{\chi}\eta^2   \, \d x\,\d t \\
& \geq \frac{1}{2c_1}\int_{-\rho_2^2}^{\tau} \int_{B_{\rho_2}(x_0)}  \left(\phi(|D{\bf u}|)\phi(|\tfrac{\bfu}{r}|)^{\chi} + \chi r^2 \frac{\phi'(|D{\bf u}|)}{|D{\bf u}|}\frac{\phi(|\tfrac{\bfu}{r}|)^{\chi}}{|\tfrac{\bfu}{r}|^2}|D [|\tfrac{\bfu}{r}|^2]|^2\right)\eta^2 \xi^q\, \d x\,\d t\\
&\quad - \frac{c}{(s_2-s_1)^q} \int_{Q_{\rho_2}(z_0)}  \phi(|\tfrac{\bfu}{r}|)^{\chi+1} \, \d z\,,
\end{aligned}
\label{eq:estimate2}
\end{equation}
where in the first inequality we  have applied Young's inequality to the second integral, while in the last one we used the inequality $|D\xi|\le \frac{c}{(s_2-s_1)r}$ and \eqref{eq:estimpq}.

Combining the above estimates \eqref{eq:estimate1} and \eqref{eq:estimate2}, we have that for every $\tau\in (-\rho_1^2,0]$
\[\begin{aligned}
&r^2\int_{B_{\rho_2}(x_0)}   \Phi_\chi(|\tfrac{\bfu(x,\tau)}{r}|) \xi^q \, \d x\\
&  + \int_{-{\rho^2_2}}^{\tau} \int_{B_{\rho_2}(x_0)}  \left(\phi(|D{\bf u}|)\phi(|\tfrac{\bfu}{r}|)^{\chi} +\chi r^2 \frac{\phi'(|D{\bf u}|)}{|D{\bf u}|} \frac{\phi(|\tfrac{\bfu}{r}|)^{\chi}}{|\tfrac{\bfu}{r}|^2}|D [|\tfrac{\bfu}{r}|^2]|^2\right)\eta^2 \xi^q\, \d x\,\d t\\
&\le  \frac{c}{s_2^2-s_1^2} \int_{Q_{\rho_2}(z_0)}  \Phi_\chi(|\tfrac{\bfu}{r}|)  \, \d z +\frac{c}{(s_2-s_1)^q} \int_{Q_{\rho_2}(z_0)}   \phi(|\tfrac{\bfu}{r}|)^{\chi+1} \, \d z\,.
\end{aligned}\]
Therefore, taking into account \eqref{eq:equivalence}, neglecting a non-negative term in the left hand side and recalling \eqref{Phigamma}, we have 
\begin{equation}
\begin{aligned}
&\sup_{-\rho_1^2<\tau<0} \int_{B_{\rho_2}(x_0)}   |\bfu(x,\tau)|^2 \phi(|\tfrac{\bfu(x,\tau)}{r}|)^\chi \xi^q \, \d x  +  \int_{Q_{\rho_2}(z_0)}  \phi(|D{\bf u}|)\phi(|\tfrac{\bfu}{r}|)^{\chi} \eta^2 \xi^q\, \d z\\
&\le  \frac{c(1+\chi)}{(s_2-s_1)^2} \int_{Q_{\rho_2}(z_0)}  |\tfrac{\bfu}{r}|^2\phi(|\tfrac{\bfu}{r}|)^\chi  \, \d z +\frac{c (1+\chi)}{(s_2-s_1)^q} \int_{Q_{\rho_2}(z_0)}   \phi(|\tfrac{\bfu}{r}|)^{\chi+1} \, \d z\,.
\end{aligned}
\label{eq:estimate1bis}
\end{equation}

\smallskip

\textit{Step 2. (Sobolev inequality)} Set
\[
{G(z) := r\tphi(|\tfrac{\bfu}{r}|)^{\chi+1} \eta^{\frac{2}{p_0}} \xi^{\frac{q}{p_0}} , \quad \text{where }\ \tphi(s):= \phi(s)^{\frac{1}{p_0}}
\ \ \text{and}\ \ 
p_0:=\frac{2n}{n+2}.}
\]
Then
\[
DG= (1+\chi)\tphi(|\tfrac{\bfu}{r}|)^{\chi} \tphi'(|\tfrac{\bfu}{r}|)\tfrac{\bfu D{\bf u}}{|\bfu|} \eta^{\frac{2}{p_0}} \xi^{\frac{q}{p_0}}  + r \tfrac{q}{p_0} \tphi(|\tfrac{\bfu}{r}|)^{\chi+1} \eta^{\frac{2}{p_0}} \xi^{\frac{q}{p_0}-1}D\xi\,,
\]
Now we apply Young's inequality \eqref{eq:young} to the $N$-function $\tphi$, with $t=|D{\bf u}|$ and $s=\tphi'(|\tfrac{\bfu}{r}|)$, together with \eqref{eq:hok2.4} to get 
\[\begin{aligned}
|DG| & \leq (1+\chi) \tphi(|\tfrac{\bfu}{r}|)^{\chi} \tphi'(|\tfrac{\bfu}{r}|)|D{\bf u}| \eta^{\frac{2}{p_0}} \xi^{\frac{q}{p_0}}  + c r \tphi(|\tfrac{\bfu}{r}|)^{\chi+1} |D\xi| \\
&\le c (1+\chi) \tphi(|\tfrac{\bfu}{r}|)^{\chi} \tphi(|D{\bf u}|) \eta^{\frac{2}{p_0}} \xi^{\frac{q}{p_0}}  + c \left(1+ \chi+\frac{1}{s_2-s_1}\right) \tphi(|\tfrac{\bfu}{r}|)^{\chi+1}\,.
\end{aligned}\]
Therefore, combining with \eqref{eq:estimate1bis} and recalling the definition of $\tphi$, we have 
\begin{equation}\label{1}\begin{aligned}
&\sup_{-\rho_1^2<\tau<0} \int_{B_{\rho_2}(x_0)}   |\bfu(x,\tau)|^2 \phi(|\tfrac{\bfu(x,\tau)}{r}|)^\chi \xi^q \, \d x +  \int_{Q_{\rho_2}(z_0)}  |DG|^{p_0}\, \d z\\
&\le  \frac{c(1+\chi)^2}{(s_2-s_1)^2} \int_{Q_{\rho_2}(z_0)}  |\tfrac{\bfu}{r}|^2\phi(|\tfrac{\bfu}{r}|)^\chi  \, \d z +\frac{c(1+\chi)^{p_0+1}}{(s_2-s_1)^q} \int_{Q_{\rho_2}(z_0)}   \phi(|\tfrac{\bfu}{r}|)^{\chi+1} \, \d z\\
&\le \frac{c(1+\chi)^{p_0+1}}{(s_2-s_1)^q} \int_{Q_{\rho_2}(z_0)}   \left(|\tfrac{\bfu}{r}|^2+\phi(|\tfrac{\bfu}{r}|)\right) \phi(|\tfrac{\bfu}{r}|)^{\chi} \, \d z .
\end{aligned}\end{equation}
Now, applying H\"older's inequality, {the Sobolev inequality to  the function $G\in W^{1,p_0}_0(B_r)$} and using \eqref{1}, we can write
\[\begin{aligned}
&\int_{Q_{\rho_1}(z_0)} |\tfrac{\bfu}{r}|^{\frac{2p_0}{n}} \phi(|\tfrac{\bfu}{r}|)^{1+\chi+\frac{\chi p_0}{n}} \,\frac{\d z}{r^{n+2}} \\
&\le \frac{1}{r^{n+2+\frac{2p_0}{n}}} \int_{-\rho_1^2}^0 \bigg(\int_{B_{\rho_1}(x_0)} |\bfu|^2 \phi(|\tfrac{\bfu}{r}|)^{\chi} \,\d x\bigg)^{\frac{p_0}{n}} \bigg(\int_{B_{\rho_1}(x_0)} \tphi(|\tfrac{\bfu}{r}|)^{\frac{(1+\chi)np_0}{n-p_0}} \,\d x\bigg)^{\frac{n-p_0}{n}} \,\d t\\
& \le \frac{1}{r^{(n+2)(1+\frac{p_0}{n})}} \bigg(\sup_{-\rho_1^2<\tau<0}\int_{B_{\rho_1}(x_0)} |\bfu(x,\tau)|^2\phi(|\tfrac{\bfu(x,\tau)}{r}|)^{\chi} \,\d x\bigg)^{\frac{p_0}{n}} \int_{-\rho_1^2}^0  \bigg(\int_{B_{\rho_2}(x_0)}|G|^{p_0^*} \,\d x\bigg)^{\frac{p_0}{p_0^*}} \,\d t\\
& \le \frac{c}{r^{(n+2)(1+\frac{p_0}{n})}} \bigg(\sup_{-\rho_1^2<\tau<0}\int_{B_{\rho_1}(x_0)} |\bfu(x,\tau)|^2 \phi(|\tfrac{\bfu(x,\tau)}{r}|)^{\chi} \,\d x\bigg)^{\frac{p_0}{n}} \int_{Q_{\rho_2}(z_0)} |DG|^{p_0} \,\d z\\
&\le c\left\{\frac{(1+\chi)^{p_0+1}}{(s_2-s_1)^q} \int_{Q_{\rho_2}(z_0)}   \left(|\tfrac{\bfu}{r}|^2+\phi(|\tfrac{\bfu}{r}|)\right) \phi(|\tfrac{\bfu}{r}|)^{\chi} \, \frac{\d z}{r^{n+2}} \right\}^{1+\frac{p_0}{n}}.
\end{aligned}\]
Then, since $s^{2}\le c \phi(s)s^{\frac{2p_0}{n}}=c \phi(s)s^{\frac{4}{n+2}}$ for $s\ge 1$ by $p> \frac{2n}{n+2}$, recalling \eqref{psi}, we have 
\begin{equation}
\begin{aligned}
&\int_{Q_{\rho_1}(z_0)} \psi(|\tfrac{\bfu}{r}|)\phi(|\tfrac{\bfu}{r}|)^{\chi(1+\frac{p_0}{n})} \,\frac{\d z}{|Q_{2r}|} \le c \int_{Q_{\rho_1}(z_0)} \phi(|\tfrac{\bfu}{r}|)^{1+\chi+\frac{\chi p_0}{n}}|\tfrac{\bfu}{r}|^{\frac{2p_0}{n}} \,\frac{\d z}{|Q_{2r}|} +c\\
&\le c \left\{\frac{(1+\chi)^{p_0+1}}{(s_2-s_1)^q} \int_{Q_{\rho_2}(z_0)}  \left[\psi(|\tfrac{\bfu}{r}|)\phi(|\tfrac{\bfu}{r}|)^\chi+1\right]  \, \frac{\d z}{|Q_{2r}|}\right\}^{1+\frac{p_0}{n}}.
\end{aligned}
\label{eq:estimate2bis}
\end{equation}

\smallskip

\textit{Step 3. (Iteration)} We first notice that by applying the Gagliardo-Nirenberg type interpolation inequality
to $\phi$ with $p>\frac{2n}{n+2}$
provided by Lemma~\ref{Lem:paraembedding} and Remark~\ref{rmk:embedding}, we have 
\[
\int_{Q_{2r}}  \psi(|\bfu|)\phi(|\bfu|)^{\chi_0}  \, \d z<\infty
\]
where
\begin{equation}\label{chi0}
\chi_0:= \min\left\{\left(\frac{p(n+2)}{n}-2\right)\frac{1}{q},\frac{2}{n}\right\}>0\, .
\end{equation}

For $m=0,1,2,\dots$, set $\chi_m=\chi_0\theta^m$ and
\[
J_{m}:= \int_{Q_{r_m}}  \left[\psi(|\tfrac{\bfu}{r}|)\phi(|\tfrac{\bfu}{r}|)^{\chi_m}  +1 \right]\, \frac{\d z}{|Q_{2r}|}, 
\quad \text{where} \ \ \theta:=1+\frac{p_0}{n} \ \ \text{and}\ \ r_m:=r(1+2^{-m})\,.
\]
Then, we can iterate \eqref{eq:estimate2bis} and write
\[
J_{m}  \le c 2^{q\theta m}(1+\chi_0\theta^{m-1})^{(p_0+1)\theta}J_{m-1}^{\theta}  \le c_0^m  J_{m-1}^\theta , \quad m=1,2,\dots,
\]
for some $c_0\ge 1$ depending on $n$, $N$,  $p$, $q$ and $L$. Hence, for $m\ge 2$,
\[
J_{m}  \le c_0^m  \left(c_0^{m-1}J_{m-2}^\theta \right)^{\theta}  \le c_0^{m+(m-1)\theta} J_{m-2}^{\theta^2} \le \cdots  \le  c_0^{\sum_{k=1}^m(m-k+1)\theta^{k-1}} J_0^{\theta^m}   \le (c_1 J_0 )^{\theta^m}
\]
for some large $c_1\ge 1$ depending on $n$, $N$, $p$, $q$ and $L$. Consequently, setting 
\[
\d\mu(z) := \psi(|\tfrac{\bfu(z)}{r}|)\frac{\d z}{|Q_{2r}|},
\]
we have 
\[\begin{aligned}
\|\phi(|\tfrac{\bfu}{r}|)\|_{L^\infty(Q_r)} \le \|\phi(|\tfrac{\bfu}{r}|)\|_{L^\infty(Q_r;\d\mu)} 
&=\lim_{m\to \infty} \left(\int_{Q_{r}}  \phi(|\tfrac{\bfu}{r}|)^{\chi_0\theta^m}  \, \d\mu \right)^{\frac{1}{\chi_0\theta^m}} \\
& \le  \limsup_{m\to \infty}\, (J_m)^{\frac{1}{\chi_0\theta^m}} \le (c_1 J_0 )^{\frac{1}{\chi_0}} .
\end{aligned}\]
This implies the estimate \eqref{estimate:bounded} and the proof is concluded.
\end{proof}

\vspace{0.5cm}

\section{Approximating problems and Second order differentiability}

Let $\bfu$ be a weak solution to \eqref{eq:system}. Then for $Q_{R} \Subset \Omega_{T}$ and a sufficiently small $\varepsilon\in(0,1)$ we consider the following non-degenerate parabolic system with Cauchy-Dirichlet boundary condition: 
\begin{equation}\label{eq:nondegenerate system}
\begin{cases}
\displaystyle (\bfu_\epsilon)_{t}-\div\left(\frac{\phi_{\epsilon}'(|D \bfu_\epsilon|)}{|D \bfu_\epsilon|} D \bfu_\epsilon\right)=0 & \quad \text{in }\ Q_{R}\,,\\
 \bfu_\epsilon=\bfu &\quad \text{on }\ \partial_{\rm p}Q_{R}\,.
\end{cases}
\end{equation}
where $\phi_\epsilon$ is the shifted $N$-function with $a=\epsilon$. 
The existence of the weak solution to the above system follows from the theory of monotone operators or by using the Galerkin approximation method, see for instance \cite{Lions}.
We will show that the system \eqref{eq:system} can be approximated by \eqref{eq:nondegenerate system}, in the sense that if $\bf\bfu_\epsilon$ is the weak solution to \eqref{eq:nondegenerate system}, {then $D{\bf u}_\epsilon$ converges to  $D{\bf u}$ in $L^p(Q_R)$ (see Lemma~\ref{Lem:approximation}).} 
%

We first prove  second differentiability in the spatial variable $x$ for each weak solution to the following non-degenerated problem without boundary condition:
\begin{equation}
\displaystyle \bfw_{t}-\div\left(\frac{\phi_{\epsilon}'(|D \bfw|)}{|D \bfw|} D \bfw\right)=0 \quad \text{in }\ \Omega_T\,.
\label{eq:nondegenerate systemeq}
\end{equation}
%
In order to do that, we fix some notation. For a (vector-valued) function $f$, we introduce the notation
$$
\Delta_{k,s} f(x, t):= f\left(x+s e_{k}, t\right)-f(x, t)\,,
$$
where $s \in \mathbb{R}$ and $e_{k}$ with $k \in\{1,2, \ldots, n\}$ is a standard unit vector in $\R^{n}$. Moreover, we define $T_{k,s}:\R^{n+1}\to \R^{n+1}$ by
\begin{equation}
T_{k,s}(x,t):=(x+se_k,t), \qquad (x,t)\in \R^{n+1}.
\label{eq:translations}
\end{equation}
Then we have the following result (cfr. \cite[Theorem 6]{Choe91} where analogous estimates are devised for parabolic systems with $p$-growth, and \cite[Theorem 11]{DieEtt08}, \cite[Lemma 5.7]{DieStrVer09} for analogous arguments for elliptic systems with $\varphi$-growth).

\begin{lemma} \label{Lem:seconddifferential}
Let $\phi$ satisfy Assumption~\ref{Ass2} with \eqref{eq:hypp} and \eqref{p2q}, and let $\bfu_\epsilon$  be a weak solution to \eqref{eq:nondegenerate systemeq} with $\epsilon>0$. 
Then
\begin{itemize}
\item [$(i)$] $\bV^\epsilon (D \bfu_\epsilon)\in L^2_{\loc}(0,T; W^{1,2}_{\loc}(\Omega;\R^{Nn}))$;
\item [$(ii)$]$D \bfu_\epsilon \in L^p_{\loc}(0,T; W^{1,p}_{\loc}(\Omega;\R^{Nn}))\cap L^\infty_{\loc}(0,T; L^2_{\loc} (\Omega;\R^{Nn}))$;
\item [$(iii)$] if, in addition, $D \bfu_\epsilon \in L_{\loc}^{\infty}(\Omega_T;\R^{Nn})$, then $D \bfu_\epsilon \in L^2_{\loc}(0,T; W^{1,2}_{\loc}(\Omega;\R^{Nn}))$ and $\phi_\epsilon(|D \bfu_\epsilon|) \in L^2_{\loc}(0,T; W^{1,2}_{\loc}(\Omega))$.
\end{itemize}
\end{lemma}


\begin{proof} 
In order to enlighten the notation, we will denote $\bfu_{\varepsilon}$ by $\bfw $. 
Note that in view of Theorem~\ref{thm:bound}, $\bfw  \in L^{\infty}_{\loc}\left(\Omega_T\right)$. 
Fix any $Q_{25 r}=Q_{25 r}\left(x_{0}, t_{0}\right) \Subset \Omega_T$ with $r\in(0,1)$,  and set $\lambda:=\sup _{Q_{2r}} \bfw  $.
Let $s,h\in \R$ be such that $0<s\le h<r/2$ or $-r/2<h\le s<0$. Then $\bfw $ satisfies
\begin{equation}
\left(\Delta_{k,s} \bfw \right)_{t}-\operatorname{div}\left(\Delta_{k,s} (\bA^{\epsilon}(D \bfw ))\right)=0 \text { in } Q_{3r/2}
\label{eq:weakeqneps}
\end{equation}
in the weak sense, where $\bA^{\epsilon}$ is defined as in \eqref{eq:shiftedop} with $a=\varepsilon$.

Let $\zeta \in C^{\infty}\left(Q_{3 r/2}\right)$ be such that $0 \le \zeta \le 1$, $\zeta=0$ in $Q_{2r}\setminus Q_{3 r/2}$, $\zeta \equiv 1$ in $Q_{r}$, and $r^{-1}|D \zeta|+$ $\left|D^{2} \zeta\right|+\left|\zeta_{t}\right| \le \frac{c(n)}{r^{2}}$. Then, testing \eqref{eq:weakeqneps} with the function $(\Delta_{k,s} \bfw ) \zeta^q$, 
for every $t_0\in (-r^2,0]$ we obtain
$$\begin{aligned}
0=&\int^{t_0}_{-4r^2}\int_{B_{2r}}\left(\Delta_{k,s} \bfw \right)_{t} \cdot (\Delta_{k,s} \bfw ) \zeta^q \, \d x\, \d t\\
&+\int^{t_0}_{-4r^2}\int_{B_{2r}} \Delta_{k,s}(\bA^{\epsilon}(D \bfw )): (\zeta^q D [\Delta_{k,s} \bfw ] + q\zeta^{q-1} \Delta_{k,s} \bfw  \otimes D \zeta  )\, \d x\, \d t .
\end{aligned}$$
Note that an integration by parts with respect to the time variable gives
$$\begin{aligned}
\int^{t_0}_{-4r^2} \int_{B_{2r}} (\Delta_{k,s} \bfw )_{t} \cdot (\Delta_{k,s} \bfw )\zeta^q \,\d x\, \d t 
& =\frac{1}{2} \int_{B_{2r}}|\Delta_{k,s} \bfw (x,t_0)|^{2} \zeta^q \, \d x\\
&\qquad -\frac{q}{2} \int^{t_0}_{-4r^2} \int_{B_{2r}} |\Delta_{k,s} \bfw |^2  \zeta^{q-1} \zeta_{t} \, \d x\, \d t.
\end{aligned}$$
Hence, for every $t_0\in[-r^2,0]$ we have that
$$\begin{aligned}
& \int_{B_{r}} |\Delta_{k,s} \bfw  (x, t_{0})|^{2}  \, \d x  \\
&\qquad+2 \int^{t_0}_{-4r^2} \int_{B_{2 r}} \Delta_{k,s}(\bA^{\epsilon}(D\bfw )):\left( \Delta_{k,s} (D\bfw ) \zeta^q+q\zeta^{q-1}\Delta_{k,s} \bfw  \otimes D \zeta \right) \, \d x\, \d t \\
&\le q \int^{t_0}_{-4r^2} \int_{B_{2r}}|\Delta_{k,s} \bfw |^{2} \zeta^{q-1} \zeta_{t} \, \d x\, \d t.
\end{aligned}$$
We first observe that by \eqref{monotonicity}, 
$$
\begin{aligned}
&\Delta_{k,s}  (\bA^{\epsilon}(D \bfw )): \Delta_{k,s} D  \bfw   \\
&  = \big[\bA^{\epsilon}(D\bfw (x+se_k,t))-\bA^{\epsilon}(D\bfw (x,t))\big]: \big[D\bfw (x+se_k,t)-D\bfw (x,t)\big]\\
&\ge \frac{1}{c} |\Delta_{k,s} \bV^\epsilon (D \bfw )|^2.
\end{aligned}
$$
Now, by \eqref{PhiPQ} and \eqref{eq:DE(6.7)},
$$\begin{aligned}
|\Delta_{k,s} \bA^\epsilon(D \bfw )| & =  |\bA^\epsilon(D\bfw (x+se_k,t))-\bA^\epsilon(D\bfw (x,t))| \\
&\le c \frac{\phi_\epsilon'(|D\bfw (x+se_k,t)|+|D\bfw (x,t)|)}{|D\bfw (x+se_k,t)|+|D\bfw (x,t)|}  | \Delta_{k,s} D\bfw (x,t)| \\
&\le c\phi_{\epsilon+|D\bfw |}'(| \Delta_{k,s} D\bfw (x,t)| )\,,
\end{aligned}$$
whence
\[\begin{aligned}
&|\Delta_{k,s}(\bA^{\epsilon}(D \bfw )):\Delta_{k,s} \bfw  \otimes D \zeta \zeta^{q-1}|\\
&\le c \fint_0^s\underbrace{\phi_{\epsilon+|D\bfw |}'(| \Delta_{k,s} D\bfw | )  |s| |D\bfw (x+\tau e_k,t)| | D \zeta | }_{=:J}\, \d \tau\,.
\end{aligned}\]
Note that we shall write $\fint^s_0 = \frac{1}{s} \int_0^s$ even if $s\in \R$ is negative.
Recalling \eqref{eq:translations}, we have
\begin{equation*}
\bfw (x+\tau e_k,t) = T_{k,\tau}\circ \bfw (x,t)\,.
\end{equation*}

Now, from \eqref{eq:DE(6.19)}, 
Young's inequality in \eqref{eq:young4} and \eqref{eq:(3.4)}, recalling $h\in \R$ fixed above, for any sufficiently small $\delta>0$ we have that
\[\begin{aligned}
J& \le  c \left(\phi_{\epsilon+|D[T_{k,\tau}\circ \bfw ]|}'(|\Delta_{k,s-\tau}D[T_{k,\tau}\circ \bfw ]|)+ \phi_{\epsilon+|D[T_{k,\tau}\circ \bfw ]|}'(|\Delta_{k,\tau}D \bfw |) \right) |s| |D[T_{k,\tau}\circ \bfw ]| |D\zeta|  \\
&\le  \delta \left(\frac{|s|}{|h|}\right)^{\frac{q}{q-1}}\Big\{\phi_{\epsilon+|D[T_{k,\tau}\circ \bfw ]|}(|\Delta_{k,s-\tau}D[T_{k,\tau}\circ \bfw ]|)+ \phi_{\epsilon+|D[T_{k,\tau}\circ \bfw ]|} (|\Delta_{k,\tau} D \bfw |) \big\} \\
&\qquad + c_\delta\phi_{\epsilon+|D[T_{k,\tau}\circ \bfw ]|}\left( |h| |D[T_{k,\tau}\circ \bfw ]| |D\zeta|   \right)\\
&\le  \delta  \frac{|s|}{|h|} \Big( |\Delta_{k,s-\tau} \bV^{\epsilon} (D[T_{k,\tau}\circ \bfw ])|^2 + |\Delta_{k,\tau} \bV^{\epsilon}(D \bfw )|^2 \big) + c_\delta \frac{h^2}{r^2} \phi_{\epsilon}\left(  |D[T_{k,\tau}\circ \bfw ]|    \right)\,,
\end{aligned}\]
where we have used the facts that $0<|s|\le |h| < r/2$, 
$|D\zeta|\le \frac{c}{r}\le \frac{c}{|h|}$ and for $\epsilon>0$ and $s\in(0,1]$ 
$$
\phi^*_a(s t) \le  s^{\frac{q}{q-1}} \phi^*_a(t)
$$
and
$$
\phi_{\epsilon+t}(st)\sim \frac{\phi'(\epsilon+t+st)}{\epsilon+t+st}(st)^2 \sim s^2 \frac{\phi'(\epsilon+t)}{\epsilon+t}t^2 \sim s^2\phi_\epsilon(t), \quad s\in [0,1].
$$
Note that applying Fubini's Theorem and the change of variables $y=x+\tau e_k$ and $\tilde \tau = s-\tau$, we get
$$\begin{aligned}
& \int_{Q_{3r/2}}\fint_0^s |\Delta_{k,s-\tau} \bV^{\epsilon} (D[T_{k,\tau}\circ \bfw ])|^2 \, \d\tau \, \d x\\
& =\fint_0^s  \int_{Q_{3r/2}} |\bV^\epsilon(D\bfw (x+\tau e_k+ (s-\tau)e_k,t))-\bV^\epsilon(D\bfw (x+\tau  e_k,t))|^2 \, \d x \,  \d\tau\\
& \leq \fint_0^s  \int_{Q_{2r}} |\bV^\epsilon(D\bfw (y+ \tilde \tau e_k,t))-\bV^\epsilon(D\bfw (y,t))|^2 \, \d y \,  \d\tilde \tau \,.
\end{aligned}$$
Therefore, we have
\begin{equation}\label{supint}
\begin{aligned}
&\sup_{t_0\in[-r^2,0]}\int_{B_{r}} [\Delta_{k,s} \bfw \left(x, t_{0}\right)]^{2}  \, \d x +\int_{Q_{r}} |\Delta_{k,s} \bV^\epsilon (D \bfw )|^2\, \d z  \\
& \le c \delta \frac{|s|}{|h|}\fint_0^s  \int_{Q_{2r}} |\Delta_{k,\tau} \bV^\epsilon (D \bfw )|^2\, \d x\, \d \tau  + c(\delta) \frac{h^2}{r^2} \int_{Q_{2r}} \phi_\epsilon (|D\bfw |)\, \d z  +  \frac{c}{r^2} \int_{Q_{2r}}|\Delta_{k,s} \bfw |^{2}\zeta^{q-1}\, \d z.
\end{aligned}\end{equation}
Further, we estimate the final term in the right hand side of the inequality above using the integration by parts:
$$\begin{aligned}
&\int_{Q_{2r}}|\Delta_{k,s} \bfw |^{2}\zeta^{q-1}\, \d z 
= \int_0^s \int_{Q_{2r}}  \bfw _{x_k}(x+\tau e_k,t) \cdot \Delta_{k,s} \bfw  \zeta^{q-1}  \, \d z \, \d \tau \\
&= - \int_0^s \int_{Q_{2r}}  \bfw (x+\tau e_k,t) \cdot (\Delta_{k,s}  \bfw _{x_k} \zeta^{q-1}+ (q-1)\Delta_{k,s} \bfw  \zeta^{q-2}\zeta_{x_k}) \,  \d z \, \d \tau \\
&\le |s|\lambda \int_{Q_{3r/2}}  |\Delta_{k,s}  D\bfw | + |\Delta_{k,s} \bfw | |D\zeta| \,  \d z \\ 
&\le |s|\lambda \int_{Q_{3r/2}} (\epsilon+|D\bfw (x+s e_k,t)|+|D\bfw (x,t)|)^{\frac{p-2}{2}+\frac{2-p}{2}}|\Delta_{k,s}  D\bfw |  \,  \d z \\
&\qquad  +  \frac{c\lambda s^2}{r} \int_{Q_{3r/2}}  |s^{-1}\Delta_{k,s} \bfw |  \,  \d z \,.
\end{aligned}$$
Now applying Young's inequality and using facts that $\frac{\phi(t)}{t^{p-1}}$ is increasing and $t^p\le c (\phi(t)+1)$, we have that for any $\tilde \delta\in(0,1)$, in order to reabsorb some terms to the left-hand side,
$$\begin{aligned}
\int_{Q_{2r}}|\Delta_{k,s} \bfw |^{2}\zeta^{q-1}\, \d z  & \le \tilde \delta   \frac{\phi'(\epsilon)r^2}{\epsilon^{p-1}}  \int_{Q_{3r/2}} (\epsilon+|D\bfw (x+s e_k,t)|+|D\bfw (x,t)|)^{p-2}|\Delta_{k,s}  D\bfw |^2  \,  \d z \\
&\ \  + c \tilde \delta^{-1}  \frac{\epsilon^{p-1}}{r^2\phi'(\epsilon)} \lambda^2s^2 \int_{Q_{3r/2}} (\epsilon+|D\bfw (x+s e_k,t)|+|D\bfw (x,t)|)^{2-p} \,  \d z \\
&\ \ +  \frac{c\lambda s^2}{r} \int_{Q_{2r}}  |D \bfw |  \,  \d z \\
& \le \tilde \delta  r^2 \int_{Q_{3r/2}} \frac{\phi'(\epsilon+|D\bfw (x+s e_k,t)|+|D\bfw (x,t)|)}{\epsilon+|D\bfw (x+s e_k,t)|+|D\bfw (x,t)|}|\Delta_{k,s}  D\bfw |^2  \,  \d z \\
&\qquad +  c \frac{\lambda s^2}{\tilde \delta r^2} \left( \frac{\epsilon^{p-1}}{\phi'(\epsilon)}\lambda +\frac{1}{r}\right) \int_{Q_{2r}} [\phi (|D \bfw |)+1]  \,  \d z \\
& \le\tilde  \delta  r^2\int_{Q_{2r}}  |\Delta_{k,s}\bV^\epsilon (D\bfw )|^2  \,\d z\\
&\quad+ c \frac{s^2 \lambda}{\tilde \delta r^2 }  \left( \frac{\epsilon^{p-1}}{\phi'(\epsilon)}\lambda +1\right) \int_{Q_{2r}} [\phi (|D \bfw |)+1]  \,  \d z.
\end{aligned}$$
Finally we have that for every $Q_{2r}\Subset Q_R$, $\delta,\tilde \delta \in(0,1)$ and $s,h\in \R $ with $0<s \le h<r/2$ or $-r/2 <h\le s<0$,
\begin{equation}\label{quotient_pf1}\begin{aligned}
&\int_{Q_{r}} |\Delta_{k,s} \bV^\epsilon (D \bfw )|^2\, \d z
 \le  c \delta \frac{|s|}{|h|}\fint_0^s  \int_{Q_{2r}} |\Delta_{k,\tau} \bV^\epsilon (D \bfw )|^2\, \d x\, \d \tau \\
& \qquad +c \tilde\delta \int_{Q_{2r}}  |\Delta_{k,s}\bV^\epsilon (D\bfw )|^2  \,\d z  
  +  \frac{h^2}{r^4} C(\delta,\tilde\delta,\epsilon,\lambda) \int_{Q_{2r}} [\phi (|D \bfw |)+1]  \,  \d z \,.
\end{aligned}\end{equation}

Now, we re-absorb the first two terms on the right hand side. To do this, we first integrate both the sides of \eqref{quotient_pf1} with respect to $s$ from $0$ to $h$ and apply Fubini's Theorem, so that
$$\begin{aligned}
&\fint_0^h\int_{Q_{r}} |\Delta_{k,s} \bV^\epsilon (D \bfw )|^2\, \d z\, \d s \\
& \le c \delta \fint_0^h \frac{|s|}{|h|}\fint_0^s  \int_{Q_{2r}} |\Delta_{k,\tau} \bV^\epsilon (D \bfw )|^2\, \d x\, \d \tau \, \d s  + c  \tilde \delta \fint_0^h\int_{Q_{2r}}  |\Delta_{k,s}\bV^\epsilon (D\bfw )|^2  \,\d z\, \d s \\
&\qquad +  \frac{h^2}{r^4} C(\delta,\tilde\delta,\epsilon,\lambda)\int_{Q_{2r}} [\phi (|D \bfw |)+1]  \,  \d z\\
& \le c (\delta +\tilde\delta) \fint_0^h  \int_{Q_{2r}} |\Delta_{k,s} \bV^\epsilon (D \bfw )|^2\, \d x\, \d s  +  \frac{h^2}{r^4} C(\delta,\tilde\delta,\epsilon,\lambda) \int_{Q_{2r}} [\phi (|D \bfw |)+1]  \,  \d z.
\end{aligned}$$
Therefore, applying the Giaquinta-Modica type covering argument in \cite[Lemma 13]{DieEtt08} we have that 
\[
\fint_0^h\int_{Q_{r}} |\Delta_{k,s} \bV^\epsilon (D \bfw )|^2\, \d z\, \d s
\le   c \frac{h^2}{r^4} C(\epsilon,\lambda) \int_{Q_{5r}} [\phi (|D \bfw |)+1]  \,  \d z
\]
for every $Q_{5r}\Subset Q_R$ and $h\in \R $ with $0<h<\frac{r}{10}$ or $-\frac{r}{10} <h<0$. Inserting this into \eqref{quotient_pf1} with $s=h$, we have 
$$
\int_{Q_{r}} |\Delta_{k,h} \bV^\epsilon (D \bfw )|^2\, \d z
 \le  c \tilde \delta  \int_{Q_{2r}}  |\Delta_{k,h}\bV^\epsilon (D\bfw )|^2  \,\d z  
+ c \frac{h^2}{r^4} C(\tilde\delta,\epsilon,\lambda) \int_{Q_{5r}} [\phi (|D \bfw |)+1]  \,  \d z.
$$
Again applying the same covering argument, we have that 
for every $Q_{25r}\Subset Q_R$ and $h\in \R $ with $0<|h|<\frac{r}{50}$, 
\begin{equation}\label{quotient_pf2}
\frac{1}{h^2}\int_{Q_{r}} |\Delta_{k,h} \bV^\epsilon (D \bfw )|^2\, \d z
 \le \frac{c}{r^4} C(\epsilon,\lambda) \int_{Q_{25r}} [\phi (|D \bfw |)+1]  \,  \d z.
\end{equation}
Letting $h\to0 $ in \eqref{quotient_pf2}, we then obtain $(i)$.

The results in $(ii)$ and $(iii)$ are  direct consequences of \eqref{quotient_pf2}. Let $Q_{25r}\Subset Q_R$ and $h\in \R $ with $0<|h|<\frac{r}{50}$. By Young's inequality we have that
$$\begin{aligned}
\int_{Q_{r}} h^{-p}|\Delta_{k,h} D \bfw |^p\, \d z & \le c h^{-2} \int_{Q_{r}} (\epsilon+|D\bfw (x+he_k)|+|D\bfw (x)|)^{p-2}|\Delta_{k,h} D \bfw |^2\, \d z \\
&\qquad + c  \int_{Q_{r}} (\epsilon+|D\bfw (x+he_k)|+|D\bfw (x)|)^{p}\, \d z\\
& \le c h^{-2} \frac{\epsilon}{\phi'(\epsilon)} \int_{Q_{r}} \frac{\phi'(\epsilon+|D\bfw (x+he_k)|+|D\bfw (x)|)}{\epsilon+|D\bfw (x+he_k)|+|D\bfw (x)|}|\Delta_{k,h} D \bfw |^2\, \d z \\
&\qquad + c  \int_{Q_{25r}}[ \phi(|D\bfw |)+1]\, \d z\\
&\le   \frac{c}{r^4} C(\epsilon,\lambda) \int_{Q_{25r}} [\phi (|D \bfw |)+1]  \,  \d z,
\end{aligned}
$$
and, passing to the limit as $h\to0$, this implies  the first half of $(ii)$. Moreover, by estimate \eqref{supint}, passing to the limit as $h\to0$, we get the second half of $(ii).$

Finally we prove $(iii)$. Set  $M:= \|D\bfw \|=\|D{\bf u}_\epsilon\|_{L^\infty(Q_{25r},\R^{Nn})}$. Then, from \eqref{quotient_pf2}, we have 
$$\begin{aligned}
\int_{Q_{r}}|\Delta_{k,h} D\bfw |^2\, \d z&\le \frac{\epsilon+2M}{\phi'(\epsilon)} \int_{Q_{r}} \frac{\phi'(\epsilon+|D \bfw (x+he_k,t)|+|D \bfw (x,t)| )}{\epsilon+|D \bfw (x+he_k,t)|+|D \bfw (x,t)| }|\Delta_{k,h} D\bfw |^2\, \d z\\
 &\le c\frac{\epsilon+2M}{\phi'(\epsilon)}\int_{Q_{r}} |\Delta_{k,h} \bV^\epsilon (D \bfw )|^2\, \d z \\
 & \le c  \frac{\epsilon+2M}{\phi'(\epsilon)} \frac{h^2}{r^4} C(\epsilon,\lambda) \int_{Q_{25r}} [\phi (|D \bfw |)+1]  \,  \d z.
\end{aligned}$$
This implies that $D\bfw \in L^2_{\loc}(-R^2,0; W^{1,2}_{\loc}(B_R;\R^{Nn}))$. Moreover, since $(\phi_\epsilon(|D\bfw |))_{x_k}=\phi'_\epsilon(|D\bfw |)\frac{D\bfw  :D\bfw _{x_k}}{|D\bfw |}$,
$$
\int_{Q_{r}}|D[\phi_\epsilon(|D\bfw |)]|^2\, \d z \le c [\phi_\epsilon'(M)]^2\int_{Q_{r}}| D^2\bfw |^2\, \d z.
$$
From this we get $\phi_\epsilon(|D\bfw |) \in  L^2_{\loc}(-R^2,0; W^{1,2}_{\loc}(B_R))$, and the proof concludes.
\end{proof}

Combining (ii) of the above lemma with the parabolic Sobolev inequality, see \cite[I, Proposition 3.1]{DiB_book},
we obtain the following result:
\begin{lemma}\label{Lem:seconddifferentiability2}
Let $\phi$ satisfy Assumption~\ref{Ass2} with \eqref{eq:hypp}, and  let $\bfu_\epsilon$ be a weak solution to \eqref{eq:nondegenerate systemeq}. 
Then 
$|D\bfu_\epsilon| \in L^{\frac{p(n+2)}{n}}_{\loc}(\Omega_T)$. 
\end{lemma}


We end this section with the convergence results of $D{\bf u}_\epsilon$ to $D{\bf u}$.


\begin{lemma}\label{Lem:approximation}
Let $u$ be a weak solution to \eqref{eq:system} and $\bfu_\epsilon$ be the weak solution to \eqref{eq:nondegenerate system} with $Q_R\Subset \Omega_T$. Then 
$D{\bf u}_\epsilon$ converges to  $D{\bf u}$ in $L^\phi(Q_R)$. 
\end{lemma}

\begin{proof}
By virtue of \eqref{monotonicity}, it suffices to show that 
\begin{equation}
\lim_{\epsilon \to 0+}\int_{Q_R} |\bV(D{\bfu}_\varepsilon)-\bV(D{\bf u})|^2\,\mathrm{d}z = 0.
\label{eq:L2conv0}
\end{equation}
By following  the proof of \cite[Theorem~3.5]{DieSchSch19}, one has 
\begin{equation}
\lim_{\epsilon \to 0+}\int_{Q_R} |\bV^\varepsilon(D{\bfu}_\varepsilon)-\bV(D{\bf u})|^2\,\mathrm{d}z = 0.
\label{eq:L2conv}
\end{equation}
 Moreover, together with \eqref{eq:approx}, this implies that 
$$\begin{aligned}
\int_{Q_R} \phi(|D\bfu_\epsilon|)\, \d z 
&\le c \int_{Q_R} \Big[\phi_\epsilon(|D\bfu_\epsilon|) + \phi(\epsilon)\Big]\, \d z \\
& \le c \int_{Q_R} \Big[|\bV^\epsilon(D\bfu_\epsilon)|^2 + \phi(\epsilon)\Big]\, \d z \\
& \le c \int_{Q_R} \Big[|\bV(D\bfu)|^2 + \phi(\epsilon)+1\Big]\, \d z \\
& \le  c \int_{Q_R} \left[\phi(|D\bfu|) +1\right]\, \d z
\end{aligned}$$
for any sufficiently small $\epsilon>0$. Applying  \eqref{eq:(3.22)}, 
the change of shift formula \eqref{(5.4diekreu)}, $\varphi_a(a)\leq c \varphi(a)$ and the preceding inequality we have that for any $\delta\in(0,1)$,
$$\begin{aligned}
\int_{Q_R}  |\bV^\varepsilon(D{\bfu}_\varepsilon)-\bV(D\bfu_\epsilon)|^2 \, \d z 
&\le c \int_{Q_R}  \phi_{|D\bfu_\epsilon|} (\epsilon)  \, \d z \\  
& \le c\delta \int_{Q_R}   \phi_{|D\bfu_\epsilon|}(|D\bfu_\epsilon|) \, \d z  + c_\delta \phi(\epsilon) |Q_R|\\
& \le c\delta \int_{Q_R}   \phi(|D\bfu_\epsilon|) \, \d z  + c_\delta \phi(\epsilon) |Q_R|\\
& \le c \delta \int_{Q_R}   \Big[ \phi(|D\bfu|)+1\Big] \, \d z  + c_\delta \phi(\epsilon) |Q_R| \,,
\end{aligned}$$
hence
$$
\limsup_{\epsilon \to 0^+ }\int_{Q_R}  |\bV^\varepsilon(D{\bfu}_\varepsilon)-\bV(D\bfu_\epsilon)|^2 \, \d z  \le 
  c \delta \int_{Q_R}   \Big[ \phi(|D\bfu|)+1\Big] \, \d z \,.
$$
Therefore, since $\delta\in (0,1)$ is arbitrary, we have 
$$
\lim_{\epsilon \to 0+ }\int_{Q_R}  |\bV^\varepsilon(D{\bfu}_\varepsilon)-\bV(D\bfu_\epsilon)|^2 \, \d z =0.
$$
This and \eqref{eq:L2conv} yield \eqref{eq:L2conv0}.
\end{proof}

\vspace{0.5cm}

\section{Local boundedness of the gradient} \label{sec:locboundgradient}

Now, we address the problem of obtaining an $L^{\infty}$-bound for $D{\bf u}$, by deriving uniform estimates in $\epsilon$ for weak solutions to the non-degenerate systems \eqref{eq:nondegenerate system}. We follow some ideas underlying the Moser iteration for the Lipschitz regularity for parabolic $p$-Laplace systems, which can be found in \cite[Theorem 4]{Choe91}, \cite[Theorem 4]{Choe92} and \cite[Proposition 3.1]{CheDiBe89}.

 Note that $D{\bf u}_\epsilon$ is weakly differentiable with respect to the spatial variable $x$ by Lemma~\ref{Lem:seconddifferential}$(ii)$. Hence differentiating \eqref{eq:nondegenerate systemeq} with respect to $x_{k}$ we find
\begin{equation}\label{diffsystem}
\begin{aligned}
\partial_{t} (\bfu_\epsilon)_{x_{k}}^{\alpha} &=\left(\frac{\phi_\epsilon^{\prime}(|D \bfu_\epsilon|)}{|D \bfu_\epsilon|} (\bfu_\epsilon)_{x_{i} x_{k}}^{\alpha}+\left(\frac{\phi_\epsilon^{\prime \prime}(|D \bfu_\epsilon|)}{|D \bfu_\epsilon|}-\frac{\phi_\epsilon^{\prime}(|D \bfu_\epsilon|)}{|D \bfu_\epsilon|^{2}}\right) \frac{(\bfu_\epsilon)_{x_{j}}^{\beta} (\bfu_\epsilon)_{x_{j} x_{k}}^{\beta}}{|D \bfu_\epsilon|} (\bfu_\epsilon)_{x_{i}}^{\alpha}\right)_{x_{i}} \\
&=:\left(\bar{A}^{\alpha \beta}_{i j} (\bfu_\epsilon)_{x_{j} x_{k}}^{\beta}\right)_{x_{i}}, \quad \alpha =1,2, \ldots, N,
\end{aligned}
\end{equation}
where 
$$
\bar{A}^{\alpha \beta}_{i j}:= ({\bf A}^\epsilon)^{\alpha \beta}_{i j}(D{\bf u}_\epsilon)= \left.\frac{\partial \bA^\epsilon({\bf Q})^{\beta}_{j}}{\partial Q^{\beta}_{i}}\right|_{{\bf Q}=D{\bf u}_\epsilon}.
$$
We start with obtaining a Caccioppoli type inequality for the system \eqref{diffsystem}.

\begin{lemma}\label{Lem:Caccio1}
Let $\phi$ satisfy Assumption~\ref{Ass2} with \eqref{eq:hypp} and \eqref{p2q}, and let $\bfu_\epsilon$ be a weak solution to  \eqref{eq:nondegenerate systemeq}. Suppose $f\in C^{0,1}([0,\infty))$ is positive, increasing and satisfying $f'(s)>0$ for a.e. $s$ with $f(s)>0$, and set $F(s) := \int_0^s \tau f(\tau)\, \d \tau$.  For every $Q:=B_\rho \times [t_1,t_2] \Subset \Omega_T$,  $\xi \in C^\infty_0(B_\rho)$ with $0\le \xi \le1$ and  $\eta\in C^\infty(\R)$ with $0\le \eta \le 1$ and $\eta_t\ge 0$, we have   
\begin{equation}
\begin{aligned}
&\int_{B_{\rho}}   F(|D{\bf u}_\epsilon(x,\tau_2)|) \xi^2\eta  \, \d x  -\int_{B_{\rho}}   F(|D{\bf u}_\epsilon(x,\tau_1)|) \xi^2\eta  \, \d x  \\
&\quad + \frac{1}{c} \int_{Q}  \frac{\phi_\epsilon'(|D{\bf u}_\epsilon|)}{|D{\bf u}_\epsilon|}\bigg[f(|D{\bf u}_\epsilon|)|D^2\bfu_\epsilon|^2 +    \frac{f'(|D{\bf u}_\epsilon|)}{|D{\bf u}_\epsilon|}|D (|D{\bf u}_\epsilon|^2)|^2\bigg] \xi^2\eta\, \d z\\
&\le    c \int_{Q} \phi_\epsilon'(|D{\bf u}_\epsilon|)  \frac{f(|D{\bf u}_\epsilon|)^2}{f'(|D{\bf u}_\epsilon|)} |D\xi|^2\eta \, \d z + \int_{Q} F(|D{\bf u}_\epsilon|) \xi^2\eta_t\, \d z \,.
\end{aligned}
\label{Caccio estimate}
\end{equation}
Moreover, the term $\phi_\epsilon'(|D{\bf u}_\epsilon|)  \frac{f(|D{\bf u}_\epsilon|)^2}{f'(|D{\bf u}_\epsilon|)}$ in the above estimate can be replaced by $\phi_\epsilon(|D{\bf u}_\epsilon|)  f(|D{\bf u}_\epsilon|) $.
\end{lemma}

\begin{proof}
For simplicity, we shall write $\bphi=\phi_\epsilon$ and $\bfw =(w^\alpha) = \bfu_\epsilon$.

 We test  \eqref{diffsystem} with $w_{x_k}^\alpha f(|D\bfw |)\xi^2\eta$ to obtain that
\begin{equation*}
\underbrace{\int_{\tau_1}^{\tau_2} \int_{B_{\rho}} (w^\alpha_{x_k})_t w_{x_k}^\alpha f(|D\bfw |)\xi^2 \eta \, \d x\d t}_{=:I_1} + \underbrace{\int_{Q}  \bar{A}^{\alpha\beta}_{ij} w^\beta_{x_j x_k}  [w^\alpha_{x_k} f(|D\bfw |)\xi^2\eta]_{x_i} \, \d z}_{=:I_2} =0.
\end{equation*}
We estimate $I_1$ and $I_2$ separately. We have
\begin{equation*}
\begin{aligned}
I_1 & = \int_{\tau_1}^{\tau_2} \int_{B_{\rho}}  \left[\partial_t( F(|D\bfw |)\eta) -  F(|D\bfw |)\eta_t \right]\xi^2 \, \d x\,\d t \\
& =  \int_{B_{\rho}} F(|D\bfw (x,\tau_2)|) \xi^2\eta \, \d x -\int_{B_{\rho}} F(|D\bfw (x,\tau_1)|) \xi^2\eta \, \d x  -   \int_{Q}  F (|D\bfw |) \xi^2 \eta_t  \, \d z\,.
\end{aligned}
\end{equation*}
As for $I_2$, we first observe that, with \eqref{defA}, \eqref{ellipticity},
$$
\bar{A}_{ij}^{\alpha\beta} w^\beta_{x_j x_k}  w^\alpha_{x_k x_i} \ge (p-1) \frac{\bphi'(|D\bfw |)}{|D\bfw |} |D^2\bfw |^2\,,
$$
$$\begin{aligned}
&\bar{A}_{ij}^{\alpha\beta} w^\beta_{x_j x_k}  w^\alpha_{x_k}   w^\kappa_{x_l} w^\kappa_{x_l x_i} \\
&=\frac{\bphi'(|D\bfw |)}{|D\bfw |}\left\{\delta_{ij}\delta^{\alpha\beta} + \left(\frac{\bphi''(|D\bfw |)|D\bfw |}{\bphi'(|D\bfw |)}-1 \right) \frac{w_{x_i}^\alpha w^{\beta}_{x_j}}{|D\bfw |^2}\right\} w^\beta_{x_j x_k}  w^\alpha_{x_k}   w^\kappa_{x_l} w^\kappa_{x_l x_i}\\
&=\frac{\bphi'(|D\bfw |)}{|D\bfw |}\left\{\frac{|D(|D\bfw |^2)|^2}{4} + \left(\frac{\bphi''(|D\bfw |)|D\bfw |}{\bphi'(|D\bfw |)}-1 \right) \frac{\sum_{\alpha=1}^N [D w ^\alpha \cdot D(|D\bfw |^2)]^2}{4|D\bfw |^2}\right\} \\
&\ge (p-1)\frac{\bphi'(|D\bfw |)}{|D\bfw |}\frac{|D(|D\bfw |^2)|^2}{4}\,,
\end{aligned}$$
and
$$\begin{aligned}
|\bar{A}_{ij}^{\alpha\beta} w^\beta_{x_j x_k}  w^\alpha_{x_k}  |  &=\frac{\bphi'(|D\bfw |)}{|D\bfw |}\left|\left\{\delta_{ij}\delta^{\alpha\beta} + \left(\frac{\bphi''(|D\bfw |)|D\bfw |}{\bphi'(|D\bfw |)}-1 \right) \frac{w_{x_i}^\alpha w^{\beta}_{x_j}}{|D\bfw |^2}\right\} w^\beta_{x_j x_k}  w^\alpha_{x_k}\right|   \\
&\le c \frac{\bphi'(|D\bfw |)}{|D\bfw |}\big|D(|D\bfw |^2)\big| \,.
\end{aligned}$$
Inserting these inequalities into $I_2$, we have
\begin{equation}
\begin{aligned}
I_2 & = \int_{Q}   \bar{A}_{ij}^{\alpha\beta} w^\beta_{x_j x_k} \left[w^\alpha_{x_k x_i}f(|D\bfw |) + w^\alpha_{x_k}   w^\kappa_{x_l} w^\kappa_{x_l x_i}\frac{f'(|D\bfw |)}{|D\bfw |} \right] \xi^2\eta \, \d z\\
&\qquad +  \int_{Q} 2\bar A_{ij}^{\alpha\beta}   w^\beta_{x_j x_k}   w_{x_k}^\alpha f(|D\bfw |)\xi \xi_{x_i} \eta \, \d z \\
& \geq \frac{1}{c_1}\int_{Q}  \frac{\bphi'(|D\bfw |)}{|D\bfw |}\left[f(|D\bfw |)|D^2\bfw |^2 +    \frac{f'(|D\bfw |)}{|D\bfw |}|D (|D\bfw |^2)|^2\right] \xi^2 \eta\, \d z\\
&\qquad - c \int_{Q} \frac{\bphi'(|D\bfw |)}{|D\bfw |}\big|D(|D\bfw |^2)\big| f(|D\bfw |)  \xi |D\xi| \eta   \, \d z 
\end{aligned}\label{I2estimate}
\end{equation}
for some constant $c_1>0$. Applying Young's inequality to the last integrand, we obtain
$$
\begin{aligned}
I_2 &  \geq \frac{1}{2c_1}\int_{Q}  \frac{\bphi'(|D\bfw |)}{|D\bfw |}\left[ f(|D\bfw |)|D^2\bfw |^2 +    \frac{f'(|D\bfw |)}{|D\bfw |}|D (|D\bfw |^2)|^2\right] \xi^2 \eta\, \d z\\
&\qquad - c \int_{Q} \bphi'(|D\bfw |)  \frac{f(|D\bfw |)^2}{f'(|D\bfw |)} |D\xi|^2\eta   \, \d z \,,
\end{aligned}
$$
or
$$
\begin{aligned}
I_2 & \geq \frac{1}{c_1}\int_{Q}  \frac{\bphi'(|D\bfw |)}{|D\bfw |}\left[f(|D\bfw |)|D^2\bfw |^2 +  \frac{f'(|D\bfw |)}{|D\bfw |}|D (|D\bfw |^2)|^2\right] \xi^2 \eta\, \d z\\
&\qquad - c \int_{Q} \frac{\bphi'(|D\bfw |)}{|D\bfw |}|D\bfw | |D^2\bfw | f(|D\bfw |)  \xi |D\xi| \eta  \, \d z \\
&  \geq \frac{1}{2c_1}\int_{Q} \frac{\bphi'(|D\bfw |)}{|D\bfw |}\left[f(|D\bfw |)|D^2\bfw |^2 +    \frac{f'(|D\bfw |)}{|D\bfw |}|D (|D\bfw |^2)|^2\right] \xi^2 \eta\, \d z\\
&\qquad - c \int_{Q} \bphi(|D\bfw |)  f(|D\bfw |) |D\xi|^2\eta   \, \d z \,.
\end{aligned}
$$
Therefore, combining the above estimates, we get \eqref{Caccio estimate}.
\end{proof}

\begin{theorem}\label{thm:bddgrad} 
Let $\phi$ satisfy Assumption~\ref{Ass2} with \eqref{eq:hypp} and \eqref{p2q},
and  $\bfu_\epsilon$ with $\epsilon\in (0,1)$ be a weak solution to \eqref{eq:nondegenerate systemeq}. 
Then $D{\bf u}_\epsilon \in L^\infty_{\loc} (\Omega_T,\R^{Nn})$. Moreover, we have that for every $Q_{2r}(z_0)\Subset \Omega_T$,
\begin{equation}\label{supestimate0}
\|D{\bf u}_\epsilon\|_{L^\infty(Q_r(z_0),\R^{Nn})} \le  c \left(\fint_{Q_{2r}(z_0)} \phi_\varepsilon(|D{\bf u}_\epsilon|) \, \d z+1 \right)^{\frac{2}{(n+2)p-2n}}
\end{equation}
for some $c\ge 1$ depending on $n$, $N$, $p$ and $q$, and independent of $\epsilon$.
\end{theorem}


\begin{proof}
\textit{Step 1. (Setting and Caccioppoli type estimate)}
To enlighten the notation, we will write $\bphi:=\phi_\epsilon$ and $\bfw :={\bf u}_\epsilon$. Let $Q_{2r}=Q_{2r}(z_0)\Subset Q_T$. Without loss of generality, we assume that $z_0=(x_0,t_0)=(0,0)$. Let $\rho_1=s_1r$ and $\rho_2=s_2r$ with $1\le s_1<s_2\le 2$, $\xi \in C^\infty_0(B_{\rho_2})$ and $\eta\in C^\infty(\R)$ be as in \eqref{eq:cutoff1} and \eqref{eq:cutoff2}, respectively. 

Then applying Lemma~\ref{Lem:Caccio1} with $f(t)=t^\chi$ where $\chi\ge 0$, $\rho=\rho_2$, $\tau_1=-\rho_2^2$ and $\tau_2\in(-\rho_1^2,0)$, we have 
\begin{equation}
\begin{aligned}
&\sup_{-\rho_1^2<\tau<0} \int_{B_{\rho_2}}   |D\bfw (x,\tau)|^{2+\chi} \xi^2 \, \d x +  \int_{Q_{\rho_2}}  \frac{\bphi'(|D\bfw |)}{|D\bfw |}|D\bfw |^\chi |D^2\bfw |^2\eta^2 \xi^2\, \d z\\
&\le  \frac{c(2+\chi)}{(\rho_2-\rho_1)^2} \int_{Q_{\rho_2}} \left[ |D\bfw |^2  +  \bphi(|D\bfw |)\right] |D\bfw |^{\chi} \, \d z\,.
\end{aligned}
\label{Lip_pf1}
\end{equation}

\textit{Step 2. (Improving inequality)} We set
\begin{equation}
F(z) := [\bphi(|D\bfw (x,t)|)|D\bfw (x,t)|^{\chi}]^{\frac{1}{2}} \eta(t) \xi(x)\,.
\label{eq:F}
\end{equation}
Note that, in order to enlighten the notation, we will often omit the dependence of $F$, $u$, $\eta$ and $\xi$ on the respective arguments. Differentiating \eqref{eq:F} with respect to $x_i$ we then have
$$\begin{aligned}
F_{x_i}&= \tfrac{1}{2}\big[\bphi(|D\bfw |)|D\bfw |^{\chi}\big]^{-\frac{1}{2}}\big[\bphi'(|D\bfw |)|D\bfw |^{\chi}+ \chi |D\bfw |^{\chi-1}\bphi(|D\bfw |)\big]\frac{w^\alpha_{x_j} w^\alpha_{x_j x_i}}{|D\bfw |} \eta \xi\\
&\qquad  + \big[\bphi(|D\bfw |)|D\bfw |^{\chi}\big]^{\frac{1}{2}} \eta \xi_{x_i}\,,
\end{aligned}$$
whence, recalling the upper bound \eqref{eq:cutoff1} for $|D\xi|$, we obtain
$$
|DF|^2  \leq c (\chi+1)^2\frac{\bphi'(|D\bfw |)}{|D\bfw |} |D\bfw |^{\chi} |D^2\bfw |^2\eta \xi  + \frac{c}{(\rho_2-\rho_1)^2} \bphi(|D\bfw |) |D\bfw |^{\chi}\,.
$$
Therefore, combining with \eqref{Lip_pf1}  we have 
\begin{equation}\label{Lip_pf2}\begin{aligned}
&\sup_{-\rho_1^2<\tau<0} \int_{B_{\rho_2}}   |D\bfw (x,\tau)|^{2+\chi} \xi^2 \, \d x +  \int_{Q_{\rho_2}}  |DF|^2\, \d z\\
&\le \frac{c(1+\chi)^{3}}{(\rho_2-\rho_1)^2} \int_{Q_{\rho_2}}   \left[|D\bfw |^2+\bphi(|D\bfw |)\right] |D\bfw |^{\chi} \, \d z \, .
\end{aligned}\end{equation}
Now, applying H\"older's inequality, the Sobolev inequality to function $F\in W^{1,2}_0(B_r)$ and using \eqref{Lip_pf2}, we can write

\begin{equation}\label{Lip_pf3}\begin{aligned}
&\int_{Q_{\rho_1}} |D\bfw |^{\frac{4}{n}+\chi(\frac{2}{n}+1)}\bphi(|D\bfw |) \,\d z \\
&\le  \int_{-\rho_1^2}^0 \left(\int_{B_{\rho_1}} |D\bfw |^{2+\chi}  \,\d x\right)^{\frac{2}{n}} \left(\int_{B_{\rho_1}}[ \bphi(|D\bfw |) |D\bfw |^{\chi}]^{\frac{n}{n-2}} \,\d x\right)^{\frac{n-2}{n}} \,\d t\\
& \le  \left(\sup_{-\rho_1^2<\tau<0}\int_{B_{\rho_1}} |D\bfw (x,\tau)|^{2+\chi} \,\d x\right)^{\frac{2}{n}} \int_{-\rho_1^2}^0  \left(\int_{B_{\rho_2}}|F|^{2^*} \,\d x\right)^{\frac{2}{2^*}} \,\d t\\
& \le  \left(\sup_{-\rho_1^2<\tau<0}\int_{B_{\rho_1}} |D\bfw (x,\tau)|^{2+\chi} \,\d x\right)^{\frac{2}{n}} \int_{Q_{\rho_2}(z_0)}|DF|^{2} \,\d z\\
&\le c \left(\frac{(1+\chi)^{3}r^2}{(\rho_2-\rho_1)^2} \int_{Q_{\rho_2}}   \left(|D\bfw |^2+\bphi(|D\bfw |)\right) |D\bfw |^{\chi} \, \d z\right)^{1+\frac{2}{n}} .
\end{aligned}\end{equation}
By Lemma~\ref{Lem:seconddifferentiability2}, 
we have that $|D\bfw |\in L^{p(n+2)/n}(Q_{2r})$. Note that  by under the assumption \eqref{eq:hypp} on $p$, $p(n+2)/n>2$.
Therefore, it holds that
\begin{equation}
\int_{Q_{2r}}\big[ |D\bfw |^2+ \bphi(|D\bfw |) \big] \, \d z<\infty\,.
\label{eq:finitebound}
\end{equation}
Since $2<\frac{4}{n}+p$ again by \eqref{eq:hypp}, setting
\begin{equation}\label{chi1}
\chi_1:=\frac{4}{n}+p-2\, ,
\end{equation}
we may improve estimate \eqref{Lip_pf3} as
\begin{equation}
\begin{aligned}
\int_{Q_{\rho_1}}& (|D\bfw |^2+\bphi(|D\bfw |))|D\bfw |^{\chi_1+\chi(1+\frac{2}{n})} \, \frac{\d z}{|Q_{2r}|}\\
&\le c \left(\frac{(1+\chi)^{3}r^2}{(\rho_2-\rho_1)^2} \int_{Q_{\rho_2}} \left[(|D\bfw |^2+\bphi(|D\bfw |)) |D\bfw |^{\chi} +1\right]\, \frac{\d z}{|Q_{2r}|}\right)^{1+\frac{2}{n}}\, .
\end{aligned}
\label{Lip_pf4}
\end{equation}

\textit{Step 3. (Iteration)} 
Let $s_1,s_2$ such that $1\le s_1< s_2 \le 2$ be fixed. For $m=0,1,2,\dots$, we set
\[
\chi_0:=0\quad \text{and}  \quad \chi_{m} := \chi_1  +\theta \chi_{m-1} \ \ (m\ge 1) ,
\quad \text{where} \ \ \theta:=1+\frac{2}{n},
\]
and
\[
J_{m}:= \int_{Q_{r_m}}  \left[(|D\bfw |^2+\bphi(|D\bfw |)) |D\bfw |^{\chi_m}  +1\right]\, \frac{\d z}{|Q_r|}, 
\quad \text{where}\ \ r_m:=(s_1+2^{-m}(s_2-s_1)) r\,.
\]
Note that $\chi_m=(\theta^m-1)\frac{\chi_1n}{2}$.
Then we have from \eqref{Lip_pf4} that 
\[
J_{m}  \le \frac{c 4^{\theta m}\theta^{3\theta m}}{(s_2-s_1)^{2\theta}} J_{m-1}^{\theta}  \le \frac{c_0^m}{(s_2-s_1)^{2\theta}} J_{m-1}^\theta , \quad m=1,2,\dots,
\]
where $c_0\ge 1$ depends on $n$, $N$, $p$ and $q$. Hence, for $m\ge 2$,
\[\begin{aligned}
J_{m} & \le \frac{c_0^m}{(s_2-s_1)^{2\theta}}  \left(\frac{c_0^{m-1}}{(s_2-s_1)^{2\theta}} J_{m-2}^\theta  \right)^{\theta}  \\ 
&\le \frac{c_0^{m+(m-1)\theta}}{(s_2-s_1)^{2(\theta+\theta^2)}}J_{m-2}^{\theta^2} \\
& \le \cdots \cdots \cdots \cdots \cdots \cdots \\ 
& \le \frac{ c_0^{\sum_{k=1}^m(m-k+1)\theta^{k-1}}}{(s_2-s_1)^{2\sum_{k=1}^{m} \theta^k} } J_0^{\theta^m}  \le \left(\frac{c_1}{(s_2-s_1)^{\beta_0} }J_0\right)^{\theta^m}
\end{aligned}\]
for some large $c_1,c_2,\beta_0>1$ depending on $n$, $N$, $p$ and $q$. Consequently, setting 
$$
\d\mu(z) := \left[|D\bfw (z)|^2+\bphi(|D\bfw (z)|)\right]\frac{\d z}{|Q_{2r}|},
$$
we have 
$$\begin{aligned}
\|D\bfw \|_{L^\infty(Q_{s_1r},\R^{Nn})}& \le \||D\bfw |\|_{L^\infty(Q_{s_1r};\d\mu)} 
=\lim_{m\to \infty} \left(\int_{Q_{r}} |D\bfw |^{\chi_m}  \, \d\mu \right)^{\frac{1}{\chi_m}} \le \limsup_{m\to \infty}\, J_m^{\frac{1}{\chi_m}} \\
 & \le \limsup_{m\to \infty}\, \left(\frac{c_1}{(s_2-s_1)^{\theta_0}} J_0 \right)^{\frac{\theta^i}{\chi_m}}  \le \frac{c}{(s_2-s_1)^{\theta_1}}  J_0^{\frac{2}{n\chi_1}},
\end{aligned}$$
where we used also the fact that $\chi_m=(\theta^m-1)\frac{\chi_1n}{2}$. Therefore, we have
\begin{equation}\label{Lip_pf6}
\|D\bfw \|_{L^\infty(Q_{s_1r},\R^{Nn})} \le \frac{c}{(s_2-s_1)^{\theta_1}} \left(\fint_{Q_{s_2r}} \left[|D\bfw |^2+ \bphi(|D\bfw |)\right]  \, \d z +1\right)^{\frac{2}{n\chi_1}}.
\end{equation}
By virtue of \eqref{eq:finitebound}, this shows that $D\bfw \in L^\infty_{\loc}(Q_R;\R^{Nn})$.

\textit{Step 4. (Interpolation)} 
Now we get rid of the term $|D\bfw |^2$ in the integrand in \eqref{Lip_pf6} by using an interpolation argument. Since  $D\bfw \in L^\infty_{\loc}(Q_R;\R^{Nn})$ and $\frac{2(2-p)}{n\chi_1} <1$ by \eqref{eq:hypp}, using Young's inequality we have that for every $1\le s_1< s_2 \le 2$,
$$\begin{aligned}
\|D\bfw \|_{L^\infty(Q_{s_1r},\R^{Nn})} &\le \frac{c\|D\bfw \|_{L^\infty(Q_{s_2r},,\R^{Nn})}^{\frac{2(2-p)}{n\chi_1}} }{(s_2-s_1)^{\theta_1}} \left(\fint_{Q_{s_2r}} |D\bfw |^p  \, \d z\right)^{\frac{2}{n\chi_1}}\\
&\qquad + \frac{c}{(s_2-s_1)^{\theta_1}}  \left( \fint_{Q_{s_2r}} \bphi(|D\bfw |) \, \d z+1 \right)^{\frac{2}{n\chi_1}}\\
&\le \tfrac12 \|D\bfw \|_{L^\infty(Q_{s_2r},\R^{Nn})}+   \frac{c}{(s_2-s_1)^{\frac{\theta_1n\chi_1}{n\chi_1-2(2-p)}}}  \left(\fint_{Q_{s_2r}} |D\bfw |^p  \, \d z\right)^{\frac{2}{n\chi_1-2(2-p)}}\\
&\qquad +  \frac{c}{(s_2-s_1)^{\theta_1}}  \left( \fint_{Q_{s_2r}} \bphi(|D\bfw |) \, \d z+1 \right)^{\frac{2}{n\chi_1}}\\
&\le \tfrac12 \|D\bfw \|_{L^\infty(Q_{s_2r},\R^{Nn})}+   \frac{c}{(s_2-s_1)^{\theta_2}}  \left(\fint_{Q_{2r}} \bphi(|D\bfw |)+1  \, \d z\right)^{\frac{2}{n\chi_1-2(2-p)}}
\end{aligned}$$
Therefore, we can remove the first term on the right hand side (cfr. \cite[Lemma 6.1]{Giusti_book}). Finally, recalling  \eqref{chi1}, 
we obtain \eqref{supestimate0}.
\end{proof}

From the previous theorem and Lemma~\ref{Lem:approximation}, 
we obtain the boundedness of the gradient of a weak solution to \eqref{eq:system}.

\begin{corollary}\label{cor:bddgrad}
Let $\phi$ satisfy Assumption~\ref{Ass2} with \eqref{eq:hypp}, 
and $\bfu$  be a weak solution to \eqref{eq:system}.
Then $D{\bf u} \in L^\infty_{\loc} (\Omega_T,\R^{Nn})$. Moreover, we have that for every $Q_{2r}\Subset \Omega_T$,
\begin{equation}\label{supestimate}
\|D{\bf u}\|_{L^\infty(Q_r(z_0),\R^{Nn})} \le c \left(\fint_{Q_{2r}(z_0)} \phi(|D{\bf u}|) \, \d z +1\right)^{\frac{2}{(n+2)p-2n}}
\end{equation}
for some $c=c(n,N,p,q)\ge 1$.
\end{corollary}

\begin{remark}
When $\phi (t)=t^p$ with $\frac{2n}{n+2}<p<2$, the estimate \eqref{supestimate}  is exactly the same as \cite[eq. (5.10)]{DiB_book}. 
\end{remark}


\section{H\"older continuity of $D{\bf u}$ revisited}

We prove local H\"older continuity for the gradient of weak solution to \eqref{eq:system},
 where $\phi$ satisfies Assumption~\ref{Ass3}.  We remark that the result was already obtained by Lieberman in \cite{Lie06} by assuming the local boundedness of $D\bfu$. In this section, we take advantage of the results of Section~\ref{sec:locboundgradient} and 
revisit his $C^{1,\alpha}$-regularity's proof, according to the setting of our paper. We also note that Lieberman's proof parallels the one given by Di Benedetto and Friedman \cite{DiBeFried84,DiBeFried85}, using a measure theoretic approach. In addition, we are adapting the geometry of the cylinders accordingly, due to the growth conditions of the operator.

 We define the intrinsic parabolic cylinder associated with an $N$-function $\phi$ as 
$$
Q^\lambda_r(x_0,t_0) := B_r(x_0) \times (t_0-r^2/\phi''(\lambda), t_0].
$$
where $\lambda,r >0$, {and oscillation of a function $f:U\to \R^m$ by
$$
\underset{U}{\mathrm{osc}} \, f := \sup_{x,y\in U} |f(x)-f(y)|\,.
$$}

Now, we state the main result of  this section.
 \begin{theorem}\label{thm:holdergrad_Lie}
 Let $\phi:[0,\infty)\to[0,\infty)$ satisfy Assumption~\ref{Ass3}, and let $\bfu$ be a weak solution to the parabolic system \eqref{eq:system}.
  and $Q_R(z_0)\Subset \Omega_T$.  If $D{\bf u} \in L^\infty_\loc(Q_R(z_0);\R^{Nn})$, then $D{\bf u}\in C^{0,\alpha}_{loc}(Q_R(z_0);\R^{Nn})$ for some $\alpha\in(0,1)$  depending on $n,N,p,q,\gamma_1$ and $c_h$. Moreover, any $Q_R(z_0)\Subset \Omega_T$, $r\in(0,R)$ and $\lambda \ge \|D{\bf u}\|_{ L^\infty(Q_R(z_0);\R^{Nn})}$, we have 
 $$
 \underset{Q_r(z_0)}{\mathrm{osc}} \, D{\bf u} \le c \lambda \left(\max\left\{\phi''(\lambda)^{\frac{1}{2}},\phi''(\lambda)^{-\frac{1}{2}}\right\} \frac{r}{R}\right)^{\alpha}
  $$
for some $c=c(n,N,p,q,\gamma_1,c_h)>0$.
 \end{theorem}  
 
  This result can be obtained by approximation via Lemma~\ref{Lem:approximation}, once we obtain the analog of Theorem~\ref{thm:holdergrad_Lie} for the gradients $D\bfu_\epsilon$ of weak solutions $\bfu_\epsilon$ to the approximating nondegenerate parabolic system  \eqref{eq:nondegenerate systemeq}, where $\epsilon\in(0,1]$. This will be a consequence of the following two propositions (cfr. \cite[Propositions 1.3 and 1.4]{Lie06}) for $\bfu_\epsilon$. Note that all estimates in this section are independent of $\epsilon\in(0,1]$. Thus, for simplicity, we shall write $\bfu = \bfu_\epsilon$ and $\phi=\phi_\epsilon$. 
  
  The first proposition provides an estimate on the oscillation of $D{\bf u}$ on subcylinders when $|D{\bf u}|$ is small on a small portion of the main cylinder.

\begin{proposition} \label{prop:nondegenerate}
Let $\phi:[0,\infty)\to[0,\infty)$ satisfy Assumption~\ref{Ass3}, and let $\bfu$ be a weak solution to  \eqref{eq:nondegenerate systemeq}. Suppose that for some $\lambda,R>0$, $Q^{\lambda}_{R}(z_0) \Subset \Omega_T$ and 
\begin{equation}\label{Dubound}
|D{\bf u}| \le \lambda \quad \text{in }\ Q^{\lambda}_{R}(z_0).
\end{equation}
There exist $\sigma\in (0,2^{-(n+1)})$ and $C \ge 1$ depending on $n,N,p,q,\gamma_1$ and $c_h$  such that if
\begin{equation}\label{prop:nondegenerate_ass}
\big|\{|D{\bf u}| \le (1-\sigma) \lambda\} \cap Q^{\lambda}_{R}(z_0)\big| \le \sigma |Q^{\lambda}_{R}(z_0)|\,,
\end{equation}
then 
\begin{equation}
\underset{Q^{\lambda}_r(z_0)}{\mathrm{osc}} \, D{\bf u} \le C \left(\frac{r}{R}\right)^{\frac{3}{4}}  \underset{Q^{\lambda}_R(z_0)}{\mathrm{osc}} \, D{\bf u}
\label{nondegenerate_oscillation}
\end{equation}
for all $r\in (0,R)$.
\end{proposition}

If \eqref{prop:nondegenerate_ass} fails, the following proposition gives an estimate of how $| D{\bf u} |$ decreases.

\begin{proposition} \label{prop:degenerate}
Let $\phi:[0,\infty)\to[0,\infty)$ satisfy Assumption~\ref{Ass3}, and let $\bfu$ be a weak solution to \eqref{eq:nondegenerate systemeq}. Suppose that for some $\lambda,R>0$, $Q^{\lambda}_{R}(z_0) \Subset \Omega_T$ and \eqref{Dubound} holds.
For any $\sigma\in(0,\frac{1}{2})$, there exists $\nu\in (0,1)$ depending on $n,N,p,q$ and $\sigma$ such that if
\begin{equation}\label{prop:degenerate_ass}
|\{|D{\bf u}| \le (1-\sigma) \lambda\} \cap Q^{\lambda}_{R}(z_0)| > \sigma |Q^{\lambda}_{R}(z_0)|\, ,
\end{equation}
then
\begin{equation}
| D{\bf u} | \le \nu \lambda  \quad \text{in }\ Q^{\lambda}_{\sigma R/2}(z_0)\,.
\label{degenerate_decay}\end{equation}
\end{proposition}

Proposition ~\ref{prop:nondegenerate} and Proposition~\ref{prop:degenerate} will be proved in Subsection~\ref{sec:proofnondeg} and Subsection~\ref{sec:proofdeg}, respectively.  In the remaining subsections, we always assume that $\phi:[0,\infty)\to[0,\infty)$ satisfies  Assumption~\ref{Ass3}, $\bfu$ is a weak solution to \eqref{eq:nondegenerate systemeq}, and  \eqref{Dubound}  holds for some $Q^\lambda_R(z_0)\Subset \Omega_T$ with $\lambda,R>0$. In addition, without loss of generality, we further assume that assume \eqref{p2q} and $z_0=(x_0,t_0)=(0,0)=0$, and write $Q^\lambda_r=Q^\lambda_r(0)$ for all $r\in(0,R]$.

\subsection{Proof of Proposition~\ref{prop:nondegenerate}} \label{sec:proofnondeg}

Before starting the proof, we recall the following weighted version of Poincar\'e's inequality, which is quite elementary and can be deduced, for instance, from \cite[Theorem 7]{DieEtt08}.
\begin{lemma}\label{lem:wpoincare}
Suppose $f\in W^{1,p}(B_R; \R^m)$ and $\xi \in L^1(B_R)$ is nonnegative and satisfies $\|\xi\|_{L^1(B_R)}=1$.  Then we have
$$
\int_{B_R} \left|\frac{f-\langle f \rangle_\xi}{R}\right|^p \, \d x \le c \int_{B_R} |Df|^p \, \d x 
$$
for some $c=c(n,m,p)>0$, where   $\langle f \rangle_\xi= \int f \xi \, \d x$.
\end{lemma}

We first derive a higher integrability result for $D{\bf u}$ (cfr. \cite[Lemma~4.2]{Lie06}).

\begin{lemma}\label{Lem:high}
Let ${\bf P}\in\R^{Nn}$ satisfying $\frac{\lambda}{2}\le |{\bf P}| \le \lambda$. There exist $\gamma,c>0$ depending on $n,N,p$ and $q$ such that
$$
\left(\fint_{Q^\lambda_{R/2}} |D{\bf u} -{\bf P}|^{2(1+\gamma)} \, \d z \right)^{\frac{1}{1+\gamma}} \le c \fint_{Q^\lambda_{R}} |D{\bf u} -{\bf P}|^{2} \, \d z \,.
$$
\end{lemma}

\begin{proof}
Fix any  $Q^\lambda_{2r}(z_1)\subset Q^\lambda_R$ with $z_1=(x_1,t_1)\in  Q^\lambda_R$ and $r\le r_1<r_2\le 2r$. We further set
$r_3=\frac{r_1+r_2}{2}$, $r_4=\frac{r_1+3r_2}{4}$ and $t_2=t_1-\frac{r_2^2}{\phi''(\lambda)}$.  Note that $r_1 < r_3 < r_4 < r_2$.
We consider two cut-off functions. Let $\xi_0\in C^\infty_0 (B_{r_2}(x_1))$ satisfying $0\le \xi_0\le 1$,  $\xi_0 =1$ in $B_{r_4}(x_1)$ and $|D\xi_0|\le 8/(r_2-r_1)$, and set $\xi = \|\xi_0\|_{L^1(B_{r_2})}^{-1} \xi_0$.  Note that $|B_r|\le |B_{r_4}| \le \|\xi_0\|_{L^1(B_{r_2})} \le |B_{r_2}|\le 2^n|B_r|$ and $\|\xi\|_{L^1(B_{r_2})}=1$. Next, let $\zeta \in C^\infty(Q^\lambda_{r_2}(z_1))$ such that $\zeta =0$ on $\partial_{\mathrm{p}}Q^\lambda_{r_2}(z_1)$, $\zeta =1$ in $Q^\lambda_{r_1}(z_1)$,     
$$
|D\zeta|^2 +|D^2 \zeta| \le \frac{c}{(r_2-r_1)^{2}} \quad \text{and}\quad  0\le \zeta_t  \le \frac{c\phi''(\lambda)}{(r_2-r_1)^{2}}.  
$$
Finally, define
$$
\bfw (z):=\bfu(z) - {\bf P} (x-x_1),
\quad {\bf W}(t):=  \int_{B_{r_2}(x_1)} \bfw (x,t) \xi(x) \, \d x\,,
$$
$$
\text{and}\quad  \tilde \bfw  := \bfw  -\bfw _0 \ \ \text{with }\ \bfw _0 := \frac{\phi''(\lambda)}{r^2} \int_{Q^\lambda_{r_2}(z_1)} \bfw  \xi \, \d z = \fint_{t_2}^{t_1}  {\bf W}(t) \, \d t\,.
$$

We take $ \zeta^\chi \tilde \bfw $ with $\chi \ge 2 $ as a test function in the weak form of  \eqref{eq:system} to get,
for every $\tau\in I^\lambda_{r_2}(t_1)$,
\[
\int_{t_2}^\tau \int_{B_{r_2}(z_1)} (\bfw _t \cdot \tilde \bfw ) \zeta^\chi  \, \d x\, \d t +  \int_{t_2}^\tau \int_{B_{r_2}(z_1)} \left(\bA(D{\bf u}) - \bA({\bf P})\right)  :  D( \zeta^\chi \tilde \bfw ) \, \d x\, \d t=0 ,
\]
which yields
\[\begin{aligned}
&\sup_{\tau \in I_{r_2}^\lambda(t_1)} \int_{B_{r_2}(x_1)} |\tilde \bfw |^2 \zeta ^\chi \, \d x + \int_{Q^\lambda_{r_2}(z_1)} \frac{\phi'(|D{\bf u}|+|{\bf P}|)}{|D{\bf u}|+|{\bf P}|} |D\bfw |^2 \zeta^\chi \, \d z \\
& \le  \frac{c}{r_2-r_1}  \int_{Q^\lambda_{r_2}(z_1)}  \frac{\phi'(|D{\bf u}|+|{\bf P}|)}{|D{\bf u}|+|{\bf P}|} |D\bfw | |\tilde \bfw | \zeta^{\chi-1} \, \d z  + c  \frac{\phi''(\lambda)}{(r_2-r_1)^2}\int_{Q^\lambda_{r_2}(z_1)} |\tilde \bfw |^2 \zeta^{\chi-1} \, \d z.
\end{aligned}\]
Here we used \eqref{PhiPQ}. Set
\[
S_\chi:= \sup_{\tau \in I_{r_2}^\lambda(t_1)} \int_{B_{r_2}(x_1)} |\tilde \bfw |^2 \zeta ^\chi \, \d x \,.
\]
Applying Young's inequality to the integrand of the first integral on the right hand side and using the facts that $\chi\ge 2$, $|D{\bf u}|\le \lambda$ and $\frac{\lambda}{2}\le |{\bf P}|\le \lambda$, we have 
$$
S_\chi + \phi''(\lambda)\int_{Q^\lambda_{r_2}(z_1)}  |D\bfw |^2 \zeta^\chi \, \d z  \le   c  \frac{\phi''(\lambda)}{(r_2-r_1)^2}\int_{Q^\lambda_{r_2}(z_1)} |\tilde \bfw |^2 \zeta^{\chi-2} \, \d z
$$
Then, when $\chi=4$ and $\chi=2$, we have
\begin{equation}\label{Lem:reverse_pf1}
\int_{Q^\lambda_{r_1}(z_1)}  |D\bfw |^2 \, \d z  \le   c  \frac{S_2^{\frac{2}{n+2}}}{(r_2-r_1)^2}\int_{t_2}^{t_1}\left(\int_{B_{r_2}(x_1)} |\tilde \bfw (x,t)|^2 \, \d x\right)^{\frac{n}{n+2}}\, \d t\,,
\end{equation}
and 
$$
S_2  \le   c  \frac{\phi''(\lambda)}{(r_2-r_1)^2}\int_{Q^\lambda_{r_2}(z_1)} |\tilde \bfw |^2  \, \d z \le   c  \frac{\phi''(\lambda)}{(r_2-r_1)^2}\int_{Q^\lambda_{r_2}(z_1)}\Big[ |\bfw  - {\bf W}(t)|^2 + |\bfw _0-{\bf W}(t)|^2\Big]  \, \d z\,.
$$
By Poincar\'e's inequality with the weight $\xi$ (Lemma~\ref{lem:wpoincare}) we see that  for every $t\in I^\lambda_{r_2}(t_1)$
$$
\int_{B_{r_2}} |\bfw (x,t) -{\bf W}(t)|^2\, \d x \le cr^2 \int_{B_{r_2}} |D\bfw (x,t)|^2\, \d x.
$$
Moreover, by testing \eqref{eq:nondegenerate systemeq} with $\zeta=(\xi,\cdots,\xi)$ and using $|D\xi| \le \frac{c}{r^n(r_2-r_1)}$ and \eqref{PhiPQ} with $|D{\bf u}|\le \lambda \le 2|{\bf P}|\le 2\lambda$, we see that for  every $t_2<\tau<\tau'<t_1$,
$$\begin{aligned}
|{\bf W}(\tau)- {\bf W}(\tau')| & =  \left|\int^{\tau'}_{\tau}  \int_{B_{r_2}(x_1)}  \left(\bA(D{\bf u}) - \bA({\bf P})\right) :  D\zeta \, \d x\, \d t \right|  \\
&\le c\frac{r^{2}}{(r_2-r_1)}\fint_{Q_{r_2}^\lambda(z_1)}  |D\bfw |    \, \d z \,,
\end{aligned}$$
hence for every $t \in (t_2,t_1)$,
\begin{equation}\label{Lem:reverse_pf4}\begin{aligned}
|\bfw _0-{\bf W}(t)| \le \sup_{t_2<\tau<\tau'<t_1} |{\bf W}(\tau)-{\bf W}(\tau')|  \le  c\frac{r^2}{(r_2-r_1)}\fint_{Q_{r_2}^\lambda(z_1)}  |D\bfw |  \, \d z \,.
\end{aligned}\end{equation}
Therefore, combining the above estimates together with H\"older's inequality, we have 
\begin{equation}\label{Lem:reverse_pf2}
S_2   \le   c  \frac{\phi''(\lambda)}{(r_2-r_1)^2}\left(\frac{r^4}{(r_2-r_1)^2}+r^2\right)\int_{Q^\lambda_{r_2}(z_1)}|D\bfw |^2   \, \d z \le c   \frac{\phi''(\lambda)r^4}{(r_2-r_1)^4}\int_{Q^\lambda_{r_2}(z_1)}|D\bfw |^2   \, \d z \,.
\end{equation}
Moreover,  by   \eqref{Lem:reverse_pf4}, a weighted Sobolev-Poincar\'e type inequality and H\"older's inequality, we also see that for every $t\in (t_2,t_1)$
\begin{equation}\label{Lem:reverse_pf3}
\begin{aligned}
&\int_{B_{r_2}(x_1)} |\tilde \bfw (x,t)|^2 \, \d x   \le  c  \int_{B_{r_2}(x_1)} |\bfw (x,t)- {\bf W}(t)|^2 \, \d x  + c r^n |{\bf W}(t)-\bfw _0|^2  \\
& \le  c  \left(\int_{B_{r_2}(x_1)} |D\bfw (x,t)|^\frac{2n}{n+2} \, \d x\right)^{\frac{n+2}{n}} +  c\frac{r^{n+4}}{(r_2-r_1)^2}\left(\fint_{Q_{r_2}^\lambda(z_1)}  |D\bfw |^{\frac{2n}{n+2}}  \, \d z \right)^\frac{n+2}{n}.
\end{aligned}\end{equation}
Therefore, inserting \eqref{Lem:reverse_pf2} and  \eqref{Lem:reverse_pf3} into \eqref{Lem:reverse_pf1} and using Young's inequality, we have that for every $0<r_1<r_2 \le r$ ,
$$\begin{aligned}
&\int_{Q^\lambda_{r_1}(z_1)}  |D\bfw |^2 \, \d z  
\le   \frac{c}{(r_2-r_1)^2} \left(  \frac{\phi''(\lambda)r^4}{(r_2-r_1)^4}\int_{Q^\lambda_{r_2}(z_1)}|D\bfw |^2   \, \d z \right)^{\frac{2}{n+2}}\\
&\qquad \qquad\qquad \qquad \times \left(\frac{r}{r_2-r_1}\right)^{\frac{2n}{n+2}} \int_{Q_{r_2}^\lambda(z_1)}  |D\bfw |^{\frac{2n}{n+2}}  \, \d z  \\
&\le   c \left(\frac{r}{r_2-r_1}\right)^{\frac{4(n+3)}{n+2}} \frac{\phi''(\lambda)^{\frac{2}{n+2}}}{r^2} \left( \int_{Q^\lambda_{r_2}(z_1)}|D\bfw |^2   \, \d z \right)^{\frac{2}{n+2}}  \int_{Q_{2r}^\lambda(z_1)}  |D\bfw |^{\frac{2n}{n+2}}  \, \d z \\
&\le   \frac{1}{2} \left(\frac{r}{r_2-r_1}\right)^{2(n+3)} \int_{Q_{r_2}^\lambda(z_1)}  |D\bfw |^2   \, \d z + c \left(\frac{\phi''(\lambda)}{r^{n+2}}\right)^{\frac{2}{n}}  \left( \int_{Q^{\lambda}_{2r}(z_1)} |D \bfw |^{\frac{2n}{n+2}} \, \d z \right)^{\frac{n+2}{n}}\,.
\end{aligned}$$
Then we can remove the first term on the right hand side, and have
$$
\fint_{Q^\lambda_{r}(z_1)}  |D\bfw |^2 \, \d z   
\le  c  \left( \fint_{Q^{\lambda}_{2r}(z_1)} |D \bfw |^{\frac{2n}{n+2}} \, \d z \right)^{\frac{n+2}{n}},
$$
for every $Q^\lambda_r(z_1)\subset Q^\lambda_R$.
Finally, applying Gehring's Lemma (cfr. \cite[Theorem 6.6]{Giusti_book}), we obtain the conclusion.   
\end{proof}

Next we obtain an $L^2$-comparison estimate between $D{\bf u}$ and the gradient of a weak solution to a corresponding linear system with constant coefficients (cfr. \cite[Lemma~4.3]{Lie06}). We recall \eqref{ass3_holder} and the definition of $A_{ij}^{\alpha\beta}$ in \eqref{defA}, so that 
\begin{equation}\label{ass:holder}
\sum_{i,j,\alpha,\beta}|A_{ij}^{\alpha\beta}({\bf Q})-A_{ij}^{\alpha\beta}({\bf P}) | \le c_h \left(\frac{|{\bf Q}-{\bf P}|}{|{\bf Q}|}\right)^{\gamma_1} \phi''(|{\bf Q}|) \quad \text{for }\ |{\bf Q}-{\bf P}| \le \frac{1}{2} |{\bf Q}|.
\end{equation}

\begin{lemma}\label{Lem:comparison}
Let ${\bf P} = (P^\alpha_i)\in\R^{Nn}$ satisfying $\frac{\lambda}{2}\le |{\bf P}| \le \lambda$, and  ${\bf v}=(v^\alpha_i)$ be the weak solution to 
\begin{equation}\label{system_constant}
\begin{cases}
& (v^\alpha)_t - ( A_{ij}^{\alpha\beta}({\bf P}) v^\beta_{x_j} )_{x_i} =0 \quad \text{in}\ \ Q^\lambda_{R/2}, \quad \alpha=1,2,\dots,N, \\
& {\bf v} = {\bf u} \ \ \text{on} \ \ \partial_p Q^\lambda_{R/2} \,.
\end{cases}
\end{equation}
 Then for every $\epsilon_0 \in (0,\frac{1}{2})$,
$$
\fint_{Q^{\lambda}_{R/2}} |D{\bf u}-D{\bf v}|^2\, \d z \le  c \left[\epsilon_0^{2\gamma_1} +  \epsilon_0^{-2\gamma}\lambda^{-2\gamma}  \left(\fint_{Q^{\lambda}_{R}} |D{\bf u}-{\bf P}|^{2} \, \d z\right)^{\gamma}\right]  \fint_{Q^{\lambda}_{R}} |D{\bf u}-{\bf P}|^2 \, \d z
$$
for some $c=c(n,N,p,q,\gamma_1,c_h)>0$, where $\gamma>0$ is from Lemma~\ref{Lem:high}.
\end{lemma}

\begin{proof} 
Observe the $\bfu$ satisfies 
$$
(u^\alpha)_t - ( A_{ij}^{\alpha\beta}({\bf P}) u^\beta_{x_j} )_{x_i} =   - \Big( A_{ij}^{\alpha\beta}({\bf P}) (u^\beta_{x_j}- P^\beta_{j})+[\bA({\bf P})]^{\alpha}_{i} - [\bA(D{\bf u})]^{\alpha}_{i} \Big)_{x_i}  =: -(H_i^\alpha)_{x_i}
$$
 for every $\alpha=1,2,\dots,N$. By taking $u^\alpha-v^\alpha$ as a test function in the weak form of the above two equations, we have 
$$\begin{aligned}
\frac{1}{2} \int_{B_{R/2}}|{\bf u}(x,0)-{\bf v}(x,0)|^2 \, \d x  + 
\int_{Q^\lambda_{R/2}} A_{ij}^{\alpha\beta}({\bf P}) (u^\beta_{x_j} - v^\beta_{x_j} ) & (u^\alpha_{x_i} - v^\alpha_{x_i} ) \,\d z\\ 
&= \int_{Q^\lambda_{R/2}} H_i^\alpha (u^\alpha_{x_i} - v^\alpha_{x_i} ) \, \d z\,,
\end{aligned}$$
and
$$
\fint_{Q^{\lambda}_{R/2}} \phi''(|{\bf P}|)|D{\bf u}-D{\bf v}|^2\, \d z \le \frac{c}{\phi''(|{\bf P}|)} \fint_{Q^{\lambda}_{R/2}} |{\bf H}|^2 \, \d z \,,
$$
where ${\bf H}=( H^\alpha_i)$.
We note that 
$$\begin{aligned}
H_i^\alpha  & =   A_{ij}^{\alpha\beta}({\bf P}) (u^\beta_{x_j}- P^\beta_{j})+[\bA({\bf P})]^{\alpha}_{i} - [\bA(D{\bf u})]_{i}^{\alpha} \\
& =   A_{ij}^{\alpha\beta}({\bf P}) (u^\beta_{x_j}- P^\beta_{j}) - \left(\int_0^1  A_{ij}^{\alpha\beta}(\tau D{\bf u} + (1-\tau){\bf P}) \, \d \tau\right)    (u^\beta_{x_j}- P_j^{\beta}) \,.
\end{aligned}$$
If $|D{\bf u}-{\bf P}|\le  \epsilon_0 |{\bf P}|$, by \eqref{ass:holder}
$$
|{\bf H}| \le c \left(\int_0^1 \left[\frac{|\tau (D{\bf u}-{\bf P})|}{|{\bf P}|}\right]^{\gamma_1}\, \d \tau\right) \phi''(|{\bf P}|) |D{\bf u}-{\bf P}| \le c \epsilon_0^{\gamma_1} \phi''(|{\bf P}|) |D{\bf u}-{\bf P}| \, .
$$  
If $|D{\bf u}-{\bf P}| >   \epsilon_0 |{\bf P}|$, then by \eqref{ellipticity} and \eqref{Dubound}
$$
|{\bf H}| \le c \phi''(|{\bf P}|) |D{\bf u}-{\bf P}|  \le c  \epsilon_0^{-\gamma} \phi''(|{\bf P}|)  |{\bf P}|^{-\gamma} |D{\bf u} -{\bf P}|^{1+\gamma}\,,
$$
where $\gamma$ is from Lemma~\ref{Lem:high}. Combining the above results we have 
$$
\fint_{Q^{\lambda}_{R/2}} |D{\bf u}-D{\bf v}|^2\, \d z \le c \left(\epsilon_0^{2\gamma_1} \fint_{Q^{\lambda}_{R/2}} |D{\bf u}-{\bf P}|^2 \, \d z + c \epsilon_0^{-2\gamma}\lambda^{-2\gamma}  \fint_{Q^{\lambda}_{R/2}} |D{\bf u}-{\bf P}|^{2(1+\gamma)} \, \d z\right)\, .
$$
Therefore, applying Lemma~\ref{Lem:high}, we obtain the desired estimate.
\end{proof}

Set $\bar {\bf v}(x,t) := {\bf v}\left(Rx, \frac{R^2}{\phi''(\lambda) }t\right)$, where ${\bf v}$ is a weak solution to \eqref{system_constant} with ${\bf P}\in\R^{Nn}$ satisfying $\frac{\lambda}{2}\le |{\bf P}| \le \lambda$. Then $\bar {\bf v}$ is a weak solution to
$$
\bar v^\alpha_t - \left( \frac{A_{ij}^{\alpha\beta}({\bf P})}{\phi''(\lambda)} \bar v^\beta_{x_j}\right)_{x_i} =0 \quad \text{in }\ Q_{1/2}, \quad \alpha=1,2,\dots,N.
$$
Since $ \bar L^{-1}|\bm\omega|^{2}\le \frac{A_{ij}^{\alpha\beta}({\bf P})}{\phi''(\lambda)}\omega_i^\alpha \omega_j^\beta\le \bar L |\bm\omega|^{2}$ for all $\bm\omega=(\omega^\alpha_i)\in \R^{Nn}$ and some $ \bar L \ge 1$, by regularity theory for linear  parabolic systems with constant coefficients (see, for instance, \cite[XI. Theorem 6.6]{Giusti_book} with its proof),  we have that for every $\rho\in(0,\frac12)$,
$$
\int_{Q_\rho} |D\bar {\bf v} - (D\bar {\bf v})_{Q_\rho} |^2 \, \d z  \le c \rho^2 \int_{Q_{1/2}} |D\bar {\bf v} - (D\bar {\bf v})_{Q_{1/2}} |^2 \, \d z \, ,
$$
which implies the following estimate for ${\bf v}$: for every $\rho\in (0,\frac{R}{2})$, 
\begin{equation}\label{decayestimate_constant}
\int_{Q^\lambda_\rho} |D {\bf v} - (D {\bf v})_{Q^\lambda_\rho} |^2 \, \d z  \le c \left(\frac{\rho}{R}\right)^{n+4} \int_{Q^\lambda_{R/2}} |D {\bf v} - (D {\bf v})_{Q^\lambda_{R/2}} |^2 \, \d z \, .
\end{equation}
The estimate \eqref{decayestimate_constant} is a key ingredient to obtain the following result, which provides an estimate for the decay of the mean oscillation of $D{\bf u}$ on each scale (cfr. \cite[Lemma~4.4]{Lie06}).

\begin{lemma}
Suppose
$$
|(D{\bf u})_{Q^\lambda_R}| \ge \frac{1}{2} \lambda
\quad\text{and}\quad 
\fint_{Q^\lambda_R} |D{\bf u} - (D{\bf u})_{Q^\lambda_R}|^2 \, \d z \le \epsilon \lambda^2
$$
for some $\epsilon\in(0,1)$.
Then for every $\theta, \epsilon_0\in (0,\frac{1}{2})$, we have 
$$
\int_{Q^\lambda_{\theta R}} |D{\bf u} - (D{\bf u})_{Q^\lambda_{\theta R}}|^2 \, \d z \le c_1 (\epsilon_0^{2\gamma_1} +  \epsilon_0^{-2\gamma}  \epsilon^{\gamma}+ \theta^{n+4})\int_{Q^\lambda_{R}} |D{\bf u} - (D{\bf u})_{Q^\lambda_R}|^2 \, \d z
$$
for some $c_1\ge 1$, where $\gamma>0$ is from Lemma~\ref{Lem:high}.
\end{lemma}

\begin{proof}
By   Lemma~\ref{Lem:comparison} with ${\bf P}=(D{\bf u})_{Q^\lambda_R}$ and \eqref{decayestimate_constant}, we have that for every $\theta\in(0,\frac12)$,
$$\begin{aligned}
\int_{Q^\lambda_{\theta R}} & |D{\bf u} - (D{\bf u})_{Q^\lambda_{\theta R}}|^2 \, \d z
 \le 2 \int_{Q^\lambda_{R}} |D{\bf u} -  D\bfv|^2 \, \d z+ 2\int_{Q^\lambda_{\theta R}} |D\bfv - (D\bfv)_{Q^\lambda_{\theta R}}|^2 \, \d z \\
& \le c \left[\epsilon_0^{2\gamma_1} +  \epsilon_0^{-2\gamma}  \epsilon^{\gamma}\right]  \int_{Q^{\lambda}_{R}} |D{\bf u}-(D{\bf u})_{Q^\lambda_{R}}|^2 \, \d z  + c \theta^{n+4}\int_{Q^\lambda_{R/2}} |D\bfv - (D{\bf u})_{Q^\lambda_{R}}|^2 \, \d z\,.  
\end{aligned}$$
This concludes the proof.
\end{proof}

The counterpart of \cite[Lemma~4.5]{Lie06} is the following result.

\begin{lemma}\label{Lem:prop1lem1}
There exist small constants $\theta,\epsilon\in(0,1)$ such if
$$
|(D{\bf u})_{Q^\lambda_R}| \ge \frac{3}{4} \lambda
\quad\text{and}\quad
\fint_{Q^\lambda_R} |D{\bf u} - (D{\bf u})_{Q^\lambda_R}|^2 \, \d z \le \epsilon \lambda^2,
$$
then for every $m\in \mathbb N$,
$$
|(D{\bf u})_{Q^\lambda_{R_m}}| \ge \left(\frac12 +\frac{1}{2^{2+m}}\right) \lambda, \quad \text{where }\ R_m:= \theta^m R,
$$ 
and
$$\begin{aligned}
\fint_{Q^\lambda_{R_{m}}} |D{\bf u} - (D{\bf u})_{Q^\lambda_{R_{m}}}|^2 \, \d z
& \le  \theta^{\frac32}\fint_{Q^\lambda_{R_{m-1}}} |D{\bf u} - (D{\bf u})_{Q^\lambda_{R_{m-1}}}|^2 \, \d z  \\
& \le \cdots \le \theta^{\frac 32 m}\fint_{Q^\lambda_{R}} |D{\bf u} - (D{\bf u})_{Q^\lambda_{R}}|^2 \, \d z.
\end{aligned}$$
In particular, we have that for every $r\in(0,R]$,
$$
\fint_{Q^\lambda_{r}} |D{\bf u} - (D{\bf u})_{Q^\lambda_{r}}|^2 \, \d z
 \le  \theta^{-n-\frac{7}{2}}\left(\frac{r}{R}\right)^{\frac32}\fint_{Q^\lambda_{R}} |D{\bf u} - (D{\bf u})_{Q^\lambda_{R}}|^2 \, \d z\,.$$
\end{lemma}

\begin{proof} Choose $\theta$, $\epsilon$ and $\epsilon_0$ small so that
$$
\theta \le \min\left\{(4c_1)^{-2},2^{-4/3}\right\},\ \ \epsilon_0\le  \left(\frac{\theta^{n+7/2}}{4c_1}\right)^{\frac{1}{2\gamma_1}},\ \ 
 \epsilon \le \min\{\epsilon_0^{2} [\theta^{n+7/2}/ (4c_1)]^{1/\eta}, 2^{-6}\theta^{2n+4}\} ,
$$
where the constant $c_1\ge 1$ is from the preceding lemma.
We prove the lemma by induction. We first obtain the desired estimates when $m=1$. By the preceding lemma, we have 
$$\begin{aligned}
\fint_{Q^\lambda_{\theta R}} |D{\bf u} - (D{\bf u})_{Q^\lambda_{\theta R}}|^2 \, \d z 
& \le  \theta^{3/2} \fint_{Q^\lambda_{R}} |D{\bf u} - (D{\bf u})_{Q^\lambda_R}|^2 \, \d z \,,
\end{aligned}$$
which is the desired second estimate when $m=1$.  Moreover,
$$
|(D{\bf u})_{Q^\lambda_{\theta R}} - (D{\bf u})_{Q^\lambda_{R}}|  \le \theta^{-n-2} \fint_{Q^\lambda_{R}}|D{\bf u} - (D{\bf u})_{Q^\lambda_{R}}| \, \d z  \le \theta^{-n-2}\epsilon^{1/2}\lambda \le \frac{1}{8}\lambda \,,
$$
hence
$$
|(D{\bf u})_{Q^\lambda_{\theta R}}| \ge | (D{\bf u})_{Q^\lambda_{R}}| -  \frac{1}{8}  \lambda  \ge  \left(\frac{1}{2}+\frac{1}{2^3}  \right)\lambda. 
$$
This is the first estimate when $m=1$.

Next, we suppose the estimates hold for all $m\le  m_0$. Then from the first and the second inequalities when $m=m_0$ we see that 
$$
|(D{\bf u})_{Q^\lambda_{R_{m_0}}}| \ge  \left(\frac{1}{2}+\frac{1}{2^{2+m_0}}  \right)\lambda > \frac{1}{2} \lambda
\ \ \text{and} \ \ 
\fint_{Q^\lambda_{R_{m_0}}} |D{\bf u} - (D{\bf u})_{Q^\lambda_{R_{m_0}}}|^2 \, \d z \le \epsilon \theta^{\frac{3}{2}m_0} \lambda^2 <  \epsilon  \lambda^2\, .
$$
Therefore, applying the preceding lemma with $R$ replaced by $R_{m_0}$ and the second inequality when $m=m_0$, we see that 
$$\begin{aligned}
\fint_{Q^\lambda_{R_{m_0+1}}} |D{\bf u} - (D{\bf u})_{Q^\lambda_{R_{m_0+1}}}|^2 \, \d z 
&\le  \theta^{3/2} \fint_{Q^\lambda_{R_{m_0}}} |D{\bf u} - (D{\bf u})_{Q^\lambda_{R_{m_0}}}|^2 \, \d z  \le \theta^{\frac{3}{2}(m_0+1)} \epsilon \lambda^2.
\end{aligned}$$
This is the second estimate when $m=m_0+1$. Moreover, by the first estimate when $m=m_0$ and the upper bounds of $\theta$ and $\epsilon$, 
$$\begin{aligned}
|(D{\bf u})_{Q^\lambda_{R_{m_0+1}}} - (D{\bf u})_{Q^\lambda_{R_{m_0}}}| \le \frac{1}{\theta^{n+2}} \fint_{Q^\lambda_{R_{m_0}}}|D{\bf u} - (D{\bf u})_{Q^\lambda_{R_{m_0}}}| \, \d z  \le \frac{\theta^{\frac{3}{4}m_0}\epsilon^{\frac 12}}{\theta^{n+2}}\lambda \le \frac{1}{2^{3+m_0}}\lambda\,.
\end{aligned}$$
  Therefore we have
$$
|(D{\bf u})_{Q^\lambda_{R_{m_0+1}}}| \ge | (D{\bf u})_{Q^\lambda_{R_{m_0}}}| -  \frac{1}{2^{3+m_0}}\lambda  \ge \left(\frac{1}{2}+\frac{1}{2^{3+m_0}}  \right)\lambda . 
$$
This is the first estimate when $m=m_0+1$, and the proof is concluded. 
\end{proof}

We conclude the list of the auxiliary results needed to prove Proposition~\ref{prop:nondegenerate} with the following one, which corresponds to \cite[Lemma~4.6]{Lie06}. For a $\sigma>0$ small enough such that \eqref{prop:nondegenerate_ass} holds, we have that the average of $D{\bf u}$ is comparable with $\lambda$ and that $D{\bf u}$ remains close to its average.

\begin{lemma}\label{Lem:prop1lem2}
For $\epsilon >0$, there exists $\sigma=\sigma(\epsilon)\in (0,2^{-(n+2)})$ such that if $\sigma$ satisfies   \eqref{prop:nondegenerate_ass}, then 
$$
\frac{7}{8} \lambda \le |(D{\bf u})_{Q^\lambda_{R/2}}| \le \lambda
$$
and
$$
\fint_{Q^\lambda_{R/2}} |D{\bf u}-(D{\bf u})_{Q^\lambda_{R/2}}|^2 \le \epsilon \lambda^2.
$$
\end{lemma}

\begin{proof}

Let $f(s):=(s- (1-2\theta)\lambda)_+$ with $\theta\in (0,1/4)$ to be determined and $F(s):=\int_0^s \tau f(\tau)\,\d \tau$. Note that  we have that
\begin{itemize}
\item[(i)] when $f(|D{\bf u}|)>0$, $\frac{1}{2}\lambda< (1-2\theta)\lambda \le |D{\bf u}| \le \lambda$ (hence $2^{-q}\phi(\lambda) \le \phi(|D{\bf u}|)\le \phi(\lambda)$) and $f'(|D{\bf u}|)=1$;
\item[(ii)] $0\le f(|D{\bf u}|)\le 2\theta\lambda$ and $0\le F(|D{\bf u}|) \le  4\theta^2\lambda^3$. 
\end{itemize}
 Let  $\xi\in C^\infty_0(B_R)$ and $\eta\in C^\infty(\R)$  be cut-off functions such that $\xi\equiv 1$ in $B_{R/2}$, $|D\xi|\le \frac{4}{R}$, $\eta\equiv  0$ in $(-\infty,-\frac{R^2}{\phi''(\lambda)})$, $\eta\equiv 1$ in $(-\frac{R^2}{4\phi''(\lambda)},\infty)$ and $0\le \eta_t \le \frac{8\phi''(\lambda)}{R^2}$.
Then, the Caccioppoli  estimate \eqref{Caccio estimate} with $\rho=R$, $\tau_1=-\frac{R^2}{\phi''(\lambda)}$ and $\tau_2=-\frac{R^2}{4\phi''(\lambda)}$ yields
$$\begin{aligned}
 \int_{A((1-\theta)\lambda,R/2)}  \phi''(\lambda)\theta \lambda |D^2\bfu|^2\, \d z & \le  c \int_{Q^\lambda_{R/2}}  \frac{\phi'(|D{\bf u}|)}{|D{\bf u}|}f(|D{\bf u}|)|D^2\bfu|^2 \, \d z\\
 & \le  cR^{-2} \int_{Q^\lambda_{R}}  \left[\phi'(|D{\bf u}|)\frac{f(|D{\bf u}|)^2}{f'(|D{\bf u}|)} +\phi''(|D{\bf u}|) F(|D{\bf u}|)\right] \, \d z\\
 &  \le  cR^{-2} \big|\{|D{\bf u}| >(1-2\theta)\lambda\}\cap Q^\lambda_{R}\big|  \theta^2 \lambda \phi(\lambda)\,,
\end{aligned}$$
where $A(k,r):= \{z\in Q^\lambda_r: |D{\bf u}(z)|>k\}$. Note that in the second inequality, we used the fact that $\phi''(\lambda)\le 2^{q+2}\phi''(|D\bfu|)$ when $F(|D\bfu|)>0$.  Hence we have
\begin{equation}\label{lem5.6:pf1}
 \int_{A((1-\theta)\lambda,R/2)}   |D^2\bfu|^2\, \d z   \le  c \, \theta R^{-2} \lambda^2\big|\{|D{\bf u}| > (1-2\theta)\lambda\}\cap Q^\lambda_{R}\big|   \,.
\end{equation}

Let $h_0\in C^1(\R)$ be increasing such that $h_0(t)=0$ for $t\le 3\lambda/4$,  $h_0(t)=1$ for $t>7\lambda/8$, and $h_0'\le 16 \lambda^{-1}$, and set ${\bf h}(z)=h_0(|D{\bf u}(z)|)D{\bf u}(z)$.  Then we have $|D{\bf h}|^2\le c |D^2\bfu|^2$. 
Let $\xi_0 \in C^\infty_0(B_{R/2})$ be a cut-off function such that $0\le \xi_0  \le 1$ with $\xi_0\equiv 1$ in $B_{R/4}$ and $|D^2\xi_0|+|D\xi_0|^2 \le c /R^2$, and $\xi = \|\xi_0\|_{L^1(B_{R/2})}^{-1}\xi_0\approx \frac{\xi_0}{R^n}$. Set
$$
{\bf W}(t) := \int_{B_R} D{\bf u} (x,t) \xi(x)\, \d x  \, ,
\quad
{\bf W}_h(t) := \int_{B_R} {\bf h} (x,t)\xi(x)\, \d x 
$$
 and
 $$
{\bf W}_0 = \frac{\phi''(\lambda)}{R^2} \int^0_{-\frac{R^2}{\phi''(\lambda)}}\int_{B_R} D{\bf u} (x,t) \xi(x) \, \d x \, \d t = \fint^{0}_{-\frac{R^2}{\phi''(\lambda)}} {\bf W}(t) \, \d t \, .
$$
Note that by Lemma~\ref{lem:wpoincare} with $f={\bf h}(\cdot,t)$ and $p=\frac{2n}{n+1}$, we have 
\begin{equation} \begin{aligned}
\iint_{Q^\lambda_{R/2}} |{\bf h}(x,t)-{\bf W}_h(t)|^2 \, \d x  \, \d t  &\le c \lambda^{\frac{2}{n+1}} \iint_{Q^\lambda_{R/2}} |{\bf h}(x,t)-{\bf W}_h(t)|^{\frac{2n}{n+1}} \, \d x \, \d t \\
&\le c (|B_R|\lambda)^{\frac{2}{n+1}} \int_{Q^\lambda_{R/2}} |D{\bf h}|^{\frac{2n}{n+1}} \, \d z \,.
\label{hestimate}
\end{aligned}\end{equation}
Set 
$$
\Sigma_0:= A(3\lambda/4, R/2)\setminus A((1-\sigma)\lambda, R/2)  
\quad\text{and}\quad
\Sigma:= A((1-\sigma)\lambda, R/2)   \,.
$$
Since $|D{\bf h}|=0$ on $Q^\lambda_{R/2}\setminus   A(3\lambda/4, R/2)$, we have
$$
\int_{Q^\lambda_{R/2}} |D{\bf h}|^{\frac{2n}{n+1}} \, \d z = \int_{\Sigma_0} |D{\bf h}|^{\frac{2n}{n+1}} \, \d z + \int_{\Sigma} |D{\bf h}|^{\frac{2n}{n+1}} \, \d z \,.
$$
By H\"older's inequality, \eqref{prop:nondegenerate_ass} and \eqref{lem5.6:pf1} with $\theta=\frac{1}{4}$, the first term on the right hand side can be estimated as
$$
\int_{\Sigma_0} |D{\bf h}|^{\frac{2n}{n+1}} \, \d z  \le |\Sigma_0|^{\frac{1}{n+1}}\left( \int_{A(3\lambda/4, R/2) } |D{\bf h}|^{2}\, \d z\right)^{\frac{n}{n+1}} \le c \sigma^{\frac{1}{n+1}} \lambda^{\frac{2n}{n+1}}R^{-\frac{2n}{n+1}}| Q^{\lambda}_{R}|\,.
$$
Moreover, by \eqref{lem5.6:pf1} with $\theta=\sigma$
$$\begin{aligned}
\int_{\Sigma} |D{\bf h}|^{\frac{2n}{n+1}} \, \d z & \le |Q^{\lambda}_{R}|^{\frac{1}{n+1}}\left( \int_{A((1-\sigma)\lambda, R/2) } |D{\bf h}|^{2}\, \d z\right)^{\frac{n}{n+1}}\\
& \le c \sigma^{\frac{n}{n+1}} \lambda^{\frac{2n}{n+1}}R^{-\frac{2n}{n+1}}| Q^{\lambda}_{R}|\,.
\end{aligned}$$
Therefore, we have 
$$
\int_{Q^\lambda_{R/2}} |D{\bf h}|^{\frac{2n}{n+1}} \, \d z  \le c \sigma^{\frac{1}{n+1}} \lambda^{\frac{2n}{n+1}}R^{-\frac{2n}{n+1}}| Q^{\lambda}_{R}| \,,
$$
hence, inserting this into \eqref{hestimate},
\begin{equation}\label{Sobolev_h}
\iint_{Q^\lambda_{R/2}} |{\bf h}(x,t)-{\bf W}_h(t)|^2 \, \d x  \, \d t  \le c   \sigma^{\frac{1}{n+1}}\lambda^{2}|Q^\lambda_R|\, .
\end{equation}

Note that 
$$
|D{\bf u} - \bfw _0|^2 \le c \left(|D{\bf u} - {\bf h} |^2 + |{\bf h} - {\bf W}_h |^2 + |{\bf W}_h - {\bf W} |^2 + |{\bf W} - \bfw _0 |^2\right).
$$
First, by \eqref{prop:nondegenerate_ass} with the definition of ${\bf h}$, we have 
\begin{equation}\label{comparison:Duh}
\int_{Q^\lambda_{R/2}} |D{\bf u} - {\bf h} |^2 \, \d z = \int_{Q^\lambda_{R/2}\setminus \Sigma} (1-h_0(|D{\bf u}|))^2|D{\bf u}|^2 \, \d z \le  \sigma \lambda^2  |Q^\lambda_{R}| \,.
\end{equation}
Moreover, by H\"older's inequality and the definition of $\zeta$, we also have  
\begin{equation}\label{comparison:WhW}
\begin{aligned}
&\int_{Q^\lambda_{R/2}} |{\bf W}_h - {\bf W} |^2\, \d z   = |B_{R/2}| \int^0_{-R^2/(4\phi''(\lambda))} \left|\int_{B_{R/2}} (D{\bf u} (x,t) -{\bf h} (x,t)  )\xi(x)\, \d x \right|^2 \, \d t \\
& \le |B_{R/2}| \int^0_{-R^2/(4\phi''(\lambda))} \left[\int_{B_{R/2}} (D{\bf u} (x,t) -{\bf h} (x,t)  )^2\, \d x\right]\left[ \int_{B_{R/2}} \xi^2\, \d x \right] \, \d t \\
& \le c  \int_{Q^\lambda_{R/2}} |D{\bf u}  -{\bf h}   |^2 \, \d z \le c\sigma \lambda^2 |Q^\lambda_R|.
\end{aligned}
\end{equation}
Second, since $\bfu=\bfu_\epsilon$ is a weak solution to  \eqref{eq:nondegenerate systemeq}, by testing \eqref{eq:nondegenerate systemeq} with $\zeta= (\xi_{x_i},\dots,\xi_{x_i})\in C^\infty_0(B_R,\R^N)$, $i=1,2,\dots, n$, we have  that for every $-\frac{R^2}{\phi''(\lambda)} <\tau <\tau' <0$,
$$
\int_{B_r} (u^\alpha_{x_i}(x,\tau')  - u^\alpha_{x_i}(x,\tau) )  \xi(x) \,dx  = \int^{\tau'}_{\tau} \int_{B_R} \left([\bA(D{\bf u})]^\alpha -[\bA({\bf W}(t)]^{\alpha})\right) \cdot D\xi_{x_i} \, \d z\, , 
$$
$\alpha=1,2,\dots, N$, hence, using \eqref{PhiPQ} and the facts that $|D{\bf u}| \le \lambda$ in $Q^\lambda_R$ and $\phi(t)/t^{p-1}$ is increasing,
\begin{equation}\label{estimateW}\begin{aligned}
|{\bf W}(\tau')&-{\bf W}(\tau)|  \le c \int_{Q^\lambda_{R/2}} |\bA(D{\bf u})-\bA({\bf W}(t))| |D^2\xi| \, \d z \\
& \le \frac{c}{R^{n+2}} \int_{Q^\lambda_{R/2}} \frac{\phi'(|D{\bf u}|+|{\bf W}(t)|)}{|D{\bf u}|+|{\bf W}(t)|} (|D{\bf u}|+|{\bf W}(t)|)^{2-p} |D{\bf u}-{\bf W}(t)|^{p-1} \, \d z \\
& \le c\frac{\phi'(\lambda)}{\lambda^{p-1}R^{n+2}}  \int_{Q^\lambda_{R/2}} |D{\bf u}-{\bf W}|^{p-1} \, \d z  \\   
& \le c \lambda^{2-p} \left(\fint_{Q^\lambda_{R/2}} |D{\bf u}-{\bf W}|^2 \, \d z\right)^{\frac{p-1}{2}}   .    
\end{aligned}\end{equation}
Applying the estimates \eqref{Sobolev_h}, \eqref{comparison:Duh} and \eqref{comparison:WhW}, we obtain for every $-\frac{R^2}{\phi''(\lambda)} <\tau <\tau' <0$, 
$$
|{\bf W}(\tau')-{\bf W}(\tau)| \le c  \lambda^{2-p}  \left( \sigma^{\frac{1}{n+1}}\lambda^{2} \right)^{\frac{p-1}{2}}   \le c \sigma^{\frac{p-1}{2(n+1)}} \lambda\, ,
$$
which also implies 
\begin{equation}\label{comparaion:Ww0}
\int_{Q^\lambda_{R/2}} |{\bf W}-\bfw _0|^2\,\d z \le c|Q^\lambda_{R}|\sup_{-\frac{R^2}{\phi''(\lambda)}<\tau<\tau'<0} |{\bf W}(\tau')-{\bf W}(\tau) |^2 
\le  c \sigma^{\frac{p-1}{n+1}}  \lambda^2 |Q^\lambda_{R}|\,.
\end{equation}
Therefore, combining the results in \eqref{Sobolev_h}--\eqref{comparaion:Ww0}, we have 
$$
\fint_{Q^\lambda_{R/2}} |D{\bf u} - (D{\bf u})_{Q^\lambda_{R/2}}|^2 \, \d z  \le \fint_{Q^\lambda_{R/2}} |D{\bf u} - \bfw _0|^2 \, \d z  \le c \sigma^{\frac{p-1}{n+1}} \lambda^2.
$$
This implies the second desired estimate by choosing $\sigma$ sufficiently small. Moreover, since $|D{\bf u}|\le |D{\bf u}-(D{\bf u})_{Q^\lambda_{R/2}}|+|(D{\bf u})_{Q^\lambda_{R/2}}|$ and by the assumption of the lemma
$$
\big|\{|D{\bf u}|>(1-\sigma)\lambda \}\cap Q^\lambda_{R/2}\big| > (1-\sigma)|Q^{\lambda}_{R/2}|>(1-2^{-(n+2)})|Q^{\lambda}_{R/2}|\,,
$$
we have 
$$\begin{aligned}
| (D{\bf u})_{Q^\lambda_{R/2}}| &\ge  \fint_{Q^\lambda_{R/2}} |D{\bf u}| \, \d z -  \left(\fint_{Q^\lambda_{R/2}} |D{\bf u} - (D{\bf u})_{Q^\lambda_{R/2}}|^2 \, \d z \right)^{\frac{1}{2}} \\
& \ge \left[(1-\sigma) (1-2^{-(n+2)}) - c \sigma^{\frac{p-1}{2(n+1)}}\right] \lambda. 
\end{aligned}$$
Finally, by choosing $\sigma$ sufficiently small we obtain the first desired estimate.
\end{proof}

We are now in position to prove Proposition~\ref{prop:nondegenerate}.

\begin{proof}[Proof of Proposition~\ref{prop:nondegenerate}]
As we mentioned above, we assume $z_0=0$. Let $\epsilon>0$ be given from Proposition~\ref{Lem:prop1lem1}. With this $\epsilon$, we determine $\sigma$ as in Lemma~\ref{Lem:prop1lem2} with $\frac{\epsilon}{2}$ in place of $\epsilon$, and suppose that $Q^\lambda_R=Q^\lambda_R(0)$ satisfies \eqref{prop:nondegenerate_ass}. Therefore, we have 
$$
\frac{7}{8}\lambda \le |(D{\bf u})_{Q^\lambda_{R/2}}| \le \lambda
$$
and
$$
\fint_{Q^\lambda_{R/2}} |D{\bf u}- (D{\bf u})_{Q^\lambda_{R/2}}|^2\, \d z \le \frac{\epsilon}{2} \lambda\,.
$$
Choose $z_1=(x_1,t_1)\in Q^\lambda_R$ such that $\max\{|x_1|, \sqrt{\phi''(\lambda)|t_1|}\}  < \epsilon' R$, where $\epsilon'\in (0,1/2)$ is a small constant to be determined. Then $Q_{R/2}^\lambda(z_1) \subset Q_R^\lambda$ and $|Q_{R/2}^\lambda(z_1)\setminus Q_{R/2}^\lambda|, |Q_{R/2}^\lambda\setminus Q_{R/2}^\lambda(z_1)| \le c \epsilon' |Q_R^\lambda|$ for some $c>0$ depending $n$ if $\epsilon'$ is sufficiently small.
Therefore,  by  \eqref{Dubound} and the preceding two inequalities, we have  
$$\begin{aligned}
&\fint_{Q^\lambda_{R/2}(z_1)} |D{\bf u}- (D{\bf u})_{Q^\lambda_{R/2}(z_1)}|^2\, \d z  \\
&\le \frac{1}{|Q^\lambda_{R/2}|}\int_{Q_{R/2}^\lambda(z_1)\setminus Q_{R/2}^\lambda} |D{\bf u}- (D{\bf u})_{Q^\lambda_{R/2}}|^2\, \d z + \fint_{Q^\lambda_{R/2}} |D{\bf u}- (D{\bf u})_{Q^\lambda_{R/2}}|^2\, \d z\\
&\le  \left(c\epsilon'+\frac{\epsilon}{2}\right)  \lambda^2 \le \epsilon \lambda^2
\end{aligned}$$
and
$$\begin{aligned}
|(D{\bf u})_{Q^\lambda_{R/2}(z_1)}| 
&\ge  |(D{\bf u})_{Q^\lambda_{R/2}}| - |(D{\bf u})_{Q^\lambda_{R/2}}  -(D{\bf u})_{Q^\lambda_{R/2}(z_1)}|\\ 
&\ge \frac{7}{8}\lambda  -  \frac{1}{|Q^\lambda_{R/2}|}\int_{(Q_{R/2}^\lambda(z_1)\setminus Q_{R/2}^\lambda )\cup ( Q_{R/2}^\lambda \setminus Q_{R/2}^\lambda(z_1))}  |D{\bf u}|\, \d z\\
& \ge \frac{7}{8}\lambda - c \epsilon' \lambda \ge \frac{3}{4}\lambda\, .  
\end{aligned}$$
provided that $\epsilon'$ is sufficiently small.
Therefore, by Lemma~\ref{Lem:prop1lem1} with $Q^\lambda_{R/2}(z_1)$ in place of $Q^\lambda_R$, we obtain that for every $z_1\in Q^\lambda_{R/2}$ and every $0< r < \frac{R}{2}$ we have 
\begin{equation}\label{meandecay}
\fint_{Q^\lambda_r(z_1)}  |D{\bf u} - (D{\bf u})_{Q^\lambda_r(z_1)}|^2 \,dz  \le c \left(\frac{r}{R}\right)^{3/2}\fint_{Q^\lambda_{R/2}(z_1)}  |D{\bf u} - (D{\bf u})_{Q^\lambda_{R/2}(z_1)}|^2 \, \d z \,. 
\end{equation}
Consequently, the desired oscillation decay estimate \eqref{nondegenerate_oscillation} follows from  \eqref{meandecay}  by the standard embedding argument by Campanato in the parabolic setting (see for instance \cite[Lemma 4.3]{Lie_book}).
\end{proof}

\subsection{Proof of Proposition~\ref{prop:degenerate}} \label{sec:proofdeg}

We start with a density result from \eqref{prop:degenerate_ass} (cfr. \cite[Lemma 6.2]{Lie06}).

%
%
%

\begin{lemma}\label{Lem:density_timeslice2}
Suppose \eqref{prop:degenerate_ass} holds for some $\sigma\in(0,\frac{1}{2})$.  There exists $m_1 \in \mathbb N$ depending on $\sigma$ such that
$$
\sup_{-\frac{\sigma R^2}{2\phi''(\lambda)}\le t \le 0}  |\{ |D{\bf u}(x,t)|>(1-2^{-m_1}\sigma)\lambda\}\cap B_{\tilde{\sigma} R}| \le \tilde{\sigma} |B_{\tilde{\sigma} R}| \, , \quad \text{where }\ \tilde{\sigma} := \left(1-\frac \sigma 2 \right)^{\frac{1}{n+2}}.
$$ 

\end{lemma}

\begin{proof}
\textit{Step 1.} We first prove that there exists $t_1\in (-\frac{R^2}{\phi''(\lambda)}, -\frac{\sigma R^2}{2\phi''(\lambda)})$ such that
\begin{equation}\label{density_timeslice}
|\{x\in B_R: |D{\bf u}(x,t_1)|>(1-\sigma)\lambda\}| \le \left(1-\frac{\sigma}{2}\right) |B_R| \, .
\end{equation}

Set
$$
I := \int_{-\frac{R^2}{\phi''(\lambda)}}^{-\frac{\sigma R^2}{2\phi''(\lambda)}} |\{x\in B_R: |D{\bf u}(x,t)|>(1-\sigma)\lambda\}| \, \d t \, .
$$
Then by \eqref{prop:degenerate_ass}, 
$$
I \le   |\{|D{\bf u}| > (1-\sigma) \lambda\} \cap Q^{\lambda}_{R}| \le (1-\sigma)  \frac{R^2}{\phi''(\lambda)} |B_{R}| \, .
$$
On the other hand, by the mean value theorem for integrals, there exists $t_1\in (-\frac{R^2}{\phi''(\lambda)}, -\frac{\sigma R^2}{2\phi''(\lambda)})$ such that
$$
I = |\{x\in B_R: |D{\bf u}(x,t_1)|>(1-\sigma)\lambda\}| \left( 1-\frac{\sigma}{2} \right) \frac{R^2}{\phi''(\lambda)} \, .
$$
This inequality with $\frac{1-\sigma}{1-\sigma/2} < 1-\frac{\sigma}{2}$ yields \eqref{density_timeslice}.

\textit{Step 2.} Let 
$$
\Psi(s) := \ln_+ \left( \frac{\sigma}{1-s/\lambda +2^{1-m}\sigma}\right)  ,  \qquad 0 \le s \le \lambda \,,
$$
where $g_+(s):=\max\{g(s),0\}$ and $m\in \mathbb N$ with $m\ge 2$, and set
$$
 f (s) := \frac{2\Psi(s)\Psi'(s)}{s}\,,  \qquad 0 \le s \le \lambda \,.
$$
Then we have the following straightforward properties for $\Psi$ and $f$:
\begin{itemize}
 \item[(i)] $0\le \Psi(s)\le (m-1)\ln 2$ and $\Psi(s) =0$ if and only if $0\le s \le (1-\sigma  +2^{1-m}\sigma)\lambda$;
 \item[(ii)] for $(1-\sigma  +2^{1-m}\sigma)\lambda < s < \lambda$, $\Psi'(s) = \frac{1}{(1+2^{1-m}\sigma)\lambda-s}$ and $\Psi''(s)=\Psi'(s)^2$;
\item[(iii)]   for $(1-\sigma  +2^{1-m}\sigma)\lambda < s < \lambda$, $f'(s)=\frac{2(1+\Psi(s))\Psi'(s)^2}{s}-\frac{f(s)}{s}>0$;
\item[(iv)] $F(s):= \int^s_0 \tau f(\tau)\, \d \tau = \Psi(s)^2$.

\end{itemize}
 Choose any  $\tilde t\in \left[-\frac{\sigma R^2}{2\phi''(\lambda)}, 0\right]$ and let $t_1$ be that given from \eqref{density_timeslice}. Then from the proof of Lemma~\ref{Lem:Caccio1} (applying \eqref{I2estimate} to $I_2$) with $\rho =R$, $\tau_1=t_1$, $\tau_2=\tilde t$ and $\eta\equiv 1$, we deduce that  
\begin{equation} 
\begin{aligned}
&\int_{B_{R}}   F(|D{\bf u}(x,\tilde t)|)\xi^2  \, \d x \\
&\qquad + \int_{t_1}^{\tilde t} \int_{B_R}  \frac{\phi'(|D{\bf u}|)}{|D{\bf u}|}\left(f(|D{\bf u}|)|D^2\bfu|^2 +   \frac{f'(|D{\bf u}|)}{|D{\bf u}|}\frac{|D (|D{\bf u}|^2)|^2}{4}\right) \xi^2\, \d x \d t\\
&  \le \int_{B_{R}}   F(|D{\bf u}(x,t_1)|)\xi^2 \, \d x   + c  \int^{\tilde t}_{t_1} \int_{B_R}  \frac{\phi'(|D{\bf u}|)}{|D{\bf u}|}\big|D(|D{\bf u}|^2)\big| f(|D{\bf u}|)  \xi  |D\xi|  \, \d x \d t  \,, \\
\end{aligned}\label{CaccioPsi1}
\end{equation}
where $\xi\in C^\infty_0(B_R)$ is a cut-off function satisfying that $\xi\equiv 1$ on $B_{\tilde{\sigma} R}$ and $|D\xi|\le \frac{2}{(1-\tilde{\sigma})R}$.
For what concerns the left hand side, applying (iii) and (iv), we have
$$\begin{aligned}
\text{(LHS of \eqref{CaccioPsi1})} & \ge \int_{B_{R}}   \Psi(|D{\bf u}(x,\tilde t)|)^2\xi^2  \, \d x  \\
&\qquad + \int^{\tilde t}_{t_1} \int_{B_R}    \frac{\phi'(|D{\bf u}|)}{|D{\bf u}|}  \frac{(\Psi(|D{\bf u}|)+1)\Psi'(|D{\bf u}|)^2}{|D{\bf u}|^2}\frac{|D (|D{\bf u}|^2)|^2}{2} \xi^2\, \d x \d t\, .
\end{aligned}$$ 
On the other hand, as for the right hand side, Young's inequality with the definition of $f$ and (iv) yields.
$$\begin{aligned}
\text{(RHS of \eqref{CaccioPsi1})} & \le \int_{B_{R}}   \Psi(|D{\bf u}(x,t_1)|)^2\xi^2  \, \d x  + c \int^{\tilde t}_{t_1} \int_{B_R}  \frac{\phi'(|D{\bf u}|)}{|D{\bf u}|}\Psi(|D{\bf u}|)  |D\xi|^2  \, \d x \d t  \\
&\qquad + \frac{1}{4}\int^{\tilde t}_{t_1} \int_{B_R}  \frac{\phi'(|D{\bf u}|)}{|D{\bf u}|}\frac{\Psi(|D{\bf u}|)\Psi'(|D{\bf u}|)^2}{|D{\bf u}|^2}\big|D(|D{\bf u}|^2)\big|^2   \xi^2   \, \d x \d t \, .
\end{aligned}$$ 
Therefore, we have 
\begin{equation}
\begin{aligned}
\int_{B_{\tilde{\sigma} R}}  & \Psi(|D{\bf u}(x,\tilde t)|)^2  \, \d x \\
 & \le \int_{B_{R}}   \Psi(|D{\bf u}(x,t_1)|)^2 \, \d x + \frac{c}{(1-\tilde{\sigma})^2R^2} \int_{\tilde t}^{t_2} \int_{B_R}  \frac{\phi'(|D{\bf u}|)}{|D{\bf u}|}\Psi(|D{\bf u}|)  \, \d x \d t \,.
\end{aligned}\label{CaccioPsi2}
\end{equation}
For the left hand side, we see that 
$$
\begin{aligned}
\text{(LHS of \eqref{CaccioPsi2})}   &  \ge \Psi((1-2^{-m}\sigma)\lambda)^2  |\{ |D{\bf u}(x,\tilde t)|>(1-2^{-m}\sigma)\lambda\}\cap B_{\tilde{\sigma} R}|\\
& \ge ((m-2)\ln 2)^2   |\{ |D{\bf u}(x,\tilde t)|>(1-2^{-m}\sigma)\lambda\}\cap B_{\tilde{\sigma} R}| \, .
\end{aligned}
$$ 

On the other hand, to estimate the right hand side, we apply Step 1 
and (i) to get
$$
\begin{aligned}
\text{(RHS of \eqref{CaccioPsi2})}  & \le ((m-1)\ln 2)^2    |\{x\in B_R: |D{\bf u}(x,t_1)|>(1-\sigma+2^{1-m}\sigma)\lambda\}|  \\
&\qquad  +  \frac{ c (m-1)\ln 2}{(1-\sigma) (1-\tilde{\sigma})^2} |B_R|  \,\\
& \le \left\{((m-1)\ln 2)^2  \frac{\left(1-\frac{\sigma}{2}\right)}{\tilde{\sigma}^n}   +   \frac{c (m-1)  }{\tilde{\sigma}^n} \right\}|B_{\tilde{\sigma} R}|\\
& \le \left\{((m-1)\ln 2)^2 \tilde{\sigma}^2   +   \frac{c (m-1)  }{\tilde{\sigma}^n} \right\}|B_{\tilde{\sigma} R}|\, .
\end{aligned}
$$
Therefore, combining the above results, we have 
$$
\begin{aligned}
 |\{ |D{\bf u}(x,\tilde t)|>(1-2^{-m}\sigma)\lambda\}\cap B_{\tilde{\sigma} R}|  \le \left\{\left(\frac{m-1}{m-2}\right)^2 \tilde{\sigma}^2   +   \frac{c_* (m-1)  }{\tilde{\sigma}^n (m-2)^2} \right\}|B_{\tilde{\sigma} R}|.
\end{aligned}
$$
Finally, by choosing $m\in \mathbb N$ sufficiently large so that
$$
\left(\frac{m-1}{m-2}\right)^2 \le \frac{1+\tilde{\sigma}}{2\tilde{\sigma}}  
\quad \text{and}  \quad  \frac{m-1}{(m-2)^2} \le \frac{\tilde{\sigma}^{n+1}(1-\tilde{\sigma})}{2c_*}\,,
$$
we obtain the conclusion.
\end{proof}
We note that $\tilde{\sigma}>\frac{\sigma}{2}$ since $\sigma\in(0,\frac12)$. Now, we are ready for proving Proposition~\ref{prop:degenerate}.

\begin{proof}[Proof of Proposition~\ref{prop:degenerate}]
We remark that constants $c$ in the proof depend also on $\sigma$.

 \textit{Step 1.} Let 
$$
f(t)  := \frac{(t-(1-\nu_1)\lambda)^{\chi-1}_+}{t}\,, \quad \chi\ge 2,
$$
where $\nu_1\in(0,\frac{1}{2})$ is a sufficiently small constant  to be determined, and set $F(t) :=\int_0^tsf(s)\,ds=\frac{1}{\chi} (t-(1-\nu_1)\lambda)_+^{\chi}$. Note that, since
$f'(t)=\frac{(t-(1-\nu_1)\lambda)^{\chi-2} ((\chi-2)t+(1-\nu_1)\lambda) }{t^2}$ if $t>(1-\nu_1)\lambda$,
we have
$$
f'(t)\ge \frac{(t-(1-\nu_1)\lambda)_+^{\chi-2} (1-\nu_1)}{\lambda} \ge \frac{(t-(1-\nu_1)\lambda)_+^{\chi-2} }{2 \lambda}  \quad \text{for }\ 0< t \le\lambda,
$$
and
$$
\frac{f(t)^2}{f'(t)} =\frac{(t-(1-\nu_1)\lambda)_+^{\chi}}{(\chi-2)t+(1-\nu_1)\lambda} \le \frac{(t-(1-\nu_1)\lambda)_+^{\chi}}{(1-\nu_1)\lambda} \le 2 \frac{(t-(1-\nu_1)\lambda)_+^{\chi}}{\lambda} \,.
$$

Let $\xi_0\in C^\infty_0(B_{\tilde{\sigma} R})$ be a cut-off function with $0\le \xi_0\le 1$, $\xi_0\equiv 1$ on $B_{\tilde{\sigma} R/2}$ and $|D\xi_0|\le \frac{4}{\tilde{\sigma} R}$, 
and $\eta_0 \in C^\infty(\R)$ with $0\le \eta_0\le 1$, $\eta_0\equiv 0$ in $(-\infty,-\frac{\sigma R^2}{2\phi''(\lambda)}]$, $\eta_0\equiv 1$ in $[-\frac{\sigma R^2}{4\phi''(\lambda)}, \infty)$  and $0\le (\eta_0)_t \le \frac{8\phi''(\lambda)}{\sigma R^2}$.  
Set $\xi =\xi_0^{\frac{(n+2)\chi-n}{2}}$, $\eta =\eta_0^{(n+2)\chi-n}$. 
Then applying \eqref{Caccio estimate} with $\rho=\tilde \sigma R$, $\tau_1=t_1$ and $\tau_2=t'\in (-\frac{\sigma R^2}{2\phi''(\lambda)},0)$,  and using the fact that $\frac{1}{2}\lambda \le |D{\bf u}| \le \lambda$ when $|D{\bf u}|>(1-\nu_1)\lambda$, we have
$$
\begin{aligned}
& \frac{1}{\chi}\int_{B_{\tilde{\sigma} R}}   w(x,t')^{\chi}  \zeta(x,t')^{(n+2)\chi-n}  \, \d x  +  \phi''(\lambda) \int^{t'}_{t_1} \int_{B_{\tilde{\sigma} R}}  w^{\chi-2}\frac{|D(|D{\bf u}|^2)|^2}{|D{\bf u}|^2} \xi^2\eta \, \d x \d t\\
&\le  c  \int^{t'}_{t_1} \int_{B_{\tilde{\sigma} R}} \left[ \phi''(\lambda) w^{\chi}  |D\xi|^2\eta   + \frac{1}{\chi } w^{\chi} \xi^2\eta_t  \right] \, \d x \d t\\
&\le  c \chi   \int^{t'}_{t_1}  \int_{B_{\tilde{\sigma} R}} \left[ \phi''(\lambda) w^{\chi}  |D\xi_0|^2 \zeta^{(n+2)\chi-n-2}   + w^{\chi} (\eta_0)_t \zeta^{(n+2)\chi-n-1}   \right] \, \d x \d t\\
&\le  c \chi \frac{ \phi''(\lambda)}{R^2}    \int^{t'}_{t_1} \int_{B_{\tilde{\sigma} R}} w^{\chi} ( \zeta^{(n+2)\chi-n-2} +  \zeta^{(n+2)\chi-n-1})\, \d x \d t\,,
\end{aligned}
$$
where we denote
$$
w:=(|D{\bf u}|-(1-\nu_1)\lambda)_+
\quad \text{and} \quad 
\zeta:=\xi_0\eta_0.
$$
Moreover, since $D(w^{\frac{\chi}{2}}\xi) = \frac{\chi}{4} w^{\frac{\chi-2}{2}}\frac{D(|D{\bf u}|^2)}{|D{\bf u}|}\xi+ w^{\frac{\chi}{2}} D\xi$ and $0\le \zeta\le 1$, 
\begin{equation}\label{cacciolog}
\begin{aligned}
\sup_{-\frac{\sigma R^2}{2\phi''(\lambda)}< t< 0} \int_{B_{\tilde{\sigma} R}}   w^{\chi} \zeta^{(n+2)\chi-n}   \, \d x  +&  \phi''(\lambda)  \int^{0}_{-\frac{\sigma R^2}{2\phi''(\lambda)}} \int_{B_{\tilde{\sigma} R}}  |D(w^{\frac{\chi}{2}}\zeta^\frac{(n+2)\chi-n}{2})|^2  \, \d x \d t\\
&\le  c \chi^3 \frac{ \phi''(\lambda)}{R^2}   \int^{0}_{-\frac{\sigma R^2}{2\phi''(\lambda)}}  \int_{B_{\tilde{\sigma} R}} w^{\chi}  \zeta^{(n+2)\chi-n-2} \, \d x \d t \,.
\end{aligned}
\end{equation}
Setting 
$$
\bar w : = w \zeta^{n+2}\,,
$$
by H\"older's inequality and  Sobolev's inequality we have
$$
\begin{aligned}
&\int^{0}_{-\frac{\sigma R^2}{2\phi''(\lambda)}} \int_{B_{\tilde{\sigma} R}}  \bar w^{\chi(1+\frac{2}{n})}\zeta^{-n-2} \, \d x \d t\\
&\le \int^{0}_{-\frac{\sigma R^2}{2\phi''(\lambda)}} \left(\int_{B_{\tilde{\sigma} R}}  \bar w^{\chi}\zeta^{-n} \, \d x\right)^{\frac{2}{n}} \left( \int_{B_{\tilde{\sigma} R}}  \left( \bar w^{\frac \chi 2}\zeta^{\frac{-n}{2}}\right)^{\frac{2n}{n-2}} \, \d x \right)^{\frac{n-2}{n}}\d t \\
&\le c \sup_{-\frac{\sigma R^2}{2\phi''(\lambda)}< t< 0} \left(\int_{B_{\tilde{\sigma} R}}  \bar w^{\chi}\zeta^{-n} \, \d x\right)^{\frac{2}{n}}  \int^{0}_{-\frac{\sigma R^2}{2\phi''(\lambda)}}  \int_{B_{\tilde{\sigma} R}}  |D ( \bar w^{\frac \chi 2}\zeta^{\frac{-n}{2}})|^2 \, \d x \d t \\
& \le c   \phi''(\lambda)^{\frac{2}{n}} \left(  \frac{\chi^3}{R^2}  \int^{0}_{-\frac{\sigma R^2}{2\phi''(\lambda)}}  \int_{B_{\tilde{\sigma} R}} \bar w^{\chi}  \zeta^{-n-2} \, \d x \d t \right)^{1+\frac{2}{n}} .
\end{aligned}
$$
We further set $\tilde Q:= B_{\tilde{\sigma}R} \times (-\frac{\sigma R^2}{2\phi''(\lambda)},0]$. Then we have
$$
\fint_{\tilde Q}  \bar w^{\chi(1+\frac{2}{n})}\zeta^{-n-2} \, \d z \le c   \chi^{3(1+\frac{2}{n})}  \left(  \fint_{\tilde Q} \bar w^{\chi}  \zeta^{-n-2} \, \d z \right)^{1+\frac{2}{n}} , \quad \chi \ge 2\,.
$$
We now apply Moser's iteration. For $m=0,1,2,\dots$, we set
\[
\chi_{m} = 2\theta^m ,
\quad \text{where} \ \ \theta:=1+\frac{2}{n},
\]
and
$$
J_{m}:= \fint_{\tilde Q}  \bar w^{\chi_m}\zeta^{-n-2} \, \d z
$$
Then we have that for $m=2,3,\dots$,
\[\begin{aligned}
J_{m}   \le c \chi_{m-1}^\theta J_{m-1}^{3\theta} \le   c^{1+\theta} \theta^{3\theta} \chi_{m-2}^{3\theta+3\theta^2} J_{m-2}^{\theta^2}  \le \cdots &\le c^{\frac{\theta^m-1}{\theta-1}} \theta^{\frac{3\theta}{\theta-1} \left(\frac{\theta(\theta^{m-1}-1)}{\theta-1} -1\right)} 2^{\frac{\theta(\theta^m-1)}{\theta-1}} J_{0}^{\theta^m} \\
&\le \left( c J_0\right)^{\theta^m} \,,
\end{aligned}\]
which together with $\bar w = w \zeta^{n+2}$ yields 
$$
\| w\|_{L^\infty(\tilde Q)} = \|\bar w\zeta^{-n-2}\|_{L^\infty(\tilde Q)}
=\lim_{m\to \infty}\, J_m^{\frac{1}{\chi_m}} \le c \left( \fint_{\tilde Q}  \bar w^{2}\zeta^{-n-2} \, \d z\right)^{\frac{1}{2}} \le c \left( \fint_{\tilde Q}   w^{2} \, \d z\right)^{\frac{1}{2}}  \,,
$$
and by the definitions of $w$ and $\tilde Q =B_{\tilde \sigma R}  \times (-\frac{\sigma R^2}{2\phi''(\lambda)},0] \supset Q^\lambda_{\sigma R/2}$,
\begin{equation}\label{wsupbound}
\begin{aligned}
\| (|D{\bf u}| - (1-\nu_1) \lambda)_+\|_{L^\infty(Q^{\lambda}_{\sigma R/2})} & \le  c \left( \fint_{\tilde Q}  w^2 \, \d z\right)^{\frac{1}{2}} \\
&  \le c\nu_1 \lambda \left( \frac{|\{(|D{\bf u}| - (1-\nu_1) \lambda)_+>0\}\cap \tilde Q |}{|\tilde Q|}\right)^{\frac12}.
\end{aligned}\end{equation}

\textit{Step 2.} Now, we show that the ratio of measures on the right hand side of \eqref{wsupbound} is sufficiently small if  $\nu_1$  is small. We follow the argument in \cite[Lemma 6.4]{DiBeFried84}.
We start recalling  the  following well-known Poincar\'e-type inequality, see for instance \cite[Lemma 6.3]{DiBeFried84}: for $v\in W^{1,2}(B_r)$ and $k<l$,
\begin{equation}\label{new1}
(l-k)|\{v>l\}\cap B_r|^{1-\frac{1}{n}} \le c \frac{r^n}{|B_r\setminus \{v>k\}|} \int_{\{k<v\le l\}\cap B_r} |Dv| \,\d x\,.
\end{equation}
Let $m_2$ be a positive number determined later, which is larger than the constant $m_1$ determined in Lemma~\ref{Lem:density_timeslice2}, and $j\in\{m_1,m_1+1, \dots, m_2-1\}$. Putting $l=(1-\frac{\sigma}{2^{j+1}})\lambda$,  $k=(1-\frac{\sigma}{2^j})\lambda$, $v=|D\bfu(\cdot,t)|$ and $r=\tilde\sigma R$, and applying Lemma~\ref{Lem:density_timeslice2} we have that for $t\in [-\frac{\sigma R^2}{\phi''(\lambda)}, 0]$,
$$\begin{aligned}
&\frac{\sigma\lambda}{2^{j+1}}|\{|D\bfu(x,t)|>(1-\frac{\sigma}{2^{j+1}})\lambda\}\cap B_{\tilde{\sigma}R}|\\ 
&\le   c\frac{R^{n+1}}{|B_{\tilde{\sigma}R}\setminus \{|D\bfu(x,t)|>(1-\frac{\sigma}{2^{j}})\lambda\}|} \int_{\{k<|D\bfu|(x,t)\le l\}\cap B_{\tilde{\sigma}R}} |D[|D\bfu(x,t)|]| \,\d x\\
&\le c \frac{R}{1-(1-\frac{\sigma}{2})^{1/(n+2)}}  \int_{\{k<|D\bfu(x,t)|\le l\}\cap B_{\tilde{\sigma}R}} |D[|D\bfu(x,t)|]| \,\d x\\
& \le c R  \int_{\{k< |D\bfu(x,t)|\le l\}\cap B_{\tilde{\sigma} R}} |D[|D\bfu(x,t)|]| \,\d x.
\end{aligned}$$
In addtion, by H\"older's inequality,
$$\begin{aligned}
&\int_{\{k<|D\bfu(x,t)|\le l\}\cap B_{\tilde{\sigma}R}} |D[|D\bfu(x,t)|]| \,\d x\\
& \le \left( \int_{\{k<|D\bfu(x,t)|\le l\}\cap B_{\tilde{\sigma}R}} \left|D\left(|D\bfu(x,t)|-\left(1-\frac{\sigma}{2^j}\right)\lambda\right)_+\right|^2 \,\d x \right)^{\frac{1}{2}} |\{k<|D\bfu(x,t)|\le l\}\cap B_{\tilde{\sigma} R}|^{\frac{1}{2}}
\end{aligned}$$
Therefore inserting this  into the above estimate, integrating both the sides for $t$ from $-\frac{\sigma R^2}{2\phi''(\lambda)}$ to $0$ and using H\"older's inequality, we have
 $$\begin{aligned}
 &\frac{\sigma\lambda}{2^{j+1}}|\{|D\bfu|>(1-\frac{\sigma}{2^{j+1}})\lambda\}\cap \tilde Q| \\
& \le  c  R\left( \int_{\{k<|D\bfu|\le l\}\cap \tilde Q} \left|D\left(|D\bfu|-\left(1-\frac{\sigma}{2^j}\right)\lambda\right)_+\right|^2  \,\d z \right)^{\frac{1}{2}} |\{k<|D\bfu|\le l\}\cap \tilde Q|^{\frac{1}{2}}.
 \end{aligned}$$
Moreover, using almost the same analysis as the one employed to derive \eqref{cacciolog}, we also obtain  
$$\begin{aligned}
 \int_{\{k<|D\bfu|\le l\}\cap \tilde Q} \left|D\left(|D\bfu|-\left(1-\frac{\sigma}{2^j}\right)\lambda\right)_+\right|^2 \,\d z  & \le \int_{\tilde Q} \left|D\left(|D\bfu|-\left(1-\frac{\sigma}{2^j}\right)\lambda\right)_+\right|^2 \,\d z \\
&\le    \frac{c}{(1-\tilde\sigma)^2R^2}   \int_{Q^{\lambda}_{R}} \left(|D\bfu|-\left(1-\frac{\sigma}{2^j}\right)\lambda\right)_+^2 \,\d z  \,.
\end{aligned}$$
Combining the above results with the facts that $|D\bfu|\le \lambda$ in $Q^\lambda_{R}$ and $|Q^\lambda_R| \le \frac{c(n)}{\sigma^n}|\tilde Q|$, we have 
$$
\left(\frac{\left|\left\{|D\bfu| >(1-\frac{\sigma}{2^{j+1}})\lambda\right\}\cap \tilde  Q\right|}{|\tilde Q|}\right)^2 \le c \frac{ \left|\left\{(1-\frac{\sigma}{2^{j}})\lambda< |D\bfu| \le (1-\frac{\sigma}{2^{j+1}})\lambda\right\}\cap \tilde Q\right|}{|\tilde Q|}\,.
$$
Then summing over $j$ from $m_1$ to $m_2-1$,  
$$
(m_2-m_1) \left(\frac{|\{|D\bfu|>(1-\frac{\sigma}{2^{m_2}})\lambda\}\cap \tilde Q|}{|\tilde Q|}\right)^2  \le  c\,,
$$
which implies
$$
 \frac{|\{|D\bfu|>(1-v_1)\lambda\}\cap \tilde Q|}{|\tilde Q|}= \frac{|\{|D\bfu|>(1-\frac{\sigma}{2^{m_2}})\lambda\}\cap \tilde Q|}{|\tilde Q|}  \le  \frac{c}{\sqrt{m_2-m_1}}\,,
$$
where we choose
$$
\nu_1 = \frac{\sigma}{2^{m_2}}\,.
$$ 

\textit{Step 3.}
We insert the previous estimate into \eqref{wsupbound} to get
$$
\| (|D{\bf u}| - (1-\nu_1) \lambda)_+\|_{L^\infty(Q^{\lambda}_{\sigma R/2})}  \le  \frac{c_\sigma \nu_1 \lambda}{(m_2-m_1)^{1/4}}
$$
for some $c_\sigma>0$ depending on $n, N, p, q$ and $\sigma$. At this stage, we determine $m_2$ depending on $n, N, p, q$ and $\sigma$ such that $\frac{c_\sigma}{(m_2-m_1)^{1/4}}\le \frac{1}{2}$, hence $\nu_1=\frac{\sigma}{2^{m_2}}$ is also determined.
Then we obtain 
$$
\| (|D{\bf u}| - (1-\nu_1) \lambda)_+\|_{L^\infty(Q^{\lambda}_{\sigma R/2})}\le \frac{\nu_1}{2}\lambda\,,
$$ 
which implies \eqref{degenerate_decay} with $\nu := 1-\frac{\nu_1}{2}$ since $Q^\lambda_{\sigma R/2} \subset \tilde Q$.
%
%
\end{proof}

\subsection{Proof of Theorem~\ref{thm:holdergrad_Lie}} We are now in position to prove the main result, Theorem~\ref{thm:holdergrad_Lie}. Arguing as in \cite[Corollary 1.2]{Lie06}, it will be a consequence of the following claim.

  \textit{Claim. Suppose that $|D{\bf u}| \le \lambda$ in some $Q^\lambda_R=Q^\lambda_R(z_0)\Subset\Omega_T$ and $\lambda>0$. Then, for every $r\in(0,R)$ it holds that
\begin{equation}
\underset{Q^{\lambda}_r(z_0)}{\mathrm{osc}} \, D{\bf u} \le c \left(\frac{r}{R}\right)^{\alpha} \lambda
\label{osc_holder1}\end{equation} 
for some $\alpha\in(0,1)$. }

We assume that $z_0=(0,0)$ for simplicity. Fix $\sigma\in(0,2^{-n-1})$ in Proposition~\ref{prop:nondegenerate}. With this $\sigma$, choose $\nu\in(0,1)$ as in Proposition~\ref{prop:degenerate}. If the assumption \eqref{prop:nondegenerate_ass}  holds, then \eqref{osc_holder1} follows from Proposition~\ref{prop:nondegenerate}. Hence, we assume that \eqref{prop:nondegenerate_ass} does not hold, which means that \eqref{prop:degenerate_ass} holds.

Choose $\theta\in(0,1)$ sufficiently small so that  
$$
\theta \le \frac{\sigma \nu^{\frac{q}{2}-1} }{2} \le \frac{\sigma}{2} 
\quad\text{and}\quad 
\theta \le \left(\frac{\nu}{C}\right)^{\frac{4}{3}}\,,
 $$
where $C\ge 1$ is from  Proposition~\ref{prop:nondegenerate}  and for $m\in\mathbb N_0:= \mathbb N \cup \{0\}$, set $R_m: =\theta^m R$ and $\lambda_m := \nu^m\lambda$. Then 
one  has $Q^{\lambda_{m+1}}_{R_{m+1}}\subset Q^{\lambda_{m}}_{\sigma R_{m}/2}$ for every $j$. This implies  $\{Q^{\lambda_{m+1}}_{R_{m+1}}\}_{m}$ is shrinking. Define
$$
\mathcal N := \{m\in \mathbb N_0: \text{\eqref{prop:nondegenerate_ass} holds with $R_m$ and $\lambda_m$ in place of $R$ and $\lambda$}\}\,,
$$
and
$$
m_0 := 
\begin{cases}
\min\mathcal N & \text{if }\ \mathcal  N \neq \emptyset\\
\infty & \text{if }\ \mathcal N = \emptyset\,.
\end{cases}
$$
Then $m_0\ge 1$. If $1\le m \le m_0$, then by Proposition~\ref{prop:degenerate_ass} with $R_{m-1}$ and $\lambda_{m-1}$ in place of $R$ and $\lambda$, we have $|D{\bf u}|\le \lambda_{m}=\nu^m \lambda$ in $ Q^{\lambda_{m}}_{R_{m}}$ and   
\begin{equation}\label{oscDum1}
\left\{\begin{aligned}
&\underset{Q^{\lambda_{m}}_{R_{m}}}{\mathrm{osc}}\, (D{\bf u}) \le 2 \|D{\bf u}\|_{L^\infty(Q^{\lambda_{m}}_{R_{m}},\R^{Nn})}\le  2\nu^m \lambda\\
&\text{for all} \ \   m=0,1,2,\dots,m_0 \ \ \text{when}\ \ m_0<\infty\,, \ \ \text{or}\ \   m\in \mathbb N_0 \ \ \text{when} \ \ m_0=\infty\,.
\end{aligned}\right.\end{equation}
Furthermore, if $m_0<\infty$, by Proposition~\ref{prop:nondegenerate_ass} with $R_{m_0}$ and $\lambda_{m_0}$ in place of $R$ and $\lambda$, the second condition of $\theta$ in above and \eqref{oscDum1} with $m=m_0$, we have
\begin{equation}\label{oscDum2}
\underset{Q^{\lambda_{m_0}}_{R_{m}}}{\mathrm{osc}}\, (D{\bf u}) \le  C \theta^{(m-m_0)3/4}\underset{Q^{\lambda_{m_0}}_{R_{m_0}}}{\mathrm{osc}}\, (D{\bf u}) \le  2\nu^{m}\lambda \quad \text{for all} \ \ m>m_0 \ \ \text{when}\ \ m_0<\infty\,.
\end{equation}

Fix any $r \in (0,R/2]$. Let $t_1\in (-\frac{r^2}{4\phi''(\lambda)},0]$. Note that $Q^\lambda_{R/2}(0,t_1) \subset Q^\lambda_{R}$ and $\theta^{m+1}R/2\le r < \theta^{m}R/2$ for some $m\in \mathbb N_0$. Then applying \eqref{oscDum1} and \eqref{oscDum2} for $Q^\lambda_{R/2}(0,t_1)$ instead of  $Q^\lambda_R$, we see that 
\begin{equation}\label{oscDuspace}
|D{\bf u}(x,t_1)-D{\bf u}(0,t_1)| \le 2\nu^m \lambda \le c \left(\frac{r}{R}\right)^{\alpha_1} \lambda \,.
\end{equation}
where $\alpha_1 =\log_\theta \nu$, for all $(x,t_1) \in  Q^\lambda_r$. 

Let $\xi_0\in C^\infty_0(B_r)$ with $0\le \xi_0\le 1$, $\eta\equiv 1$ on $B_{r/2}$ and $|D^2\xi|+ |D\xi|^2 \le \frac{c}{r^2}$, $\xi =\|\xi_0\|_{L^1(B_r)}^{-1}\xi_0$, and 
$$
{\bf W}(t) := \int_{B_r} D{\bf u} (x,t) \xi(x)\, \d x \,.
$$
Then by testing \eqref{eq:nondegenerate systemeq} with $\zeta= (\xi_{x_i},\dots,\xi_{x_i})\in C^\infty_0(B_r,\R^N)$, $i=1,2,\dots, n$ and applying the analysis in \eqref{estimateW} along with the inequality $|D^2\xi| \le \frac{c}{r^{n+2}}$,  we have 
$$
|{\bf W}(t)-{\bf W}(0)| \le \frac{c\phi'(\lambda)}{\lambda^{p-1} r^{n+2}} \iint_{Q^\lambda_r}|D{\bf u}(y,s)-{\bf W}(s)|^{p-1}\, \d y \d s\,, \quad t\in \left(-\frac{r^2}{\phi''(\lambda)},0\right]\,.
$$
Moreover, by \eqref{oscDuspace} it follows that
$$
|D{\bf u}(x,t)-{\bf W}(t)| \le \int_{B_r}|D{\bf u}(x,t)-D{\bf u}(y,t)|\xi(y)\, \d y  \le  c \left(\frac{r}{R}\right)^{\alpha_1}\lambda \quad \text{for every } (x,t)\in Q^{\lambda}_r\,.
$$
Therefore, from the preceding two estimates we have that for every  $(x,t)\in {Q^\lambda_r}$
$$\begin{aligned}
|D{\bf u}(x,t)-D{\bf u}(0,0)| &\le |D{\bf u}(x,t)-{\bf W}(t)| + |D{\bf u}(0,0)-{\bf W}(0)| + |{\bf W}(t)-{\bf W}(0)|\\
&\le  c \left(\frac{r}{R}\right)^{\alpha_1(p-1)}\lambda \,.
\end{aligned}$$
This implies \eqref{osc_holder1}.

\section*{Acknowledgments} 
The authors gratefully thank the referees for the thorough and constructive comments.
J. Ok was supported by the National Research Foundation of Korea by the Korean Government (NRF-2022R1C1C1004523). G. Scilla has been supported by the Italian Ministry of Education, University and Research through the MIUR – PRIN project 2017BTM7SN “Variational methods for stationary and evolution problems with singularities and interfaces”. The research of B. Stroffolini was supported by PRIN Project 2017TEXA3H “Gradient flows, Optimal Transport and Metric Measure Structures”.

\bibliographystyle{amsplain}

\begin{thebibliography}{99}



\bibitem{AceFus89} E. Acerbi and N. Fusco,
{Regularity for minimizers of nonquadratic functionals: the case $1<p<2$},
\emph{J. Math. Anal. Appl.} {\bf 140} (1989), no. 1, 115--135. 


\bibitem{BL} P. Baroni and C. Lindfors, The Cauchy-Dirichlet problem for a general class of parabolic equations, \emph{Ann. Inst. H. Poincaré Anal. Non Lin\'eaire} {\bf 34}(3) (2017), 593--624. 



\bibitem{Bers_book} L. Bers,
\emph{athematical aspects of subsonic and transonic gas dynamics},
Surveys in Applied Mathematics, Vol. 3 John Wiley \& Sons, Inc., New York; Chapman \& Hall, Ltd., London 1958



\bibitem{BDLS22} V. B\"ogelein, F. Duzaar, N. Liao and C. Scheven,
{Gradient Hölder regularity for degenerate parabolic systems},
\emph{Nonlinear Anal.} {\bf 225} (2022), Paper No. 113119, 61 pp. 

\bibitem{Chen86} Y. Z. Chen,
{H\"older continuity of the gradient of solutions of nonlinear degenerate parabolic systems}, 
{\emph Acta Math. Sinica (N.S.)} {\bf 2} (1986), no. 4, 309--331. 


\bibitem{CheDiBe89} Y. Z. Chen and E. DiBenedetto, 
{Boundary estimates for solutions of nonlinear degenerate parabolic systems}, 
\emph{J. Reine Angew. Math.} {\bf 395} (1989), 102--131.



\bibitem{Cho18} Y. Cho,
{Calder\'on-Zygmund theory for degenerate parabolic systems involving a generalized $p$-Laplacian type}, 
\emph{J. Evol. Equ.} {\bf 18} (2018), no. 3, 1229--1243. 


\bibitem{Choe91}
H. Choe,
{H\"older regularity for the gradient of solutions of certain singular parabolic systems}, \emph{Comm.
Partial Differential Equations} {\bf 16} (11) (1991), 1709--1732.


\bibitem{Choe92}
H. Choe,
{H\"older continuity for solutions of certain degenerate parabolic systems}, 
\emph{Nonlinear Anal. TMA} {\bf 18} (3) (1992), 235--243.



\bibitem{DiB_book}
E. Di Benedetto, \emph{Degenerate parabolic equations}, Universitext. Springer-Verlag, New York, 1993.


\bibitem{DiBeFried84} E. Di Benedetto  and A. Friedman, \emph{Regularity of solutions of nonlinear degenerate parabolic systems},
\emph{Journal  f\"ur die Reine Angew. Math} {\bf 349} (1984), 83--128.

\bibitem{DiBeFried85} E. Di Benedetto  and A. Friedman, {H\"older estimates for nonlinear degenerate parabolic systems},
\emph{Journal  f\"ur die Reine Angew. Math} {\bf 357} (1985), 1--22, (addendum, ibid. 217--220).



\bibitem{DieEtt08} L. Diening and F. Ettwein,
{Fractional estimates for non-differentiable elliptic systems with general growth},
\emph{Forum Math.} {\bf 20} (3) (2008), 523--556. 

\bibitem{DieKre08} L.~Diening and C.~Kreuzer, Linear convergence of an adaptive finite element method for the $p$-laplacian equation, \emph{SIAM J. Numer. Anal.} \textbf{{\bf 46}} (2008), no.~2, 614--638.


%

\bibitem{DieSchSch19} L. Diening, T. Scharle, and  S. Schwarzacher,
{Regularity for parabolic systems of Uhlenbeck type with Orlicz growth},
\emph{J. Math. Anal. Appl.} {\bf 472} (1) (2019),  46--60. 


\bibitem{DieSchStrVer17} L. Diening, S. Schwarzacher, B. Stroffolini, and A. Verde,
{Parabolic Lipschitz truncation and caloric approximation},
\emph{Calc. Var. Partial Differential Equations}  {\bf 56} (2017), no. 4, Paper No. 120, 27 pp.




\bibitem{DieStrVer09} L. Diening, B. Stroffolini and A. Verde,
{Everywhere regularity of functionals with $\phi$-growth},
\emph{Manuscripta Math.} {\bf 129} (4) (2009), 449--481. 


\bibitem{FinGil57}
R. Finn and D. Gilbarg,
{Three-dimensional subsonic flows, and asymptotic estimates for elliptic partial differential equations}, 
\emph{Acta Math.} 98 (1957), 265--296. 


\bibitem{GiaMod86} M. Giaquinta and G. Modica, {Remarks on the regularity of the minimizers of certain degenerate functionals},
Manuscripta Math. {\bf 57} (1986), no. 1, 55--99. 





\bibitem{Giusti_book} 
E. Giusti,
\emph{Direct methods in the calculus of variations}, World Scientific Publishing Co., Inc., River Edge, NJ, 2003.


\bibitem{Ham92} C. Hamburger, {Regularity of differential forms minimizing degenerate elliptic functionals},
\emph{J. Reine Angew. Math.} {\bf 431} (1992), 7--64.


\bibitem{HH} P. Harjulehto and P. H\"ast\"o, \emph{Orlicz Spaces and Generalized Orlicz Spaces}, Lecture Notes in Mathematics, vol.2236, Springer, Cham, (2019).

\bibitem{HasOk21}  P. Hästö and J. Ok, {Higher integrability for parabolic systems with Orlicz growth}, \emph{J. Differ. Equ.} {\bf 300} (2021), 925--948.




\bibitem{HwaLie15}
S. Hwang and G. M. Lieberman,
{H\"older continuity of bounded weak solutions to generalized parabolic $p$-Laplacian equations I: degenerate case}, 
\emph{Electron. J. Differential Equations}  {\bf 2015}, No. 287, 32 pp. 


\bibitem{HwaLie15-1}
S. Hwang and G. M. Lieberman,
{H\"older continuity of bounded weak solutions to generalized parabolic $p$-Laplacian equations II: singular case},  
\emph{Electron. J. Differential Equations}  {\bf 2015}, No. 288, 24 pp. 


\bibitem{Ise18} T. Isernia,
{$L^\infty$-regularity for a wide class of parabolic systems with general growth},
\emph{Proc. Amer. Math. Soc.} {\bf146} (11) (2018), 4741--4753.





\bibitem{Lie06} G. M. Lieberman, {H\"older regularity for the gradients of solutions of degenerate parabolic systems}, \emph{Ukr. Mat. Visn.} {\bf 3} (3) (2006), 352--373.



\bibitem{Lie_book} G. M. Lieberman, {Second order parabolic differential equations}, World Scientific Publishing Co., Inc., River Edge, NJ, 1996.


\bibitem{Lin} C. Lindfors, 
{Obstacle problem for a class of parabolic equations of generalized $p$-Laplacian type}, 
\emph{J. Differential Equations} {\bf 261} (2016), no. 10, 5499--5540. 

\bibitem{Lions} J. L. Lions, Quelques \'ethodes de r\'esolution des probl\'emes aux limites non lin\'eaires,
Dunod; Gauthier-Villars, Paris, 1969.


\bibitem{Mar89} P. Marcellini,
{Regularity of minimizers of integrals of the calculus of variations with nonstandard growth conditions}, 
\emph{Arch. Rational Mech. Anal.} {\bf 105} (1989), no. 3, 267--284. 

\bibitem{Mar91} P. Marcellini,
{Regularity and existence of solutions of elliptic equations with $p,q$-growth conditions}, 
\emph{J. Differential Equations} {\bf 90} (1991), no. 1, 1--30. 

\bibitem{MarPapi} P. Marcellini and G. Papi,
{Nonlinear elliptic systems with general growth}
\emph{J. Differential Equations} {\bf 221}  (2006), no. 2, 412–443. 

\bibitem{Min_darkside} G. Mingione, 
{Regularity of minima: an invitation to the dark side of the calculus of variations},
\emph{Appl. Math.} {\bf 51} (2006), no. 4, 355-426. 

\bibitem{necas} J. Nec\,as, {Example of an irregular solution to a nonlinear elliptic system with analytic
coefficients and conditions for regularity},  Theor. Nonlin. Oper., Constr. Aspects.
Proc. 4th Int. Summer School. Akademie-Verlag, Berlin, 1975, pp. 197--206.

\bibitem{OhOk22} J. Oh and J. Ok, 
{Gradient estimates for parabolic problems with Orlicz growth and discontinuous coefficients},
\emph{Math. Methods Appl. Sci.}  {\bf 45} (2022), no. 14, 8718--8736


\bibitem{sverakyan}  V. \v Sver\'ak, X. Yan, {A singular minimizer of a smooth strongly convex functional in three
dimensions,} \emph{Calc. Var. Partial Differ. Equ.}, {\bf  10}  (2000), 213--221.


\bibitem{Tol83} P. Tolksdorf, {Everywhere-regularity for some quasilinear systems with a lack of ellipticity.},
\emph{Ann. Mat. Pura Appl.} (4) {\bf 134} (1983), 241--266. 


\bibitem{Uhl77} K. Uhlenbeck,
{Regularity for a class of non-linear elliptic systems}, 
\emph{Acta Math.} {\bf 138} (1977), no. 3-4, 219--240.
 
\bibitem{Wie86} M. Wiegner,
{On $C_\alpha$-regularity of the gradient of solutions of degenerate parabolic systems}, 
\emph{Ann. Mat. Pura Appl.} (4) {\bf 145} (1986), 385--405. 


%
%
%
%
%
%
%
%
%
%
%
%
%
%
%
%
%
%
%
%
%
%
%
%
%
%
%
%
%
%
%
%
%
%
%
%
%
%
%
%
%
%
%
%
%
%
%
%
%
%
%
%
%
%
%
%
%
%
%
%
%
%
%
%
%
%
%
%
%
%
%
%
%
%
%
%
%
%
%
%
\end{thebibliography}

\end{document}